\documentclass[12 pt]{amsart}

\usepackage{multicol}

\usepackage{amsmath,amssymb,amscd,graphicx, enumerate, hyperref,amsthm,fullpage,enumitem}
\usepackage[all]{xy}
\usepackage{tikz}
\usepackage{tikz-3dplot}
\tikzset{>=latex}

\theoremstyle{definition}
\newtheorem{thm}{Theorem}
\numberwithin{thm}{section}

\newtheorem{lemma}[thm]{Lemma}
\newtheorem{defn}[thm]{Definition}     
\newtheorem{prop}[thm]{Proposition}
\newtheorem{construction}[thm]{Construction}
\newtheorem{rem}[thm]{Remark}
\newtheorem{notation}[thm]{Notation}
\newtheorem{ex}[thm]{Example}

\newtheorem{cor}[thm]{Corollary}

\numberwithin{equation}{section}

\newcommand{\cut}{{\sf c}}
\newcommand{\hc}{{\sf hc}}
\newcommand{\ssc}{{\sf \tilde{c}}}
\newcommand{\ctop}{{\sf c_{Top}}}
\newcommand{\sctop}{{\sf \tilde{c}_{Top}}}
\newcommand{\res}{{\sf res}}

\newcommand{\hstb}{{\sf HSTB}}
\newcommand{\stc}{{\sf STC}}
\newcommand{\cstb}{{\sf CSTB}}
\newcommand{\stss}{{\sf STSS}}
\newcommand{\open}{{\sf Open}}
\newcommand{\set}{{\sf Sets}}
\newcommand{\gpoid}{{\sf Groupoids}}
\newcommand{\stm}{{\sf STM}}
\newcommand{\stb}{{\sf STB}}
\newcommand{\BG}{{\sf BG}}
\newcommand{\psts}{{\sf PartSymp}}

\newcommand{\reg}[1]{{#1}_{\sf reg}}
\newcommand{\op}[1]{{#1}^{\sf op}}

\newcommand{\R}{\mathbb{R}}
\newcommand{\C}{\mathbb{C}}
\newcommand{\Z}{\mathbb{Z}}
\newcommand{\T}{\mathbb{T}}
\newcommand{\fg}{\mathfrak{g}}

\newcommand{\fk}{\mathfrak{k}}
\newcommand{\Cc}{\mathcal{C}}
\newcommand{\chor}{{\sf c_{hor}}}
\newcommand{\chern}{{\sf c_1}}
\newcommand{\shor}{{\sf \tilde{c}_{hor}}}

\definecolor{darkgreen}{rgb}{0,0.6,0.2}

\title{Symplectic toric stratified spaces with isolated singularities}

\author{Seth Wolbert}

\begin{document}
\begin{abstract}The goal of this paper is to classify symplectic toric stratified spaces with isolated singularities. This extends a result of Burns, Guillemin, and Lerman which carries out this classification in the compact connected case. In making this classification, it is necessary to classify symplectic toric cones. Via a well-known equivalence between symplectic toric cones and contact toric manifolds, this allows for the classification of contact toric manifolds as well, extending Lerman’s classification of compact connected contact toric manifolds.\end{abstract}
\maketitle
\tableofcontents
%%%%%%%%%%%%%%%%%%%%%%%%%%%%%%%%%%%%%%%%%
\section{Introduction}
%%%%%%%%%%%%%%%%%%%%%%%%%%%%%%%%%%%%%%%%%
In 1988, Delzant classified compact connected symplectic toric manifolds by the images of their moment maps \cite{Delzant}.  This classification was extended by Karshon and Lerman to non-compact symplectic toric manifolds \cite{KarshonLerman}, joining a recent trend of classification of toric symplectic (or, better said, ``symplectic-like'') structures.  Research in this area
has been dominated by two separate pursuits: examining weakened symplectic structures (see origami/folded symplectic manifolds \cite{daSilvaGuilleminPires}/\cite{Hockensmith} and b-symplectic/log symplectic manifolds \cite{GuilleminMirandaPires}/\cite{GualtieriLiPelayoRatiu}) or weakened versions of manifolds (see symplectic toric orbifolds \cite{LermanTolman} or compact symplectic toric stratified spaces \cite{BurnsGuilleminLerman}).  The goal of this paper, the classification of symplectic toric stratified spaces with isolated singularities, follows the latter trend.

The importance of stratified spaces in symplectic geometry arises from the symplectic reduction of Marsden and Weinstein \cite{MarsdenWeinstein}; and Meyer \cite{Meyer}.   In 1991, Sjamaar and Lerman \cite{SjamaarLerman} showed that, in general, symplectic reduction results in a stratified space and furthermore that each stratum inherits a symplectic form from the original manifold. In 2005, Burns, Guillemin, and Lerman \cite{BurnsGuilleminLerman} defined symplectic toric stratified spaces with isolated singularities and classified these in the compact connected case using the images of their moment maps.

The foundations for Delzant's classification are the convexity and connectedness theorems of Atiyah \cite{Atiyah}; and Guillemin and Sternberg \cite{GuilleminSternberg}.  This is emulated by Burns, Guillemin, and Lerman who use a similar convexity and connectedness theorem for compact stratified spaces with isolated singularities.  The issue with the non-compact version of either case is that the image of the moment map no longer needs to be convex and its fibers need not be connected.  

Karshon and Lerman's solution to this problem in the case of a symplectic toric manifold $(M,\omega, \mu)$ is to substitute for the moment map image {\em the orbital moment map}: the unique map $\bar{\mu}$ from the quotient of $M$ to the Lie algebra dual through which $\mu$ factors. This extra information supplements the loss of connected fibers. As the quotient of $M$ by the torus action needn't be contractible, multiple isomorphism classes may be associated to each orbital moment map and these classes are quantified by cohomology classes of the quotient of $M$. Our classification will follow this approach.

Fix $G$ a torus and let $\fg$ denote its Lie algebra.   A {\em symplectic toric stratified space with isolated singularities} $(X,\omega, \mu:X\to \fg^*)$ is (roughly) defined as a symplectic toric manifold with isolated singularities whose deleted neighborhoods are modeled on symplectic toric cones (see Definition \ref{d:symptorstratspace}).   Here, $X$ is the full space, $\omega$ is a symplectic form on $\reg{X}$ (the open, dense manifold that is the top stratum of $X$), and $\mu$ is a continuous function such that $\mu|_{\reg{X}}$ is a moment map for the action of $G$ on $(\reg{X},\omega)$.

By identifying the orbital moment maps of symplectic toric stratified spaces with isolated singularities as a type of map we call {\em stratified unimodular local embeddings}, we show that, in grouping together symplectic toric stratified spaces using these orbital moment map types, we can make the following classification.

\begin{thm}\label{t:maintheorem}
Let $\psi:W\to \fg^*$ be a stratified unimodular local embedding. Then the set of isomorphism classes of symplectic toric stratified spaces with isolated singularities $(X,\omega,\mu:X\to \fg^*)$ with $G$-quotient map $\pi:X\to W$ and orbital moment map $\psi$ is naturally isomorphic to a subspace $\Cc$ of the cohomology classes $H^2(\reg{W},\Z_G\times \R)$.  Here, $\Z_G$ denotes the integral lattice of $G$ (the kernel of $\exp:\fg \to G$) and $\reg{W}$ denotes the top stratum of $W$.  

In particular:
\begin{itemize}
\item If $\dim(G)\neq 3$, $\Cc = H^2(\reg{W},\Z_G\times \R).$
\item If $\dim(G)=3$, $\Cc$ is an extension of $H^2(\reg{W},\Z_G)$ by $H^2(W,\R)$. 
\end{itemize}
\end{thm}

Indeed, we will see that, to each principal $G$-bundle $\pi:P\to \reg{W}$ we may associate a collection of classes of symplectic toric stratified spaces over $\psi$ a collection of classes in $H^2(\reg{W},\R)$ satisfying a certain local property determined by $\chern(P)$, the Chern class of $P$.  In turn, this collection of classes is non-canonical bijection with $H^2(W,\R)$.  This classification extends the classification of compact connected symplectic toric stratified spaces by Burns, Guillemin, and Lerman \cite{BurnsGuilleminLerman} not only by dropping the compact and connectedness conditions but also by dropping several technical conditions (see Remark \ref{r:classificationdistinctions}).

To obtain this classification, we will find it necessary to completely understand symplectic toric cones. Recall that a symplectic toric manifold $(M,\omega, \mu:M\to \fg^*)$ is a {\em symplectic toric cone} if $M$ has a free and proper action of $\R$ commuting with the action of $G$ and, for $\lambda \in \R$ with corresponding action diffeomorphism $\rho_\lambda:M\to M$, we have $\rho_\lambda^*\omega =e^\lambda \omega$.  Additionally, we impose that the moment map $\mu$ for $M$ is the homogeneous moment map for $(M,\omega)$: that which satisfies $\mu(t\cdot p) = e^t\mu(p)$ for every $t\in \R$ and $p\in M$ (such a moment map for $(M,\omega)$ always exists).

As in the case of symplectic toric stratified spaces, the orbital moment maps of symplectic toric cones must take a certain form: that of a {\em homogeneous unimodular local embedding}. Grouping symplectic toric cones by orbital moment map allows us to make the following classification.

\begin{thm}\label{t:stcclassify}
 Let $\psi:W\to \fg^*$ be a homogeneous unimodular local embedding. Then the set of isomorphism classes of symplectic toric cones $(M,\omega,\mu)$ with $G$-quotient $\pi:M\to W$ and orbital moment map $\psi$ is in natural bijective correspondence with the cohomology classes $H^2(W,\Z_G)$, where $\Z_G$ is the integral lattice of $G$, the kernel of the map $\exp:\fg \to G$.
\end{thm}

As symplectic toric cones and contact toric manifolds are intimately related (indeed, they form equivalent categories), the classification of Theorem \ref{t:stcclassify} extends a classification of Lerman \cite{LermanCTM} in the case of compact connected contact toric manifolds.

\begin{cor}\label{c:ctm}
Let $\psi:W\to \fg^*$ be a homogeneous unimodular local embedding. Then the set of isomorphism classes of contact toric manifolds with symplectization $(M,\omega,\mu)$ having $G$-quotient $\pi:M \to W$ and orbital moment map $\psi$ is in natural bijective correspondence with the cohomology classes $H^2(W,\Z_G)$.
\end{cor}

The paper is organized as follows. In Section \ref{s:KL}, we give a brief review of the classification result of Karshon and Lerman \cite{KarshonLerman}. This will serve as a model of the techniques the rest of this paper will use as well as a repository for the results from this classification we will be adapting. The remainder of the paper is split into two main parts: Part \ref{part:stc}, which deals with the classification of symplectic toric cones and Part \ref{part:stss}, which deals with the classification of symplectic toric stratified spaces with isolated singularities. Each part begins with its own introduction and organizational description. Finally, there is a two part appendix, one part dealing with the basics of symplectic cones and contact manifolds and the other dealing with stacks.
\\

\noindent{\bf Notation and Conventions:}  Manifolds are assumed to be finite dimensional, paracompact, and Hausdorff.  $G$ will always denote a torus (a compact connected commutative finite dimensional Lie group) and $\fg$ will always denote its Lie algebra. $\Z_G$ will always be used to denote the integral lattice of $\fg$; that is, the lattice $\ker(\exp : \fg \to G)$.  The notation $\langle \cdot,\cdot \rangle$ will denote the canonical pairing $\fg^*\times \fg \to \R$.

In using manifolds with corners, we will follow Karshon and Lerman and use the convention of Joyce (see \cite{Joyce}). An {\em $n$-dimensional manifold with corners} is a paracompact Hausdorff topological spaces with a maximal atlas of charts to {\em sectors of $\R^n$} (open subsets of subsets of $\R^n$ of the form $[0,\infty)^k\times \R^{n-k}$).  The {\em index of a point $x$} in a manifold with corners is the integer $k$ for which there exists a chart $\varphi$ to a sector  $[0,\infty)^k\times \R^{n-k}$ with $\varphi(x)=0$ (note: index is called {\em depth} in \cite{Joyce}).  The set of all points $x$ in a manifold with corners $X$ with index $k$ is denoted $S^k(X)$ and each $S^k(X)$ is naturally a manifold without boundary. Smooth maps of manifolds with corners are defined in the same manner as smooth maps of manifolds are defined: continuous maps that locally factor through smooth maps on charts. In particular, this means that we are expressly {\em not} thinking of maps of manifolds with corners as maps of stratified spaces. Indeed, our definition of smooth maps of manifolds with corners allows for the image of any stratum of the source manifold with corners to be contained in the union of multiple strata of the target.

As explained in Appendix A of \cite{KarshonLerman}, de Rham cohomology is well-defined for manifolds with corners, is invariant under smooth homotopy and thus, in particular, satisfies the Poincar\'{e} Lemma.  For a manifold with corners $W$, $H^k(W,\R)$ will denote the degree $k$ de Rham cohomology. In Part \ref{part:stss}, it will be necessary to use the singular cohomology of a stratified space $X$ with coefficients in $\R$; this will also be denoted by $H^k(X,\R)$.  Since partitions of unity exist for manifolds with corners (one may restrict restrict bump functions for $\R^n$ to bump functions on sectors), one may conclude via the line of argument in Chapter 5, Section 9 of \cite{BredondR} that any manifold with corners $W$ admits a de Rham isomorphism relating smooth de Rham forms and singular forms with coefficients in $\R$.  

For a manifold with corners $W$, we denote by $\mathring{W}$ the open dense interior of $W$ (i.e., the index $0$ elements of $W$). There always exists a manifold without corners $\bar{W}$ into which a manifold with corners $W$ embeds; in this case, it is said that $\bar{W}$ {\em contains $W$ as a domain}. Given two maps $f : M \to N$ and $g : M'\to N$, the symbol $M\times_N M'$ will denote the fiber product of $M$ with $M'$ over $N$. In the case where we wish to emphasize the maps $f$ and $g$,
we may write $M\times_{f,N,g}M'$. 

For any topological space $X$, $\open(X)$ will always denote the category of open subsets of $X$ with morphisms inclusions of subsets. The symbols $\set$ and $\gpoid$ will denote the
categories of sets and (small) groupoids. By {\em presheaf of groupoids}, we will mean a {\em strict} presheaf of groupoids with domain a (full subcategory of) the category of open subsets on
some topological space; in other words, a (1-)functor $\mathcal{F} : \op{\open(X)} \to \gpoid$. To avoid unnecessary generality involving sites and categories fibered in groupoids, we will take {\em stack} to mean such a presheaf of groupoids satisfying the standard descent conditions (see Definition \ref{d:stack}).
\\

\noindent {\bf Glossary of presheaves:} Here is a comprehensive list of the various presheaves appearing in this paper with brief descriptions and references to their definitions:
\begin{itemize}
\item $\stm_\psi$ denotes {\em the stack of symplectic toric manifolds over a unimodular local embedding $\psi$} (see Definition \ref{d:stm} and Remark \ref{r:stmastack}).
\item $\stb_\psi$ denotes {\em the stack of symplectic toric bundles over a unimodular local embedding $\psi$} (see Definition \ref{d:stb}  and Remark \ref{r:stmastack}).
\item $\stc_\psi$ denotes {\em the stack of symplectic toric cones over a homogeneous unimodular local embedding $\psi$} (see Definition \ref{d:stc} and Proposition \ref{p:stcapsheaf}).
\item $\hstb_\psi$ denotes {\em the stack of homogeneous symplectic toric bundles over a homogeneous unimodular local embedding $\psi$} (see Definition \ref{d:hstb}; and Propositions \ref{p:hstbapresheaf} and \ref{p:hstbastack}).
\item $\BG$ denotes {\em the stack of principal $G$-bundles} over a fixed topological space (see \ref{ex:bg}). 
\item $\BG_{\R}$ denotes {\em the stack of homogeneous principal $G$-bundles over $W$} for $W$ a topological space with a free and proper $\R$ action (see Definition \ref{d:bgr}).
\item $\stss_\psi$ denotes {\em the stack of symplectic toric stratified spaces with isolated singularities over a stratified unimodular local embedding $\psi$} (see Definition \ref{d:stss} and Remark \ref{r:stssapsheaf}).
\item $\cstb_\psi$  denotes {\em the stack of conical symplectic toric bundles over a stratified unimodular local embedding $\psi$} (see Definition \ref{d:cstb}, Remark \ref{r:cstbapsheaf}, and Proposition \ref{p:cstbastack}). 
\item $\psts_\psi$ denotes {\em the presheaf of partitioned symplectic toric spaces over $\psi$ over a stratified unimodular local embedding $\psi$}; see Definition \ref{d:psts} and Remark \ref{r:pstspullbacks}).\\
\end{itemize}

\noindent{\bf Acknowledgments:} The author would like to thank Eugene Lerman for introducing the problem this paper addresses as well as for his constant support during the writing process.
The author would also like to thank the referee for their careful reading, detailed notes, and suggestions on multiple drafts of this paper which helped to greatly improved this work from its original version. This work was supported in part by a gift to the Mathematics Department at the University of Illinois from Gene H. Golub.

%%%%%%%%%%%%%%%%%%%%%%%%%%%%%%%%%%%%%%%%%
\section{Symplectic toric manifolds}\label{s:KL}
%%%%%%%%%%%%%%%%%%%%%%%%%%%%%%%%%%%%%%%%%

What follows is a review of the recently published classification of non-compact symplectic toric manifolds by Karshon and Lerman \cite{KarshonLerman}. It is by no means a complete account; the aim is to give a rough outline of their classification. This section also serves as a convenient repository of relevant ideas we will later be citing and adapting. Those familiar with the result of Karshon and Lerman may safely skip this section.

Fix a torus $G$ with Lie algebra $\fg$. Recall that a {\em symplectic toric manifold} $(M,\omega,\mu:M\to \fg^*)$ is a symplectic manifold $(M,\omega)$ with an effective Hamiltonian action of $G$ and a moment map $\mu:M\to \fg^*$ satisfying $2\dim(G) = \dim(M)$. As discussed in the introduction, Karshon and Lerman replace Delzant polytopes (used to classify compact connected symplectic toric manifolds via moment map images) with orbital moment maps:

\begin{defn}\label{d:gquot}
Given a symplectic toric manifold $(M,\omega,\mu:M\to \fg^*)$ and $W$ a manifold with corners, a {\sf $G$-quotient map} $\pi:M\to W$ is a smooth map such that, for any $G$-invariant map $f : M \to N$, there exists a unique smooth map $\bar{f}:W\to N$ such that $f=\bar{f}\circ \pi$.

Given a symplectic toric manifold $(M,\omega,\mu:M\to \fg^*)$ and a $G$-quotient map $\pi:M\to W$, the unique map $\bar{\mu}:W\to \fg^*$ for which $\mu=\bar{\mu}\circ \pi$ is called {\sf the orbital moment map}. 
\end{defn}

It is of course not obvious that such a $G$-quotient map (i.e., to a manifold with corners) exists for any symplectic toric manifold.  However, one may show that the canonical choice $M/G$ for a $G$-quotient of a symplectic toric manifold $(M,\omega,\mu:M\to \fg^*)$ is, in fact, a manifold with corners:

\begin{prop}
For any symplectic toric manifold $(M,\omega,\mu:M\to \fg^*)$, the $G$-quotient $\pi:M\to M/G$ is a $G$-quotient in the sense of Definition \ref{d:gquot}.  In particular, $M/G$ is a manifold with corners and the corresponding orbital moment map $\bar{\mu}:M/G\to \fg^*$ is smooth.
\end{prop}

So every symplectic toric manifold comes with at least one $G$-quotient; however, it will be convenient (and indeed vital to the existence portion of the classification) to divorce an orbital moment map from the context of the originating manifold $M$ and to work with unimodular local embeddings defined an arbitrary manifolds with corners. These are defined as follows:

\begin{defn}\label{d:ule}
A {\sf unimodular cone with vertex $\epsilon \in \fg^*$} is a subset of the form
\[C_{\{v_1,\ldots,v_k\},\epsilon}:=\{\eta\in \fg^* \,|\,\langle \eta-\epsilon,v_i\rangle\geq 0\text{ for all }1\leq i \leq k\}\]
where $\{v_1,\ldots,v_k\}$ is a basis for an integral lattice of a subtorus of $G$.

For a manifold with corners $W$, a smooth map $\psi:W\to \fg^*$ is a {\sf unimodular local embedding} if, for each $w\in W$, there exists neighborhood $U$ of $w$ in $W$ and $\{v_1,\ldots,v_k\}$ the basis to the integral lattice of a subtorus of $G$ so that $\psi|_U$ is an open embedding of $U$ onto a neighborhood of the vertex of the unimodular cone $C_{\{v_1,\ldots,v_k\},\psi(w)}$ with vertex $\psi(w)$.
\end{defn}

\begin{prop}\label{p:KLorbmmaule}
Given a symplectic toric manifold $(M,\omega,\mu:M\to \fg^*)$ and $G$-quotient $\pi:M\to W$, the orbital moment map $\bar{\mu}:M\to \fg^*$ is a unimodular local embedding.
\end{prop}

Given two symplectic toric manifolds $(M,\omega,\mu)$ and $(M',\omega',\mu')$, if there is a $G$-equivariant symplectomorphism $\varphi:(M,\omega)\to (M',\omega')$ with $\mu=\mu'\circ \varphi$, note that we may choose a common $G$-quotient space $W$ with quotient maps $\pi:M\to W$ and $\pi':M'\to W$ satisfying $\pi=\pi'\circ\varphi$. It follows that $\mu$ and $\mu'$ descend to the same orbital moment map $\bar{\mu}:W\to \fg^*$.  Thus, to understand the collection of all symplectic toric manifolds, it makes sense to group symplectic toric manifolds together by quotient space and orbital moment map.

\begin{defn}\label{d:stm}
Let $\psi:W\to \fg^*$ be a unimodular local embedding. Then a {\sf symplectic toric manifold over $\psi$} is a symplectic toric manifold $(M,\omega,\mu)$ together with a $G$-quotient map $\pi:M\to W$ such that $\mu= \psi \circ \pi$. This data will be expressed as the triple $(M,\omega,\pi:M\to W)$.

{\sf The groupoid of symplectic toric manifolds over $\psi$}, denoted $\stm_\psi(W)$, is the groupoid with
\begin{itemize}
	\item objects: symplectic toric manifolds over $\psi$; and
	\item morphisms: $G$-equivariant symplectomorphisms
	\[f:(M,\omega,\pi:M\to W) \to (M',\omega',\pi':M'\to W)\]
	satisfying $\pi'\circ f = \pi$.
\end{itemize}
\end{defn}

The strategy for actually classifying these spaces is to relate them to a simpler class of objects, namely symplectic toric bundles.

\begin{defn}\label{d:stb}
 Let $\psi:W\to \fg^*$ be a unimodular local embedding. Then a {\sf symplectic toric principal $G$-bundle over $\psi$} is a pair $(\pi:P\to W,\omega)$, for $\pi:P\to W$ a principal $G$-bundle and $\omega$ a $G$-invariant symplectic form with moment map $\psi\circ \pi$. 

 {\sf The groupoid of symplectic toric principal $G$-bundles over $\psi$}, denoted $\stb_\psi(W)$, is the groupoid with
 \begin{itemize}
 	\item objects: symplectic toric principal $G$-bundles over $\psi$; and
 	\item morphisms: $G$-equivariant symplectomorphisms
 	\[\varphi: (\pi:P\to W,\omega) \to (\pi':P'\to W,\omega')\]
 	for which $\pi'\circ\varphi=\pi$.
 \end{itemize}
\end{defn}

\begin{rem}\label{r:stmastack}
For any open subset $U$ of $W$, $\psi|_U:U\to \fg^*$ is also a unimodular local embedding; thus we may define
\[\stm_\psi(U) := \stm_{\psi|_U}(U)\;\;\;\text{ and }\;\;\; \stb_\psi(U) := \stb_{\psi|_U}(U)\]
These collections of groupoids define presheaves of groupoids
\[\stm_\psi:\op{\open(W)} \to \gpoid\;\;\;\text{ and }\;\;\;\stb_\psi:\op{\open(W)}\to \gpoid\]
In fact, these presheaves of groupoids are stacks over $\op{\open(W)}$.
\end{rem}

To establish the equivalence of the groupoids $\stb_\psi(W)$ and $\stm_\psi(W)$, Karshon and Lerman introduce the functor $\cut:\stb_\psi(W)\to\stm_\psi(W)$, constructed with the following
steps:

\begin{enumerate}
	\item For every $w\in W$, $\psi$ determines a basis $\left\{v_1^{(w)},\ldots,v_k^{(w)}\right\}$ of the integral lattice of a subtorus $K_w$ of $G$.  In turn, this basis defines a symplectic representation $\rho:K_w\to \mathrm{Sp}\left(\C^k,\omega_{\C^k}\right)$ for $(V_w,\omega_w):=\left(\C^k,\omega_{\C^k}\right)$ the symplectic vector space with $\omega_{\C^k}= \frac{\sqrt{-1}}{2\pi}\sum dz_j\wedge d\bar{z}_j$ and $\rho$ the toric representation with symplectic weights $\{v_1^*,\ldots,v_k^*\}$, the dual basis to $\left\{v_1^{(w)},\ldots,v_k^{(w)}\right\}$; i.e., we take
	\begin{equation}\label{eq:symprep} \exp(X)\cdot (z_1,\ldots,z_k):= (e^{2\pi\sqrt{-1}\langle v_1^*,X\rangle}\cdot z_1,\ldots,e^{2\pi\sqrt{-1}\langle v_k^*,X\rangle}\cdot z_k)\end{equation}
	\item  For any principal bundle $\pi:P\to W$, let $\sim$ be the equivalence relation
	\[p\sim p'\text{ if and only if there exists }k\in K_{\pi(p)}\text{ with }p\cdot k=p'\]
	on $P$. Then define
	\[ \ctop(\pi:P\to W,\omega):=(P/\sim, \bar{\pi}:P/\sim\,\to W)\]
	for $\bar{\pi}$ the map descending from $\pi$. It follows from the $G$-equivariance of morphisms of $\stb_\psi(W)$ that these morphisms descend to continuous equivariant morphisms between these quotients modulo $\sim$.  It follows that $\ctop$ defines a functor
	\[\ctop:\stb_\psi(W)\to {\sf Top}_G(W)\]
	for ${\sf Top}_G(W)$ {\em the category of topological $G$-spaces over $W$}: spaces with $G$-quotients to $W$ and continuous equivariant morphisms intertwining quotient maps.

	This functor is natural with respect to restriction; that is, for every open $U$ in $W$,
	\[\ctop((P,\omega)|_U)=\ctop(P,\omega)|_U:=(\bar{\pi}^{-1}(U),\bar{\pi})\]
	\item  To ``symplectize'' these topological quotients, Karshon and Lerman use symplectic cuts, showing for each $w\in W$, there is a neighborhood $U_w$ of $w$ in $W$ (defined using the properties of $\psi$) so that
	\[{\sf cut}((P,\omega)|_{U_w}):=(P|_{U_w}\times V_w)//_0\,K_w\]
	is a symplectic toric manifold over $\psi|_{U_w}$. This establishes a functor
	\[{\sf cut}:\stb_\psi(U_w)\to \stm_\psi(U_w)\]
	for each $w\in W$.
	\item For each $w\in W$ and $(P,\omega)\in \stb_\psi(W)$, there is a natural homeomorphism
	\[\alpha^P_w:\ctop\left((P,\omega)|_{U_w}\right)\to {\sf cut}\left((P,\omega)|_{U_w}\right)\]
	intertwining the $G$-quotient maps to $W$ of the source and target. For any $w$, $w'$ in $W$ with $U_w\cap U_{w'}$ non-empty, $\alpha^P_{w'}\circ \left(\alpha_w^P\right)^{-1}$ is a symplectomorphism.
	Therefore, $\ctop(P,\omega)$ inherits the structure of a symplectic toric manifold.
	\item  Finally, for each morphism $\varphi:(P,\omega)\to (P,\omega')$ of $\stb_\psi(W)$ and for any $w\in W$, the diagram
	\[\xymatrix{\ctop(P,\omega)|_{U_w}\ar[r]^{\alpha_w^P}\ar[d]_{\ctop(\varphi)} & {\sf cut}((P,\omega)|_{U_w})\ar[d]^{{\sf cut}(\varphi|_{U_w})} \\ \ctop(P',\omega')|_{U_w} \ar[r]_{\alpha^{P'}_w} & {\sf cut}((P',\omega')|_{U_w}) }\]
	commutes.  Thus, $\ctop(\varphi)$ is a symplectomorphism with respect to the symplectic structure induced by the morphisms $\alpha^P_w$.
\end{enumerate}

As it will be important later, we present below an outline of the process used to ``symplectize'' the quotient space $\ctop(P,\omega)$.  First, an important theorem about extending Marsden-
Weinstein and Meyer reduction to a specific scenario involving manifolds with corners is required.

\begin{thm}[Theorem 2.23, \cite{KarshonLerman}]\label{t:reduction}
Suppose $(M,\sigma)$ is a symplectic manifold with corners with a proper Hamiltonian action of a Lie group $K$ and moment map $\Phi: M \to \fk^*$ (for $\fk$ the Lie algebra of $K$). Suppose also that
\begin{itemize}
	\item for each $x\in \Phi^{-1}(0)$, the stabilizer $K_x$ of $x$ is trivial;
	\item $\Phi$ admits an extension $\tilde{\Phi}$ to a manifold without corners $\tilde{M}$ containing $M$ as a domain; and
	\item $\tilde{\Phi}^{-1}(0)=\Phi^{-1}(0)$.
\end{itemize}
Then $\Phi^{-1}(0)$ is a manifold without corners and the symplectic reduction at $0$
\[M//_0\,K:=\Phi^{-1}(0)/K\]
is a symplectic manifold via standard Marsden-Weinstein/Meyer symplectic reduction.
\end{thm}

\begin{rem}
Note that $\tilde{M}$ and $\tilde{\Phi}$ in the statement of Theorem \ref{t:reduction} serve only to prove that $\Phi^{-1}(0)$ is a manifold without corners.  Once this has been accomplished, the usual Marsden-Weinstein and Meyer reduction procedure is then just applied to $\Phi^{-1}(0)$ and $\tilde{M}$ and $\tilde{\Phi}$ have no direct impact on the resulting reduction.
\end{rem}

We now construct ${\sf cut}((P,\omega)|_{U_w})$ for a valid choice of $U_w$.

\begin{construction}\label{const:cuts}
 Fix a symplectic toric bundle $(\pi:P\to W,\omega)$ over unimodular local embedding $\psi:W\to \fg^*$.  	

 As stated in step (1) above, because $\psi$ is a unimodular local embedding, there exists a unimodular cone $C_w:=C_{\{v_1,\ldots,v_k\},\psi(w)}$ with vertex $\psi(w)$ so that $\psi$ embeds a neighborhood of $w$ in $W$ onto a neighborhood of the vertex $\psi(w)$ in the cone. Here, $\{v_1,\ldots, v_k\}$ is the basis for the Lie algebra $\fk$ of a subtorus $K_w\leq G$ which, in turn, defines a symplectic toric $K_w$ representation $\rho_w:K_w\to \mathrm{Sp}\left(\C^k,\omega_{\C^k}\right)$ with $\omega_{\C^k}=\frac{\sqrt{-1}}{2\pi}\sum z_j\wedge \bar{z}_j$ and with symplectic weights $\{v_1^*,\ldots,v_k^*\}$ (for $\{v_1^*,\ldots,v_k^*\}$ the dual basis to $\{v_1,\ldots,v_k\}$) yielding the action given in equation \eqref{eq:symprep}.  This representation has moment map
 \[\mu_w:\C^k\to \fk^*,\;\;(z_1,\ldots,z_k)\mapsto -\sum_{j=1}^k |z_j|^2v_j^*.\]

Let $\iota:\fk\to \fg$ be the embedding of $\fk$ into $\fg$ and let $\iota^*$ be the dual to this embedding. Then, since $\psi\circ \pi$ is the moment map for the free action of $G$ on $P$, $\nu:=\iota^*\circ \psi\circ \pi:P\to \fk^*$  is the moment map for the action of $K_w$ on $P$. Define $\xi_0:=\iota^*(\psi(w))$. Then for $C'_w:= C_{\{v_1,\ldots,v_k\},\xi_0}$ the unimodular cone with vertex $\xi_0$ in $\fk^*$ and $\fk^o$ the annihilator of $\fk$ in $\fg^*$, we can identify $C_w$ with the product $C_w'\times \fk^o$. Explicitly, we are splitting $C_w$ into the product of $C_w'$, a cone in $\fk^*$ containing no affine subspaces, and the affine space $\psi(w)+\fk^o$.  This affine space corresponds (at least near $w$) to the image of the open face of $W$ containing $w$.

Thus, there exist contractible neighborhoods $\mathcal{U}$ of $w$ in the open face of $W$ containing $w$ and $\mathcal{V}$ of $\xi_0$ in $\fk^*$ so that, for $\mathcal{V}':=C_w'\cap \mathcal{V}$, there exists $U_w$ a neighborhood of $w$ in $W$ which is diffeomorphic to $\mathcal{U}\times \mathcal{V}'$.  Then $\nu|_{U_w} : P|_{U_w} \to \mathcal{V}'$ is a trivializable $\mathcal{U}\times G$ fiber bundle. Therefore, $P|_{U_w}$ is contained in a manifold $\tilde{P}$ (diffeomorphic to $\mathcal{V}\times \mathcal{U}\times G$) as a domain and $\nu|_{U_w}:P|_{U_w} \to \mathcal{V}'$ admits a smooth extension to a map $\tilde{\nu}:\tilde{P}\to \mathcal{V}$.

Define $\Phi:P|_{U_w}\times\C^k\to \fk^*$ by
\[\Phi(p,z):=\nu(p)-\xi_0+\mu_w(z)\]
Then $\Phi$ is a moment map for the action of $K_w$ on $P|_{U_w}\times \C^k$ and admits an extension to the map
\[\tilde{\Phi}(p,z):=\tilde{\nu}(p)-\xi_0+\mu_w(z)\]
satisfying the conditions of Theorem \ref{t:reduction}.  Thus, reduction at the $0$ level set of $\Phi$ yields a symplectic manifold (without corners). One may check that $(P|_{U_w}\times \C^k)//_0\,K_w$ inherits a $G$-quotient map $\bar{\pi}$ to $U_w$ with respect to which $((P|_{U_w}\times \C^k)//_0\,K_w,\bar{\pi})$ is a symplectic toric manifold of $\stm_\psi(U_w)$.  Define ${\sf cut}((P,\omega)|_{U_w}):=((P|_{U_w}\times \C^k)//_0\,K_w,\bar{\pi})$.

For $\varphi:(P,\omega)\to (P',\omega')$, the morphism $\varphi\times\mathrm{id}_{\C^k}:P|_{U_w}\times\C^k\to P'|_{U_w}\times \C^k$ descends to a symplectomorphism ${\sf cut}(\varphi): {\sf cut}((P,\omega)|_{U_w})\to {\sf cut}((P',\omega')|_{U_w})$.
\end{construction}

\begin{rem}\label{r:shiftedmomentmap}
Given a unimodular local embedding $\psi:W\to \fg^*$ and a Lie algebra dual element $\eta\in \fg^*$, let $\psi'$ be the map $\psi'(w):=\psi(w)+\eta$.  Then $\psi'$ is also a unimodular local
embedding. Let $(\pi:P\to W,\omega)$ be a symplectic toric bundle over $\psi$; then it is a symplectic toric bundle over $\psi'$ as well.

For $C_w$ and $C'_w$ the unimodular cones with vertices $\psi(w)$ and $\psi'(w)$ respectively onto which $\psi$ and $\psi'$ are local embeddings near $w$, note we have $C_w+\eta = C'_w$ and therefore $\psi$ and $\psi'$ determine the same subtorus $K_w\leq G$. It follows that $\ctop(P,\omega)$ is the same topological $G$-space when regarding $(P,\omega)$ as a symplectic toric bundle over $\psi$ or $\psi'$.  Following Construction \ref{const:cuts}, since the cutting procedures with respect to $\psi$ or $\psi'$ are performed relative to the cone vertices $\psi(w)$ and $\psi'(w)$, $\ctop(P,\omega)$ is furthermore symplectized the same way with respect to either unimodular local embedding.

Therefore, the symplectic toric manifolds $\cut(P,\omega)\in \stm_\psi(W)$ and $\cut(P,\omega)\in\stm_{\psi'}(W)$ are symplectomorphic. Of course, this symplectomorphism does not intertwine the respective moment maps.
\end{rem}

For the purposes of this paper, it will also be important to sketch the construction of the homeomorphisms $\alpha^P_w:\ctop(P,\omega)|_{U_w}\to {\sf cut}((P,\omega)|_{U_w})$.

\begin{construction}\label{const:transitions}
For each $w\in W$ and $U_w$ defined as in Construction \ref{const:cuts}, to define the homeomorphisms $\alpha^P_w:\ctop(P,\omega)|_{U_w}\to {\sf cut}((P,\omega)|_{U_w})$, first let $s:\mu_w(\C^k)\to \C^k$ be the continuous section of $\mu_w$ defined by
\[s(\eta):=\left(\sqrt{\langle -\eta,v_1\rangle},\ldots , \sqrt{\langle -\eta,v_k\rangle}\right)\]
Then one can show that the map $\alpha^P_w:\ctop(P,\omega)|_{U_w}\to (P|_{U_w}\times\C^k)//_0\,K_w$ defined by
\[[p]\mapsto [p,s(\iota^*(\psi(p))-\nu(p))]\]
is a well-defined $G$-equivariant homeomorphism. 
\end{construction}

\begin{rem}\label{r:coninterior}
For $w\in \mathring{W}$ (the interior of W), we have that $\psi|_{U_w}$ is an open embedding into $\fg^*$ itself (i.e., rather than just an embedding into a cone). This means that $K_w$ is trivial and therefore ${\sf cut}((P,\omega)|_{U_w})$ is just $(P|_{U_w},\omega,\pi)$ thought of as a symplectic toric manifold over $\psi|_{U_w}$. 
\end{rem}

\begin{rem}
Since the collection of functors $\cut:\stb_\psi(U)\to \stm_\psi(U)$ for each open $U$ in $W$ commute with restriction, it follows that we have a map of presheaves of groupoids over $\open(W)$
\[\cut:\stb_\psi\to \stm_\psi\]
\end{rem}

In service of classifying the groupoid of symplectic toric bundles over a given unimodular local embedding $\psi:W\to \fg^*$, Karshon and Lerman prove the following lemmas.

\begin{lemma}[Lemma 3.2, \cite{KarshonLerman}]\label{l:KLbijectionlemma}
Let $\psi:W\to \fg^*$ be a unimodular local embedding, let $\pi:P\to W$ be a principal $G$-bundle, and let $A \in \Omega^1(P,\fg)^G$ be a connection 1-form for $P$.  For convenience, define $\mu:=\psi\circ \pi$. Then:
\begin{itemize}
	\item Any closed $G$-invariant 2-form on $P$ with moment map $\mu$ is automatically symplectic; this includes the form $d\langle \mu, A\rangle$.
	\item The map from closed 2-forms on $W$ to closed 2-forms on $P$
	\[\beta\mapsto d\langle \mu,A\rangle +\pi^*\beta\] 
	is a bijection between the set of closed 2-forms on $W$ and the set of $G$-invariant symplectic forms on $P$ with moment map $\mu$.
\end{itemize}
\end{lemma}

\begin{lemma}[Lemma 3.3, \cite{KarshonLerman}]\label{l:KLexactgaugetranslemma}
Let $\psi:W\to \fg^*$ be a unimodular local embedding and let $\pi:P\to W$ be a principal $G$-bundle. For any 1-form $\gamma$ on $W$ and any $G$-invariant symplectic form $\omega$ on $P$ with moment map $\mu$, there exists a bundle isomorphism $f : P \to P$ with $f^*(\omega+\pi^*(d\gamma))=\omega$.
\end{lemma}

Karshon and Lerman also prove that for every open subset $U$ of $W$, $\cut_U:\stb_\psi(U)\to \stm_\psi(U)$ is a fully faithful functor. Observing that for contractible open subsets $V$ of $W$, the groupoid $\stm_\psi(V)$ is connected (i.e., all objects are isomorphic), they also conclude that $\cut$ must be locally essentially surjective. Implicitly using the fact that $\stb_\psi$ is a stack and $\stm_\psi$ is a prestack (see Definition \ref{d:prestack}), they are able to conclude the following theorem.

\begin{thm}[Theorem 4.1, \cite{KarshonLerman}]\label{t:cutequiv} 
Let $\psi:W\to \fg^*$ be a unimodular local embedding. Then
\[\cut:\stb_\psi(W)\to \stm_\psi(W)\]
is an equivalence of categories.
\end{thm}

Using the tools of Lemmas \ref{l:KLbijectionlemma} and \ref{l:KLexactgaugetranslemma}, Karshon and Lerman are able to show that the elements of $\stb_\psi(W)$ are classified by the cohomology classes $H^2(W,\Z_G)\times H^2(W,\R)$ (where $\Z_G := \ker(\exp : \fg \to G)$ is the integral lattice of $\fg$).  Explicitly, they use two characteristic classes: the first Chern class and their  horizontal class.

\begin{defn}\label{d:chor}
For $\psi:W\to \fg^*$ a unimodular local embedding and $(\pi:P\to W,\omega)$ a symplectic toric bundle over $\psi$, if, for a choice of connection $A$ on $P$, $\omega = d\langle \mu,A\rangle + \pi^*\beta$ for $\beta \in \Omega^2(W,\R)$, then {\sf the horizontal class} $\chor([P,\omega]):=[\beta]$ (for $[P,\omega]$ the isomorphism class associated to $(P,\omega)$ in $\stb_\psi(W)$).
\end{defn}

\begin{rem}
As demonstrated by Karshon and Lerman (see Proposition 5.1 \cite{KarshonLerman}), while one must choose a connection on $P$ to define this horizontal class, the class itself is independent not
only of the choice of connection but also the choice of representative of the isomorphism class $[\pi:P\to W,\omega]$ of $\stb_\psi(W)$.  Furthermore, this choice of class is natural with respect to restriction and hence yields a well-defined characteristic class.
\end{rem}

Thus, using the equivalence of categories c, they conclude the following result.

\begin{thm}[Theorem 1.3, \cite{KarshonLerman}]
Let $\psi:W\to \fg^*$ be a unimodular local embedding.  Then:
\begin{enumerate}
	\item The groupoid $\stm_\psi(W)$ is non-empty; that is, there exists a symplectic toric manifold $(M,\omega,\mu)$ with $G$-quotient $\pi:M\to W$ with respect to which $\psi$ is the orbital moment map.
	\item $\pi_0(\stm_\psi(W))$, the set of isomorphism classes of $\stm_\psi(W)$, is in bijective correspondence with the cohomology classes:
	\[H^2(W,\Z_G\times\R)\cong H^2(W,\Z_G)\times H^2(W,\R)\]
\end{enumerate}
\end{thm}

% Since $\cut$ is in fact an isomorphism of presheaves and it can be shown the identification of symplectic toric bundles with elements of $H^2(W,\Z_G)\times H^2(W,\R)$ commutes with restrictions
% as well, it is fitting to call the elements of $H^2(W,\Z_G) \times H^2(W,\R)$ characteristic classes for symplectic toric manifolds over $\psi$.

%%%%%%%%%%%%%%%%%%%%%%%%%%%%%%%%%%%%%%%%%
\part{Classifying symplectic toric cones}\label{part:stc}
%%%%%%%%%%%%%%%%%%%%%%%%%%%%%%%%%%%%%%%%%

As symplectic toric stratified spaces are built from symplectic toric cones, to understand the former spaces, it is necessary to understand the latter. In Section \ref{s:hule}, we fully describe
these cones as well as their orbital moment maps. Recall that a symplectic toric manifold $(M,\omega,\mu:M\to \fg^*)$ is a {\em symplectic toric cone} if $M$ has a free and proper action of $\R$
commuting with the action of $G$ and, with respect to any action diffeomorphism $\rho_\lambda:M\to M$ for this $\R$ action, we have $\rho_\lambda^*\omega=e^\lambda\omega$.  Additionally, we impose that the moment map $\mu$ for $(M,\omega)$ is {\em the homogeneous moment map}, i.e., that which satisfies $\mu(t\cdot p)=e^t\mu(p)$ for every $t\in \R$ and $p\in M$ (such a moment map for $(M,\omega)$ always exists).

Since any symplectic toric cone $(M,\omega,\mu)$ is, in particular, a symplectic toric manifold, it follows, as in \cite{KarshonLerman}, that the $G$-quotient $M/G$ is a manifold with corners and the orbital moment map $\bar{\mu}: M/G \to \fg^*$ is a unimodular local embedding. As a consequence of $\mu$ being homogeneous, we may conclude that $\bar{\mu}$ satisfies two additional properties: the quotient $M/G$ inherits a free and proper $\R$ action and, with respect to this action, $\bar{\mu}$ is itself homogeneous.  Given an arbitrary manifold with corners $W$ for which there is a unimodular local embedding $\psi : W \to \fg^*$, we call $\psi$ a {\em homogeneous unimodular local embedding} if it and $W$ satisfy these additional properties (see Definition \ref{d:hule}).

As in the case of symplectic toric manifolds, it makes sense to group together symplectic toric cones by orbital moment map: for any homogeneous unimodular local embedding $\psi : W \to \fg^*$, we define {\em the groupoid of symplectic toric cones over $\psi$}, denoted $\stc_\psi(W)$, as the groupoid with objects symplectic toric cones with a $G$-quotient map to $W$ for which $\psi$ is the orbital moment map and with morphisms symplectomorphisms preserving these quotients that are both $G$ and $\R$-equivariant (see Definition \ref{d:stc}). It is important to note that we may not initially be sure this groupoid is non-empty.

For any homogeneous unimodular local embedding $\psi:W\to \fg^*$ and for any $\R$-invariant open subset $U$ of $W$, $\psi|_U$ is a homogeneous unimodular local embedding as well. It follows that, for $\open_{\R}(W)$ the full subcategory of $\open(W)$ of $\R$-invariant open subsets of $W$ (see Definition \ref{d:opensubR}), we may form a presheaf of groupoids $\stc_\psi:\op{\open_{\R}(W)}\to \gpoid$.  

In Section \ref{s:hstb}, we define {\em homogeneous symplectic toric bundles over $\psi$} for any homogeneous unimodular local embedding $\psi:W\to \fg^*$ (see Definition \ref{d:hstb}).  A homogeneous symplectic toric bundle over $\psi$ is $(\pi:P\to W,\omega)$, a symplectic toric bundle over $\psi$ (see Definition \ref{d:stb}), with an $\R$ action on $P$ such that $(P,\omega)$ is a symplectic toric cone with homogeneous moment map $\psi\circ \pi$. Taking a map of homogeneous symplectic toric bundles over $\psi$ to be any isomorphism of principal $G$-bundles over $W$ that is both a symplectomorphism and $\R$-equivariant, we may then define {\em the groupoid of homogeneous symplectic toric bundles over $\psi$}, denoted $\hstb_\psi(W)$.  As in the case of symplectic toric cones, homogeneous symplectic toric bundles also define a presheaf $\hstb_\psi:\op{\open_{\R}(W)}\to \gpoid$.

We also describe in this section some of the important properties of homogeneous symplectic toric bundles. Indeed, the overarching goal of this section is to prove that, at least up to isomorphism class, homogeneous symplectic toric bundles have symplectic form more or less determined by their $\R$ action.  Specifically, we will show that, up to isomorphism, homogeneous symplectic toric bundles are identifiable with {\em homogeneous principal $G$-bundles over $W$}, principal $G$-bundles $\pi:P\to W$ with free and proper $\R$-actions with respect to which $\pi$ is equivariant (see Definition \ref{d:bgr}).  Homogeneous principal $G$-bundles over $W$ also form a presheaf $\BG_\R:\op{\open_{\R}(W)}\to \gpoid$.  Of particular note in Section \ref{s:hstb} is Proposition \ref{p:homogconn}, in which we show that every homogeneous principal $G$-bundle over $W$ admits an exact $G$-invariant symplectic form $\omega$ with respect to which $(P,\omega)$ is a homogeneous symplectic toric bundle.  Additionally, Proposition \ref{p:hstbnohorclass} shows that any two elements $(P,\omega)$ and $(P',\omega')$ of $\hstb_\psi(W)$ are isomorphic exactly when $P$ and $P'$ are isomorphic elements of $\BG_\R(W)$.

In Section \ref{s:hc}, we define the map of presheaves $\hc:\hstb_\psi \to \stc_\psi$ (Definition \ref{d:hc}). This is just the map of presheaves $\cut$ of Karshon and Lerman (see Section \ref{s:KL}) which additionally takes the $\R$ action of a homogeneous symplectic toric bundle to the $\R$ action of a symplectic toric cone. In showing that the category $\hstb_\psi(W)$ is non-empty for any homogeneous unimodular local embedding $\psi:W\to \fg^*$, this functor allows us to conclude that the groupoid $\stc_\psi(W)$ must be non-empty as well. In Theorem \ref{t:hciso}, we show that $\hc$ is an isomorphism of presheaves over $\open_\R(W)$. With this in mind, we may focus on identifying the isomorphism classes of homogeneous symplectic toric bundles.

In Section \ref{s:stccharclasses}, we provide characteristic classes for symplectic toric cones. This is done via Proposition \ref{p:hstbequalshomogbundles}, which shows that, for every homogeneous unimodular local embedding $\psi:W\to \fg^*$, the isomorphism classes of $\stb_\psi(W)$ are in bijective correspondence with the isomorphism classes of $\BG_\R(W)$. This allows us to conclude that homogeneous symplectic toric bundles are classified by Chern classes of $H^2(W,\Z_G)$, for $\Z_G$ the integral lattice $\ker(\exp : \fg \to G)$ (this is the content of Proposition \ref{p:homogchern}).  Finally, we are able to use the isomorphism of presheaves $\hc$ to conclude Theorem \ref{t:stcclassify}: the isomorphism classes of $\stc_\psi(W)$ are in bijective correspondence with the cohomology classes $H^2(W,\Z_G)$.

We end this part by remarking that this classification of symplectic toric cones descends to a classification of contact toric manifolds; indeed, there is a natural isomorphism between contact toric manifolds whose symplectization is a symplectic toric cone over $\psi:W\to \fg^*$ and the cohomology classes $H^2(W,\Z_G)$. Note that, in the case where it is known that $\psi$ comes from a symplectic toric cone, this was shown by Lerman in \cite{LermanCTM}. However, our classification does establish exactly which maps are orbital maps for the symplectizations of (not necessarily connected and compact) contact toric manifolds.

%%%%%%%%%%%%%%%%%%%%%%%%%%%%%%%%%%%%%%%%%
\section{Homogeneous unimodular local embeddings}\label{s:hule}
%%%%%%%%%%%%%%%%%%%%%%%%%%%%%%%%%%%%%%%%%

In this section, we discuss homogeneous unimodular local embeddings; we will eventually show that these maps are exactly the orbital moment maps of symplectic toric cones. We
also define the category of symplectic toric cones over a choice of homogeneous unimodular local embedding. To begin, we reintroduce symplectic cones:

\begin{defn}\label{d:sympcone}
A {\sf symplectic cone} is a symplectic manifold $(M,\omega)$ with a free and proper $\R$ action such that, for every $\lambda\in\R$ with action diffeomorphism $\rho_\lambda:M\to M$, $\rho_\lambda^*\omega = e^\lambda\omega$.

For $(M,\omega)$ a symplectic cone with an action of $G$, we call a triple $(M,\omega,\mu:M\to \fg^*)$ a {\sf symplectic toric cone} if
\begin{itemize}
	\item the actions of $G$ and $\R$ on $M$ commute;
	\item the action of $G$ on $(M,\omega)$ is a symplectic toric action with moment map $\mu$; and
	\item $\mu$ is the homogeneous moment map for $(M,\omega)$: for every $\lambda\in \R$ and $p\in M$, $\mu(\lambda\cdot p)=e^\lambda \mu(p)$.
\end{itemize}
\end{defn}

\begin{rem}\label{r:homogmomentmapformula}
For any symplectic toric cone $(M,\omega)$, the formula
\[\langle \mu,X\rangle = \omega(\Xi,X_M)\]
defines a homogeneous moment map on $(M,\omega)$, where $\Xi$ is the vector field generating the free and proper action of $\R$ on $M$ and $X_M$ is the vector field on $M$ associated to $X\in \fg$.

Additionally, note that there is only one homogeneous moment map associated to $(M,\omega)$.
\end{rem}

Basic known information and properties of symplectic cones and their relationship with contact manifolds has been relegated to Appendix \ref{app:stcs}.

Let's define homogeneous unimodular local embeddings:

\begin{defn}\label{d:hule}
For $W$ a manifold with corners with a free and proper action of $\R$, a {\sf homogeneous unimodular local embedding} is a unimodular local embedding $\psi:W\to \fg^*$ (see
Definition \ref{d:ule}) so that $\psi(t \cdot w) = e^t \psi(w)$ for every $t\in \R$ and $w\in W$.
\end{defn}

\begin{rem}\label{r:actionindexpres}
Since the action of a Lie group on a manifold with corners is carried out via diffeomorphisms, it follows that the action of $\R$ on $W$ must preserve the index of any point $w\in W$ and furthermore, for $S^k(W)$ the stratum of index $k$ elements of $W$ (see the notation section for a description of index), the action of $\R$ on $W$ restricts to a smooth action on $S^k(W)$.
\end{rem}

We will show that the orbital moment map of a symplectic toric cone is a homogeneous unimodular local embedding. To accomplish this, we need the following technical lemma:

\begin{lemma}\label{l:descendingproperactions}
Let $H$ and $K$ be Lie groups and let $K$ be compact. Suppose $X$ is a Hausdorff topological space on which $H$ and $K$ have continuous commuting actions and let $\pi:X\to X/K$ be the quotient. Then the action of $H$ descends to a continuous action on $X/K$.  Furthermore, this action is proper if and only if the action of $H$ on $X$ is proper.
\end{lemma}

\begin{proof}
That the action of $H$ descends to an action on $X/K$ follows from the assumption that the actions of $H$ and $K$ on $X$ commute. Then we have the following commutative diagram:
\begin{equation}\label{eq:actiondiag}\xymatrix{H\times X \ar[r]^\Phi\ar[d]_{\mathrm{id}_H\times \pi} & X\times X \ar[d]^{\pi\times \pi} \\ H\times X/K \ar[r]_{\bar{\Phi}} & X/K\times X/K}\end{equation}
where $\Phi$ and $\bar{\Phi}$ are the maps $\Phi(h,x):=(h\cdot x,x)$ and $\bar{\Phi}(h,[x]):=(h\cdot[x],[x]) = ([h\cdot x],[x])$.

In general, for any commutative diagram of Hausdorff topological spaces
\[\xymatrix{A \ar[r]^{f} \ar[d]_g & B \ar[d]^{h} \\ C\ar[r]_i & D}\]
such that $g$ and $h$ are both surjective and proper, $f$ is proper if and only if $i$ is proper (we leave verification of this fact to the reader).  Since $K$ is compact, we have that $\pi$ is proper and that $X/K$ is Hausdorff (see Theorem 3.1, pp. 38 of \cite{Bredon}).  Therefore, we may apply this more general fact about Hausdorff topological spaces to our above commutative diagram \eqref{eq:actiondiag} and conclude that $\Phi$ (and hence the action of $H$ on $X$) is proper if and only if $\bar{\Phi}$ (and hence the action of $H$ on $X/K$) is proper.
% First, assume the action of $H$ on $X/K$ is proper. Let $C$ be a compact subset of $X\times X$.
% It follows from the commutativity of the above diagram that
% \[\Phi^{-1}(C)\subset (\mathrm{id_H}\times \pi)^{-1}(\bar{\Phi}^{-1}((\pi\times\pi)(C)))\]
% By assumption, $\bar{\Phi}$ is proper. Since $K$ is compact, we also have that $\pi$ is proper (see Theorem 3.1, pp. 38 of \cite{Bredon}). Hence, $(\mathrm{id}_H\times\pi)^{-1}\bar{\Phi}((\pi\times\pi)(C))$ is a compact subset of $H\times X$. Since $X\times X$ is Hausdorff and $C$ is compact, $C$ is a closed subset of $X\times X$. Thus, $\Phi^{-1}(C)$ is a closed subset of the compact set $(\mathrm{id}_H\times \pi)^{-1}(\bar{\Phi}^{-1}((\pi\times\pi)(C)))$ and is therefore also compact.

% On the other hand, assume the action of $H$ on $X$ is proper. Let $D$ be a compact subset of $X/K \times X/K$. Then we may find a compact subset $D'$ of $X/K$ such that $D\subset D'\times D'$.
%  Since $\pi$ is proper, $(\pi\times \pi)^{-1}(D'\times D')=\pi^{-1}(D')\times \pi^{-1}(D')\subset X\times X$ is compact. Thus, $\Phi^{-1}(D'\times D')$ is compact, as is $(\mathrm{id}_H\times \pi)(\Phi^{-1}(D'\times D'))=\bar{\Phi}^{-1}(D'\times D')$.  As $K$ is compact, the quotient $X/K$ is Hausdorff (again, see Theorem 3.1, pp. 38 of \cite{Bredon}); therefore, $D$ is a closed
% subset of $X/K\times X/K$, hence $\bar{\Phi}^{-1}(D)$ is a closed subset of the compact subset $\bar{\Phi}^{-1}(D'\times D')$ of $H \times X/K$ and so is also compact. 
\end{proof}

\begin{prop}
Let $(M,\omega,\mu:M\to \fg^*)$ be a symplectic toric cone. Then for $G$-quotient $\pi:M\to M/G$, the orbital moment map $\bar{\mu}:M/G\to \fg^*$ is a homogeneous unimodular local embedding.
\end{prop}

\begin{proof}
By Proposition \ref{p:KLorbmmaule}, $M/G$ is a manifold with corners and $\bar{\mu}:M/G\to \fg^*$ is a unimodular local embedding. By Lemma \ref{l:descendingproperactions}, the action of $\R$ on $M$ descends to a proper action of $\R$ on $M/G$.

Now, we need to show that the descending action of $\R$ on $M/G$ is free.  Given any $\lambda\in \R$, $g\in G$, and $p\in M$, notice that, if $\lambda\cdot p = g\cdot p$, then
\[e^\lambda\mu(p)=\mu(\lambda\cdot p)=\mu(g\cdot p)=\mu(p)\]
Since $\mu(p)\neq 0$ (see Proposition \ref{p:nozeroinhomogmm}), we must have that $\lambda=0$. Finally, as $\mu$ is homogeneous, $\bar{\mu}$ must be homogeneous as well. 
\end{proof}

We will now discuss a number of results demonstrating how the free and proper $\R$ action on a manifold with corners $W$ associated to a homogeneous unimodular local embedding $\psi:W\to \fg^*$ give us some additional information about $\psi$. First up, we show that, for any homogeneous unimodular local embedding $\psi$, we may choose $\R$-invariant open subsets of $W$ on which $\psi$ is an embedding:

\begin{lemma}\label{l:Rinvunimconeembeddingnbds}
Let $\psi:W\to \fg^*$ be a homogeneous unimodular local embedding. Fix a point $w\in W$ and let $C=C_{\{v_1,\ldots,v_k\},\psi(w)}$ be the unimodular cone with vertex $\psi(w)$ into which $\psi$ is an open embedding near $w$. Then for $\fk$ the Lie subalgebra of $\fg$ generated by $\{v_1 ,\ldots ,v_k\}$, $\psi(w)\in \fk^o$.  Furthermore, there exists an $\R$-invariant neighborhood $U_w$ of $w$ so that $\psi|_{U_w}$ is an open embedding into $C$.
\end{lemma}

\begin{proof}
Note that the unimodular cone $C = C_{\{v_1,\ldots,v_k\},\psi(w)}$ with vertex $\psi(w)$, given by
\[\{\eta\in \fg^*\,|\,\langle \eta-\psi(w),v_i\rangle\geq 0,\,1\leq i\leq k\}\]
contains the affine subspace
\[A = \{\eta\in \fg^*\,|\,\langle \eta-\psi(w),v_i\rangle=0,\,1\leq i\leq k\}=\fk^o+\psi(w)\]
Note also that, in the manifold with corners $C$, we have that $A = S^k(C)$; i.e., $A$ is exactly the stratum of index $k$ elements of $C$. In particular, we may conclude that $\psi(w)$ has index
$k$ in $C$ and, since $\psi$ near $w$ is an open embedding into $C$, it follows that $w$ must also have index $k$ in $W$.

Now, suppose that $U$ is a neighborhood of $w$ such that $\psi|_U$ is an open embedding into $C$.  Choose $t\in \R\backslash\{0\}$ for which $t\cdot w$ is also in $U$. Then, as noted in Remark \ref{r:actionindexpres}, the action of $\R$ must preserve index and so $t\cdot w$ must also have index $k$. It follows that $\psi(t\cdot w)=e^t\psi(w)$ is in $A$. Therefore, since $A$ is an affine subspace, $A$ must contain the whole ray $e^\lambda \psi(w)$ and hence must contain $0$. Thus, $A$ is just the subspace $\fk^o$. In particular, for all $\lambda\in \R$, $\psi(\lambda\cdot w)\in \fk^o$.

Next, for any $t\in \R$ and $\eta\in C$, note that 
\[\langle e^t\eta-\psi(w),v_i\rangle=e^t\langle\eta-e^{-t}\psi(w),v_i\rangle=e^t\langle \eta,v_i\rangle.\]
It thus follows that $e^tC=C$. So for any $U$ on which $\psi|_U$ is an open embedding, we therefore have that $\psi|_{t\cdot U} = e^t\cdot (\psi|_U)$ is an open embedding into $e^tC=C$. Since our choice of $t$ was arbitrary, it follows that $\psi|_{\R\cdot U}$ is an $\R$-equivariant local open embedding into $C$. Since an injective local open embedding must be an open embedding, we conclude by showing that we may choose $U$ small enough so that $\psi|_{\R\cdot U}$ is injective.

Choose a norm $||\cdot ||$ on $\fg^*$ such that $S(\fg^*)$ contains $\psi(w)$. Since $\fg^*\backslash\{0\}$ is equivariantly diffeomorphic $S(\fg^*)\times \R$ and since $C$ is $\R$-invariant, we may choose an open neighborhood $V$ of $\psi(w)$ in $S(\fg^*)$ and an open interval $(-\epsilon,\epsilon)\subset \R$ so that the open subset $\mathcal{U} := C\cap(V\times (-\epsilon,\epsilon))$ of $C$ is completely contained within the image of $\psi|_U$, for $U$ a neighborhood of $w$ on which $\psi|_U$ is an open embedding to $C$.

Notice then that $\mathcal{V}:=U\cap \psi^{-1}(C\cap V)$ is a set containing $w$ which intersects the $\R$-orbits of $W$ at most once and for which $U_w:= \R\cdot \mathcal{V}$ is open. To see that the first item is true: for any $v \in \mathcal{V}$, note that $||\psi(t\cdot v)||=e^t||\psi(v)||$; hence, the action of $\R$ displaces $\psi(v)$ from $S(\fg^*)$ and so $t\cdot v$ can never be in $\mathcal{V}$ (for $t$ nonzero). To see that the second item is true: by the $\R$-equivariance of $\psi$, we have that $(-\epsilon,\epsilon)\cdot \mathcal{V} = U\cap\psi^{-1}(C\times (-\epsilon,\epsilon))$ which is open. Then one may use the action of $\R$ to write $\R\cdot \mathcal{V}$ as a union of translations of this open subset.

Finally, note that, for any two elements $t\cdot v$ and $t'\cdot v'$ of $U_w$ for $v$ and $v'$ in $\mathcal{V}$, $\psi(t\cdot v)=\psi(t'\cdot v')$ exactly when $e^{t-t'}\psi(v)=\psi(v')$. Again, since $\psi(v)$ and $\psi(v')$ are on $S(\fg^*)$, this is only possible for $t=t'$ which, since $\psi$ is injective on $\mathcal{V}$, necessitates that $t\cdot v=t'\cdot v'$. Hence, $\psi|_{U_w}$ is injective and therefore is an $\R$-equivariant embedding into $C$. 
\end{proof}

With this lemma, we may now prove that the $\R$-quotient map associated to a homogeneous unimodular local embedding is a principal $\R$-bundle in the category of manifolds with corners.

\begin{prop}\label{p:WaprincRbundle}
Let $\psi:W\to \fg^*$ be a homogeneous unimodular local embedding and let $q : W \to W/\R$ be an $\R$-quotient map for the action of $\R$ on $W$. Then $q$ is a principal $\R$-bundle in the category of manifolds with corners.
\end{prop}

\begin{rem}
While it's likely one could prove in general that the quotient map associated to a smooth proper action of a Lie group $K$ on a manifold with corners $M$ is a principal $K$-bundle in the category of manifolds with corners using an adapted version of the Slice Theorem for manifold with corners, to the knowledge of the author, such an adapted version of the Slice Theorem doesn't exist (\cite{AlbinMelrose} gets close with ``a Tube Theorem'' for compact group actions, but this relies on a doubling trick requiring $M$ to be compact).  We instead use Lemma \ref{l:Rinvunimconeembeddingnbds} to prove Proposition \ref{p:WaprincRbundle} as opposed to this more general counterpart.
\end{rem}

\begin{proof}[Proof of Proposition \ref{p:WaprincRbundle}]
First, since $W$ is a Hausdorff and locally compact space and the action of $\R$ is proper, $W/\R$ is Hausdorff (see Theorem 1.2.9 of \cite{Palais}).

Now, fix $w\in W$ and let $U_w$ be an $\R$-invariant neighborhood of $w$ for which $\psi|_{U_w}$ is an open embedding into an $\R$-invariant unimodular cone $C$ with vertex $\psi(w)$ in $\fg^*$ (as constructed in Lemma \ref{l:Rinvunimconeembeddingnbds}). Fix a norm $||\cdot ||$ on $\fg^*$ with sphere $S(\fg^*)$. As explained under the statement of Theorem 6.4 in \cite{Joyce}, to show that $C\cap S(\fg^*)$ is a manifold with corners, it is enough to show that $C$ and $S(\fg^*)$  are transverse as manifolds with corners (this is slightly more involved than the ``non-cornered'' definition; again, see \cite{Joyce}). Since the radial action of $\R$ on $C$ must preserve each of the constant index strata $S^k(C)$ of $C$, it follows that each face of $C$ is transverse to $S(\fg^*)$ and
therefore that $C$ and $S(\fg^*)$ must be transverse as manifolds with corners. Hence, $C\cap S(\fg^*)$ and, more importantly, $\psi(U_w)\cap S(\fg^*) \subset C\cap S(\fg^*)$ are manifolds with corners.

Therefore, since each $\R$-orbit of $\fg^*\backslash\{0\}$ intersects $S(\fg^*)$ exactly once and since $\psi(U_w)$ is $\R$-invariant, we have an $\R$-equivariant diffeomorphism 
\[\R\times (\psi(U_w)\cap S(\fg^*))\cong \R\cdot (\psi(U_w)\cap S(\fg^*)) = \psi(U_w)\cong U_w.\]  
From this identification, it follows both that $W/\R$ inherits the structure of a manifold with corners (i.e., via the homeomorphism $U_w/\R\cong \psi(U_w)\cap S(\fg^*)$) and that, with respect to this smooth structure, $q : W \to W/\R$ has smooth trivializations as a principal $\R$-bundle.
\end{proof} 

\begin{rem}
Notice that any principal $\R$-bundle in the category of manifolds with corners is {\em smoothly} trivializable.  Indeed, one may easily adapt the proof of Proposition 2.25, \cite{LermanCTM2} in the smooth manifolds case to the case of manifolds with corners.
\end{rem}

We now prove that the radial action of $\R$ on $\fg^*$ uniquely determines the $\R$ action on the domain of a homogeneous unimodular local embedding.

\begin{lemma}\label{l:uniqueraction}
Let $\psi:W\to \fg^*$ be a unimodular local embedding. Then there is at most one $\R$ action on $W$ with respect to which $\psi$ is a homogeneous unimodular local embedding.
\end{lemma}

\begin{proof}
Assuming that $W$ comes with an $\R$ action with respect to which $\psi$ is a homogeneous unimodular local embedding, $d\psi$ must intertwine the vector field generating the $\R$ action on
$W$ and the radial vector field on $\fg^*$. Since $\psi$ is a local embedding, this property uniquely determines an $\R$ action on $W$. 
\end{proof}

The following proposition will be useful later for quickly demonstrating certain constructed symplectic manifolds are in fact symplectic toric cones.

\begin{prop}\label{p:quickstcdemo}
Let $(M,\omega,\mu)$ be a symplectic toric manifold. Suppose further that $M$ has a free $\R$ action commuting with the action of $G$ such that
\begin{itemize}
	\item the orbital moment map $\bar{\mu}:M/G\to \fg^*$ is a homogeneous unimodular local embedding (with respect to the $\R$ action descending from $M$ to $M/G$); and
	\item for each $\lambda \in \R$ with action diffeomorphism $\rho_\lambda:M\to M$, $\rho_\lambda^*\omega = e^{\lambda}\omega$.
\end{itemize}
Then $(M,\omega,\mu)$ is a symplectic toric cone.
\end{prop}

\begin{proof}
Since $\bar{\mu}$ is a homogeneous unimodular local embedding, the action of $\R$ on $M/G$ is, by definition, proper.  Therefore, by Lemma \ref{l:descendingproperactions}, the $\R$ action on $M$ is proper, so $(M,\omega,\mu)$ is a symplectic toric cone. 
\end{proof}

We now group symplectic toric cones together by orbital moment maps to define the category of symplectic toric cones over a homogeneous unimodular local embedding $\psi$.

\begin{defn}\label{d:stc}
Let $\psi:W\to \fg^*$ be a homogeneous unimodular local embedding. Then a {\sf symplectic toric cone over $\psi$} is a symplectic toric cone $(M,\omega,\mu)$ together with a $G$-quotient $\pi:M\to W$ so that $\mu=\psi\circ \pi$. This data is represented by the triple $(M,\omega, \pi:M\to W)$.

Denote by $\stc_\psi(W)$ {\sf the category of symplectic toric cones over $\psi$}, the groupoid with
\begin{itemize}
	\item objects: symplectic toric cones over $\psi$; and
	\item morphisms: $(G \times \R)$-equivariant symplectomorphisms
	\[\varphi:(M,\omega,\pi:M\to W)\to (M',\omega',\pi':M'\to W)\]
satisfying $\pi'\circ\varphi=\pi$.
\end{itemize}
\end{defn}

Like symplectic toric manifolds over a specific unimodular local embedding, symplectic toric cones over a homogeneous unimodular local embedding form a presheaf. To account for the fact that we only want to consider $\R$-invariant open subsets, we will use a smaller category as the domain of this presheaf.

\begin{defn}\label{d:opensubR}
Let $W$ be a manifold with corners with a free and proper $\R$ action. Denote by $\open_{\R}(W)$ the full subcategory of $\open(W)$ of $\R$-invariant subsets of $W$.
\end{defn}

\begin{prop}\label{p:stcapsheaf}
 Let $\psi:W\to \fg^*$ be a homogeneous unimodular local embedding. Then
\[U\mapsto \stc_\psi(U):=\stc_{\psi|_U}(U)\]
defines a presheaf over $\open_{\R}(W)$.
\end{prop}

\begin{proof}
It is enough to note that, given $(M,\omega,\pi:M\to W)$ any symplectic toric cone over $\psi$, the $\R$ action on $M$ descends to an $\R$ action on $W$ with respect to which $\psi$
is equivariant. By Lemma \ref{l:uniqueraction}, $\pi$ must be $\R$-equivariant since $\psi$ is $\R$-equivariant with respect to both the given action of $\R$ on $W$ as well as the action descending from $M$.  So, for any $\R$-invariant open subset $U$ of $W$,
\[(M,\omega,\pi:M\to W)|_{U}:= \left(\pi^{-1}(U),\omega|_{\pi^{-1}(U)},\pi|_{\pi^{-1}(U)}\right)\]
is a well-defined symplectic toric cone over $\psi|_U$. 
\end{proof}

%%%%%%%%%%%%%%%%%%%%%%%%%%%%%%%%%%%%%%%%%
\section{Homogeneous symplectic toric bundles}\label{s:hstb}
%%%%%%%%%%%%%%%%%%%%%%%%%%%%%%%%%%%%%%%%%

To classify symplectic toric cones, we will use a homogeneous analogue of the isomorphism of presheaves $\cut : \stb_\psi\to \stm_\psi$ from \cite{KarshonLerman} (see Section \ref{s:KL}). In this section, we define homogeneous symplectic toric bundles (the homogeneous analogue to symplectic toric bundles) and present some of their basic properties.

\begin{defn}\label{d:hstb}
 Let $\psi:W\to \fg^*$ be a homogeneous unimodular local embedding. Then a {\sf homogeneous symplectic toric bundle over $\psi$} is a symplectic toric bundle $(\pi:P\to W,\omega)$ (see
Definition \ref{d:stb}) together with a free and proper $\R$ action on $P$ so that
\begin{itemize}
	\item The actions of $G$ and $\R$ on $P$ commute;
	\item $(P,\omega)$ is a symplectic cone with respect to the given $\R$ action; and
	\item $\psi\circ \pi$ is a homogeneous moment map for the action of $G$ on $(P,\omega)$.
\end{itemize}

Denote by $\hstb_\psi(W)$ {\sf the groupoid of homogeneous symplectic toric bundles over $\psi$}. This is the groupoid with objects homogeneous symplectic toric bundles $(\pi:P\to W,\omega)$ and
morphisms maps of symplectic toric principal $G$-bundles (i.e., bundle maps that are also symplectomorphisms) which are additionally $\R$-equivariant.
\end{defn}

Notice that, by requiring that $\psi\circ \pi$ is a homogeneous moment map for the action of $G$ on $(P,\omega)$, we've effectively defined homogeneous symplectic toric bundles as symplectic toric bundles with an $\R$-action with respect to which $(P,\omega,\psi\circ\pi)$ is a symplectic toric cone.  We will find that this condition is equivalent to a simpler, more useful condition we now provide.

\begin{lemma}\label{l:equivquotmap}
Let $\psi:W\to \fg^*$ be a homogeneous unimodular local embedding. Suppose that $(\pi:P\to W,\omega)$ is a symplectic toric bundle with a free and proper $\R$ action, commuting
with the action of $G$, with respect to which $(P,\omega)$ is a symplectic cone. Then $(\pi:P\to W,\omega)$ is a homogeneous symplectic toric bundle if and only if $\pi$ is $\R$-equivariant.
\end{lemma}

\begin{proof} 
The real content of this lemma is showing that $\psi\circ\pi$ is homogeneous if and only if $\pi$ is $\R$-equivariant. First, if $\pi$ is $\R$-equivariant, then for any $p\in P$, 
\[\psi(\pi(t\cdot p)) = \psi(t\cdot (\pi(p))) = e^t\psi(\pi(p)),\] 
so $\psi\circ \pi$ is homogeneous.

On the other hand, suppose $\psi\circ \pi$ is equivariant.  By Lemma \ref{l:descendingproperactions}, note that the action of $\R$ on $P$ descends to a proper action on $\R$ with respect to which $\pi$ is $\R$-equivariant. For any $w\in W$, $p\in P$ satisfying $\pi(p)=w$ and for any $t\in \R$, the action on $W$ descending from $P$ satisfies
\[\psi(t \cdot w) = \psi(t \cdot  \pi(p)) = \psi(\pi(t \cdot p)) = e^t \psi(\pi(p)) = e^t (\psi(w))\]
By Lemma \ref{l:uniqueraction}, there is exactly one action of $\R$ on $W$ with respect to which $\psi$ is $\R$-equivariant; therefore, the action on $P$ descends to the unique action making $\psi:W\to \fg^*$ a homogeneous unimodular local embedding and so $\pi$ is equivariant with respect to the given actions of $\R$ on $P$ and $W$. 
\end{proof}

Principal $G$-bundles with an action of $\R$ as in the statement of Lemma \ref{l:equivquotmap} will be important throughout the rest of Part \ref{part:stc}; they form the following presheaf:

\begin{defn}\label{d:bgr}
Let $W$ be a manifold with corners with a free and proper $\R$ action. Then a {\sf homogeneous principal $G$-bundle over $W$} is a principal $G$-bundle $\pi:P\to W$ with a free and proper $\R$ action that commutes with the action of $G$ and with respect to which $\pi$ is equivariant.

The presheaf of homogeneous principal $G$-bundles over $W$ is the presheaf of groupoids $\BG_{\R}:\op{\open_{\R}(W)} \to \gpoid$ so that, for every $\R$-invariant open subset $U$ of $W$, the groupoid $\BG_{\R}(U)$ is that with
\begin{itemize}
\item objects: homogeneous principal $G$-bundles over $W$; and
\item morphisms: $\R$-equivariant bundle isomorphisms.
\end{itemize}
\end{defn}

As in the case of symplectic toric cones, the collection of groupoids of homogeneous symplectic toric bundles over a homogeneous unimodular local embedding $\psi:W\to \fg^*$ is a
presheaf of groupoids over $\open_{\R}(W)$.

\begin{prop}\label{p:hstbapresheaf}
 Let $\psi:W\to \fg^*$ be a homogeneous unimodular local embedding. Then
\[U \mapsto \hstb_\psi(U) := \hstb_{\psi|_U}(U)\] 
with the appropriately chosen restriction morphisms defines a presheaf of groupoids $\hstb_\psi: \op{\open_\R(W)}\to \gpoid$.
\end{prop}

As the justification here is essentially the same as in Proposition \ref{p:stcapsheaf}, we omit the proof.

\begin{rem}
In fact, $\hstb_\psi:\op{\open_{\R}(W)}\to \gpoid$ is a stack; this will be important later, but as the proof is a slight adaptation of the proof that principal bundles over a site form a stack, we relegate the proof to the appendix (Proposition \ref{p:hstbastack}).
\end{rem}
 
As in the case of symplectic toric bundles, we will be able to show that, for $\pi:P\to W$ any homogeneous principal $G$-bundle over $W$, we may use a connection to build a symplectic form with respect to which it is a homogeneous symplectic toric bundle. To show this, we will need the following technical lemma.

\begin{lemma}\label{l:techlemma}
Let $\pi:P\to B$ be a principal $G$-bundle of manifolds with corners. Further, suppose $P$ and $B$ admit free and proper actions of $\R$ with respect to which $\pi$ is $\R$-equivariant and the $\R$-quotient $q' : B \to B/\R$ is a principal $\R$-bundle of manifolds with corners. Let $q : P \to P/\R$ be an $\R$-quotient. Finally, suppose that the actions of $G$ and $\R$ on $P$ commute.  

Then the unique map $\varpi : P/\R \to B/\R$ making the diagram
\begin{equation}\label{eq:technicaldiagram}\xymatrix{P\ar[r]^q \ar[d]_\pi & P/\R \ar[d]^{\varpi} \\ B \ar[r]_{q'} & B/\R}\end{equation}
commute is a principal $G$-bundle of manifolds with corners and $q : P \to P/\R$ is a principal $\R$-bundle of manifolds with corners.
\end{lemma}

\begin{proof}
The existence of $\varpi$ is a consequence of the universal property of a quotient: as $\pi$ is $\R$-equivariant and $q'$ is $\R$-invariant, the composition $q\circ \pi$ must be $\R$-invariant. Since the actions of $G$ and $\R$ on $P$ commute, the action of $G$ on $P$ descends to an action of $G$ on $P/\R$.

Let $U$ be a contractible subset of $B/\R$ and $s : U \to B$ be a smooth section. This yields a smooth trivialization $U \times \R \cong V$ for $V$ an open subset of $B$. As $V$ is contractible, there exists another smooth section $t : V \to P$ yielding a smooth trivialization $G\times V\cong X$ for $X$ an open subset of $P$. As $\pi$ is $\R$-equivariant, it follows that we have a smooth $(G\times \R)$-equivariant diffeomorphism $G \times \R \times U \cong X$. By the commutativity of diagram \eqref{eq:technicaldiagram} (and using the fact these are all quotient maps), we have that $X = \pi^{-1}({q'}^{-1}(U)) = q^{-1} (\varpi^{-1}(U))$. In particular, this means that $q(X) = \varpi^{-1}(U)$ and so the diffeomorphism $G \times \R \times U\cong X$ descends to the equivariant homeomorphism $G\times U \cong \varpi^{-1}(U)$.

Via these homeomorphisms, $P/\R$ inherits the structure of a manifold with corners (we once again use Theorem 1.2.9 of \cite{Palais} to conclude that, since the action of $\R$ on $P$ is proper, $P/\R$ is Hausdorff).  With respect to this smooth structure, $\varpi$ is naturally smooth and $ \varpi : P/\R \to B/\R$ is a principal $G$-bundle in the category of manifolds with corners. Additionally, we’ve demonstrated that $q : P \to P/\R$ is a principal $\R$-bundle of manifolds with corners as well.
\end{proof} 

Now, we build a symplectic form for any principal $G$-bundle $\pi:P\to W$ with an appropriate $\R$ action.

\begin{prop}\label{p:homogconn}
Let $\psi: W \to \fg^*$ be a homogeneous unimodular local embedding and let $\pi:P\to W$ be a homogeneous principal $G$-bundle over $W$. Then there exists a connection
1-form $A \in \Omega^1(P,\fg)^G$ so that $(\pi:P\to W,d\langle \psi\circ \pi, A\rangle)$ is a homogeneous symplectic toric bundle.
\end{prop}

\begin{proof}
First, as shown in \cite{KarshonLerman}, any connection 1-form $A$ induces a $G$-invariant symplectic form $d\langle \psi\circ \pi, A\rangle$ with respect to which $\psi \circ \pi$ is a moment map (see Lemma \ref{l:KLbijectionlemma}). We will construct a particular connection so that the form $d\langle\psi\circ \pi,A\rangle$ satisfies the additional conditions required of a homogeneous symplectic toric bundle.

Let $Q := P/\R$ and $B := W/\R$ with $\R$-quotient maps $q : P \to Q$ and $q': W \to B$.  By Proposition \ref{p:WaprincRbundle}, $q'$ is a principal $\R$-bundle of manifolds with corners and so, via Lemma \ref{l:techlemma}, we have the following commutative diagram
\[\xymatrix{P\ar[r]^q \ar[d]_\pi  & Q \ar[d]^{\varpi} \\ W \ar[r]_{q'} & B}\]
where $\varpi : Q \to B$ is a principal $G$-bundle and $q : P \to Q$ is a principal $\R$-bundle, both in the category of manifolds with corners.

Since $Q$ is a principal $G$-bundle, it follows that $\varpi\times \mathrm{id}_{\R} : Q\times \R \to B\times \R\cong W$ is a principal $G$-bundle over $W$ (with respect to the trivially extended action of $G$ on $Q\times \R$). As $q : P \to Q$ is a principal $\R$-bundle and hence trivializable, there is an isomorphism $\phi:P\to Q\times  \R$ of principal $\R$-bundles over $Q$. Since $q$ is $G$-equivariant, $\phi$ must be $G$-equivariant as well.

Let $A'$ be any connection 1-form on the principal $G$-bundle $\varpi \times \mathrm{id}_{\R} : Q \times \R \to B \times \R$ extended trivially from a connection 1-form on the $G$-bundle $ \varpi : Q \to B$. We now finish by showing that, for $A := \phi^*A'$ , $d\langle \psi\circ \pi, A\rangle$ satisfies the necessary conditions for a symplectic form of a symplectic cone.

Fix a real number $\lambda$ and let $\rho_\lambda:P\to P$ and $\tau_\lambda:Q\times \R\to Q\times \R$ be the diffeomorphisms associated to the action of $\lambda$. As $\phi$ is $\R$-equivariant, $\phi\circ \rho_\lambda=\tau_\lambda\circ \phi$ and, by design, we have $\tau^*_\lambda A'=A'$.  Using these facts, we calculate:
\begin{align*}
\rho_\lambda^*d\langle \psi\circ \pi, A\rangle &= d\langle \psi\circ \pi\circ \rho_\lambda,\rho_\lambda^*(\phi^*A')\rangle = d\langle e^\lambda\cdot (\psi\circ \pi),\phi^*(\tau^*_\lambda A')\rangle \\ &=d(e^\lambda\langle \psi \circ \pi,\phi^*A'\rangle) = e^\lambda d\langle\psi\circ\pi,A\rangle
\end{align*}
Therefore, using Proposition \ref{p:quickstcdemo}, we may conclude that the action of $\R$ on $P$ is proper and that $(\pi:P\to W, d\langle \psi\circ \pi, A\rangle)$ is a homogeneous symplectic toric bundle over $\psi$. 
\end{proof}

We will soon show that two homogeneous symplectic toric bundles are isomorphic in $\hstb_\psi$ exactly when there is an $\R$-equivariant bundle isomorphism between them. To prove this, we need the following lemma.

\begin{lemma}\label{l:hstbdiffext}
Let $\psi : W \to \fg^*$ be a homogeneous unimodular local embedding and let $\pi:P\to W$ be a homogeneous principal $G$-bundle over $W$. Suppose $\omega$ and $\omega'$ are two symplectic forms so that $(\pi:P\to W,\omega)$ and $(\pi:P\to W,\omega')$ are both homogeneous symplectic toric bundles. Then the form $\omega-\omega'$ is basic and, for $\omega-\omega'=\pi^*\beta$, $\beta$ is exact.

Furthermore, there is a primitive $\gamma$ of $\beta$ satisfying $\rho_\lambda^*\gamma = e^\lambda \gamma$ for any $\lambda\in \R$ with action diffeomorphism $\rho_\lambda:W\to W$.
\end{lemma}

\begin{proof}
Fix a connection 1-form $A$ for which $(\pi:P\to W, d\langle \psi\circ \pi, A\rangle)$ is a homogeneous symplectic toric bundle (as constructed in Proposition \ref{p:homogconn}).  Then $\omega-d\langle \psi\circ \pi, A\rangle$ and $\omega'-d\langle \psi\circ \pi,A\rangle$ are both basic (see Lemma \ref{l:KLbijectionlemma}); thus, the difference $\omega-\omega'$ is basic as well.  Write $\omega-\omega'=\pi^*\beta$.

Fix $\lambda\in \R$. Writing $\tau_\lambda:P\to P$ for the action isomorphism of $\lambda$ on $P$, we have by assumption that $\tau_\lambda^*\omega = e^\lambda\omega$ and $\tau_\lambda^*\omega' = e^\lambda\omega'$.  So $\tau_\lambda^*(\pi^*\beta) = e^\lambda (\pi^*\beta)$.

As $\pi$ is $\R$-equivariant, we have that $\pi\circ \tau_\lambda = \rho_\lambda\circ \pi$ (for $\rho_\lambda$ again the action diffeomorphism for the action of $\lambda$ on $W$). So, we calculate:
\[\pi^*(\rho_\lambda^*\beta) = \tau_\lambda^*(\pi^*\beta) = e^\lambda(\pi^*\beta) =\pi^*(e^\lambda \beta)\]
Since $\pi$ is a submersion, it follows that $\rho_\lambda^*\beta = e^\lambda\beta$.

Finally, write $\Xi$ for the vector field on $W$ generating the action of $\R$. Then $\beta$ satisfies $\mathcal{L}_\Xi\beta = \beta$ meaning, since $\beta$ is closed, $\gamma :=\iota_{\Xi}\beta$ is a primitive for $\beta$.  Since $\mathcal{L}_\Xi \gamma = \iota_\Xi(d\gamma)=\gamma$, it follows that $\rho_\lambda^*\gamma = e^\lambda\gamma$.
\end{proof}

\begin{rem}
Note that the proof of Lemma \ref{l:hstbdiffext} may be reversed to conclude that, for $\gamma$ satisfying $\rho_\lambda^*\gamma=e^\lambda\gamma$ for all $\lambda\in \R$, $(\pi:P\to W,\omega+\pi^*d\gamma)$ is a homogeneous symplectic toric bundle over $\psi$.
\end{rem}

We now prove that adding an exact form satisfying the conditions of Lemma \ref{l:hstbdiffext} does not change our isomorphism class in $\stb_\psi(W)$.

\begin{lemma}\label{l:homogmoser}
Let $\psi:W\to \fg^*$ be a homogeneous unimodular local embedding and let $(\pi:P\to W,\omega)$ be a homogeneous symplectic toric bundle. Let $\gamma$ be a 1-form on $W$ satisfying $\rho_\lambda^*\gamma = e^\lambda\gamma$ for every real $\lambda$ with action diffeomorphism $\rho_\lambda$. Then there is an isomorphism of homogeneous symplectic toric bundles
\[\varphi : (\pi : P \to W,\omega) \to (\pi : P \to W,\omega + \pi^*d\gamma)\]
\end{lemma}
\begin{proof}
We will repeat the proof of Lemma 3.3 of \cite{KarshonLerman} with the addition of an $\R$ action; for the convenience of the reader, we will sketch the borrowed details. To build the map $f$, we
can use Moser’s Method on the family of symplectic forms
\[\omega_t = \omega + t\pi^*d\gamma,\; t \in[0,1].\]
Then there is a unique time-dependent vector field $X_t$ on $P$ satisfying
\begin{equation}\label{eq:Moser} \iota_{X_t}\omega_t = -\pi^*\gamma.\end{equation}
Since $X_t$ is $G$-invariant and tangent to the compact fibers of $\pi$, the time 1 flow $\varphi:P\to P$ of $X_t$ exists and is $G$-equivariant. Therefore, $\pi\circ \varphi =\pi$ and, as is standard with Moser’s Method (see \cite{Moser}), $\varphi$ satisfies
\[\varphi^*(\omega+\pi^*d\gamma) = \varphi^*(\omega_1)=\omega_0=\omega\]

It remains to be shown that this bundle isomorphism is $\R$-equivariant. It is enough to show that the vector field $X_t$ is $\R$-invariant. Fix $\lambda\in \R$ with action diffeomorphism $\rho_\lambda$. Since $\rho^*_\lambda=e^\lambda\omega$ and $\rho^*_\lambda(\pi^*\gamma)=e^\lambda(\pi^*\lambda)$, $\omega_t$ must also satisfy $\rho_\lambda^*\omega_t=e^\lambda\omega_t$ for all $t$.  So
\[\iota_{(\rho_\lambda)_*X_t}\omega_t=\rho^*_{-\lambda}(\iota_{X_t}e^\lambda\omega_t)=e^\lambda\rho^*_{-\lambda}(-\pi^*\gamma)=e^\lambda e^{-\lambda}(-\pi^*\gamma)=-\pi^*\gamma\]
Thus, since the equality \ref{eq:Moser} uniquely determines $X_t$, $(\rho_\lambda)_*X_t=X_t$ meaning $X_t$ is $\R$-invariant and therefore $\varphi$ is $\R$-equivariant.
\end{proof}

From the previous two lemmas, we may easily conclude the following proposition:

\begin{prop}\label{p:hstbnohorclass}
Let $\psi:W\to \fg^*$ be a homogeneous unimodular local embedding. Then $(\pi:P\to W,\omega)$ and $(\pi':P'\to W,\omega')$ are isomorphic homogeneous symplectic toric bundles over $\psi$ if and only if $P$ and $P'$ are isomorphic homogeneous principal $G$-bundles over $W$.
\end{prop}

\begin{proof}
Because an isomorphism in $\hstb_\psi(W)$ is in particular an $\R$-equivariant isomorphism of principal $G$-bundles, one direction is given by definition.

Suppose there exists an isomorphism $\varphi : P \to P'$ of homogeneous principal $G$-bundles over $W$. By Lemma \ref{l:hstbdiffext}, the difference $\varphi^*(\omega')-\omega$ is basic and, for $\varphi^*(\omega')-\omega = \pi^*\beta$, $\beta$ has primitive $\gamma$ satisfying $\rho_\lambda^*\gamma = e^\lambda \gamma$ for $\rho_\lambda$ the action diffeomorphism associated to each $\lambda\in \R$.  By Lemma \ref{l:homogmoser}, there exists an $\R$-equivariant bundle isomorphism $\phi:P\to P$ with $\phi^*(\omega+d(\pi^*\gamma))=\omega$.  Therefore, $\varphi\circ \phi$ is an isomorphism of homogeneous symplectic toric bundles. 
\end{proof}

%%%%%%%%%%%%%%%%%%%%%%%%%%%%%%%%%%%%%%%%%
\section{The morphism of presheaves $\hc:\hstb_\psi\to \stc_\psi$}\label{s:hc}
%%%%%%%%%%%%%%%%%%%%%%%%%%%%%%%%%%%%%%%%%

In this section, we introduce a functor $\hc :\stb_\psi(W)\to \stc_\psi(W)$. We then show that $\hc$ is an equivalence of categories; in fact, thinking of $\hstb_\psi$ and $\stc_\psi$ as presheaves over $\open_{\R}(W)$, $\hc$ is an isomorphism of presheaves. This functor is the homogeneous version of the equivalence of categories $\cut:\stb_\psi(W)\to \stm_\psi(W)$ of \cite{KarshonLerman} (see Section \ref{s:KL}).

We first verify that when the functor $\cut$ is applied to a homogeneous symplectic toric bundle, the $\R$ action on the bundle descends to an $\R$ action on the resulting symplectic toric
manifold inducing the structure of a symplectic toric manifold.

\begin{prop}\label{p:hstbtostcobjects}
Let $\psi:W\to \fg^*$ be a homogeneous unimodular local embedding and let $(\pi:\to W,\omega)$ be a homogeneous symplectic toric bundle over $\psi$. Then, regarding $(P,\omega)$ as a symplectic toric bundle over $\psi$, the symplectic toric manifold $\cut(P,\omega)$ inherits an $\R$ action from $(P,\omega)$ with respect to which $\cut(P,\omega)$ is a symplectic toric cone over $\psi$.
\end{prop}

\begin{proof}
Fix an element $w\in W$. Let $C_w:=C_{\{v_1,\ldots,v_k\},\psi(w)}$ be the unimodular cone with vertex $\psi(w)$ onto which $\psi$ is an open embedding near $w$.  Let $\fk$ be the Lie subalgebra of $\fg$ for which $\{v_1,\ldots,v_k\}$ is a basis and let $K_w \leq G$ be the corresponding subtorus. As before, let $\iota:\fk\to \fg$ be the inclusion with dual $\iota^*:\fg^*\to \fk^*$.  By Lemma \ref{l:Rinvunimconeembeddingnbds}, $\psi(w)\in \fk^o$ so $\iota^*(\psi(w))=0$.  As in Construction \ref{const:cuts}, define the cone $C'_w$ by
\[C'_w:=\{\xi\in \fk^*\,|\, \langle \xi,v_i\rangle\geq 0,\,1\leq i \leq k\}.\]
Again using Lemma \ref{l:Rinvunimconeembeddingnbds}, we may choose an $\R$-invariant neighborhood $U_w$ of $w$ on which $\psi$ is an open embedding into $C_w$.  It follows that, after making a choice of section of $\iota^*:\fg^*\to\fk^*$, we may make an $\R$-equivariant identification $C_w\cong \fk^o\times C'_w$.  By shrinking $U_w$ as necessary, one may choose a contractible $\R$-invariant open subset $\mathcal{U}$ of $\fk^o$ and a contractible open neighborhood $\mathcal{V}$ of $0$ in $\fk^*$ so that, for $\mathcal{V}':=C'_w\cap \mathcal{V}$ the above identification restricts to the identification $U_w\cong \mathcal{U}\times \mathcal{V}'$.  Again, as in Construction \ref{const:cuts}, we may conclude that, for $\nu:=\iota^*\circ \psi\circ \pi$, there exists a manifold without boundary $\tilde{P}$ containing $P|_{U_w}$ as a domain and an extension of the map $\nu:P|_{U_w}\to \mathcal{V}'$ to a smooth map $\tilde{\nu}:\tilde{P}\to \mathcal{V}$.

Now, let $\rho:K_w\to \mathrm{Sp}\left(\C^k,\omega_{\C^k}\right)$ be the symplectic representation with weights $\{v^*_1,\ldots,v^*_k\}$ (for $\omega_{\C^k}= \frac{\sqrt{-1}}{2\pi}\sum dz_j\wedge d\bar{z}_j$ and action as in equation \eqref{eq:symprep}). We fix the moment map
\[\mu_w:\C^k\to \fk^*,\;\;\; \mu_w(z_1,\ldots,z_k):=-\sum_{j=1}^k |z_j|^2v_j^*\]
for this space.   Then the $K_w$ action on $(P|_{U_w}\times \C^k,\omega\oplus\omega_{\C^k})$ has moment map $\Phi(p,z):=\nu(p)+\mu_w(z)$ which has extension $\tilde{\Phi}(p,z):=\tilde{\nu}(p)+\mu_w(z)$ to the domain $\tilde{P}\times \C^k$ containing $P|_{U_w}\times \C^k$.

The condition $(p,z)\in \tilde{\Phi}^{-1}(0)$ imposes that $\tilde{\nu}(p)=-\mu_w(z)$, meaning the image of $\tilde{\nu}$ must be contained in $C'_w$. It therefore follows that $\Phi^{-1}(0)=\tilde{\Phi}^{-1}(0)$. Thus, by Theorem \ref{t:reduction}, the reduction $(P|_{U_w}\times\C_k)//_0\, K_w$ is a symplectic manifold.

As we’ve proceeded using the same method as in Construction \ref{const:cuts}, it follows we may use the homeomorphisms defined in Construction \ref{const:transitions} to symplectize $\ctop(P,\omega)$. To finish, we need only show that there are compatible smooth $\R$ actions on each $(P|_{U_w}\times \C^k)//_0\, K_w$ with respect to which the inherited symplectic form on $\cut(P,\omega)$ is homogeneous.

So let $\R$ act on $P|_{U_w}\times \C^k$ via the diagonal action given by the $\R$ action on $P$ restricted to $P|_{U_w}$ and via the ``half-radial
action'' on $\C^k$ : the action $t\cdot z:=e^{\frac{1}{2}t}z$.  $\mu_w:\C^k\to \fk^*$ is homogeneous with respect to this action of $\R$ on $\C^k$ and, as $\nu : P|_{U_w}\to \fk^*$ is also homogeneous, it follows that the action of $\R$ preserves the level set $\Phi^{-1}(0)$. Since the actions of $K_w$ and $\R$ commute, the action of $\R$ descends to a smooth action on $(P|_{U_w} \times \C^k )//_0\, K_w$.  After checking that the transition homeomorphisms (again, as built in Construction \ref{const:transitions})
\[\alpha^P_w:\ctop(P|_{U_w})\to (P|_{U_w}\times \C^k)//_0\, K_w\]
are $\R$-equivariant with respect to the action of $\R$ on $\ctop(P,\omega)$ descending from $P$ and the above described action of $\R$ on $(P|_{U_w}\times \C^k)//_0\, K_w$, we may conclude that the induced $\R$ action on $\cut(P,\omega)$ is in fact a smooth action.

Finally, we check that the symplectic form $\eta$ on $\cut(P,\omega)$ satisfies $\rho_\lambda^*\eta=e^\lambda\eta$ for every action diffeomorphism $\rho_\lambda$ associated to $\lambda\in \R$.  On the open dense interior $\mathring{W}$ of $W$, the functor $\cut$ is the identity (see Remark \ref{r:coninterior}); i.e., for an open subset $U\subset \mathring{W}$, $(P|_U,\omega,\pi : P|_U\to U) = \cut(P,\omega)|_U$ as symplectic toric manifolds over $\psi|_U$. Thus,
\[\rho_\lambda^*(\eta|_U)=\rho_\lambda^*(\omega|_U)=e^\lambda\omega|_U=e^\lambda\eta|_U\]
As this identity holds on the open dense subset $\cut(P,\omega)|_{\mathring{W}}$ of $\cut(P,\omega)$, it follows this holds over all $\cut(P,\omega)$. Therefore, the above action of $\R$ on $\cut(P,\omega)$ renders $\cut(P,\omega)$ a symplectic toric cone (the properness of the $\R$ action on $M$ is ensured by Proposition \ref{p:quickstcdemo}).
\end{proof}

\begin{defn}\label{d:hc}
Let $\psi:W\to \fg^*$ be a homogeneous unimodular local embedding. Then $\hc:\hstb_\psi(W)\to \stc_\psi(W)$ is the functor taking a homogeneous symplectic toric bundle $(P,\omega)$ to the symplectic toric manifold $\cut(P,\omega)$ with $\R$ action inherited from $(P,\omega)$, as outlined in Proposition \ref{p:hstbtostcobjects}.  For a morphism $\varphi:(P,\omega)\to (P',\omega')$, let $\hc(\varphi):= \cut(\varphi)$. Since $\varphi$ is $(G\times\R)$-equivariant, $\hc(\varphi)$ is $(G\times\R)$-equivariant as well. As with $\cut$, $\hc : \hstb_\psi\to \stc_\psi$
is a map of presheaves (naturality of $\hc$ follows from naturality of $\cut$).
\end{defn}

To begin showing $\hc$ is an isomorphism of presheaves, we first show that $\hc$ is fully faithful.

\begin{lemma}\label{l:hcfullyfaithful}
Let $\psi:W\to \fg^*$ be a homogeneous unimodular local embedding. Then for every $\R$-invariant open subset $U$ of $W$, the functor $\hc_U : \hstb_\psi(U)\to \stc_\psi(U)$ is fully faithful.
\end{lemma}

\begin{proof}
As $\psi|_U:U\to \fg^*$ is also a homogeneous unimodular local embedding and the groupoid $\hstb_\psi(U)$ is, by definition, the groupoid $\hstb_{\psi|_U}(U)$, we need only worry
about the case of $U = W$ as this will generalize to any $\R$-invariant open subset $U \subset W$.

Note we have the following commutative diagram
\begin{equation}\label{eq:hcfunctdiagram}\xymatrix{\hstb_\psi(W) \ar[r]^{\iota_h} \ar[d]_{\hc} & \stb_\psi(W) \ar[d]^{\cut} \\ \stc_\psi(W) \ar[r]_{\iota_c} & \stm_\psi(W)}\end{equation}
where $\iota_h$ and $\iota_c$ are the forgetful functors.  $\iota_h$ and $\iota_c$ are both faithful (the maps in the source category for both functors are distinguished only by $\R$-equivariance).  Therefore, since $\cut$ and $\iota_h$ are faithful and $\cut\circ \iota_h = \iota_c\circ \hc$, $\iota_c\circ \hc$ is faithful; so since $\iota_c$ is faithful, $\hc$ is also faithful.

Now we show $\hc$ is full. Fix two homogeneous symplectic toric bundles $(P,\omega)$ and $(P',\omega')$ in $\hstb_\psi(W)$. Let 
\[f:\hc(\pi:P\to W,\omega) \to \hc(\pi':P'\to W,\omega')\]
be a map of symplectic toric cones. Using diagram \eqref{eq:hcfunctdiagram} and the fullness of $\cut$, it follows there is a map of symplectic toric bundles
\[\varphi:\iota_h(P,\omega)\to \iota_h(P',\omega')\]
satisfying $\cut(\varphi)=\iota_c(f)$.

Let $d:P'\times_{\pi',W,\pi'}P'\to G$ be the division map for $P'$: the smooth map with $d(p,p')$ the unique element of $G$ such that $p\cdot d(p,p')=p'$ for any $p,p'\in P'$ with $\pi'(p)=\pi'(p)$.  For each element $t\in \R$, define
\[\tilde{\varphi}_t:P\to G,\;\;\; \tilde{\varphi}_t(p):=d(\varphi(t\cdot p), t\cdot \varphi(p))\]
Recall that $c|_{\mathring{W}}$ is the identity on the open dense top stratum $\mathring{W}\subset W$ (see Remark \ref{r:coninterior}).  So $\cut(\varphi|_{\mathring{W}})=\iota_c(f|_{\mathring{W}})$ is $\R$-equivariant. Therefore, we may conclude that, for every $t\in \R$, $\tilde{\varphi}_t \equiv e$ on $P|_{\mathring{W}}$ (for $e\in G$ the identity element).  $\tilde{\varphi}_t$ is continuous and constant on the open dense subset $P|_{\mathring{W}}$ of $P$, it follows that $\tilde{\varphi}_t \equiv e$ on all $P$ and so $\varphi$ is $\R$-equivariant and therefore
actually a map of homogeneous symplectic toric bundles $\varphi\in \hstb_\psi(W)$.  Using diagram \eqref{eq:hcfunctdiagram} once more, we may conclude that $\hc(\varphi)=f$. Hence, $\hc$ is full.
\end{proof} 

We require two more lemmas before we can use Lemma \ref{l:isoofstackstechlemma} to prove that $\hc:\stb_\psi\to \stc_\psi$ is an isomorphism of presheaves.

\begin{lemma}\label{l:stclocconnected}
Let $\psi:W\to \fg^*$ be a homogeneous unimodular local embedding. Then any two symplectic toric cones over $\psi$ are locally isomorphic; explicitly, for $(M,\omega,\pi:M\to W)$ and $(M',\omega',\pi':M'\to W)$ two symplectic toric cones over $\psi$, there is an open cover $\{U_\alpha\}_{\alpha\in A}$ of $W$ by $\R$-invariant open subsets and a collection of isomorphisms
\[\{\varphi_\alpha : (M,\omega,\pi)|_{U_\alpha} \to (M', \omega',\pi')|_{U_\alpha} \in \stc_\psi(U_\alpha)\}_{\alpha\in A}.\]
\end{lemma}

As the proof involves a number of known results about the relationship between symplectic toric cones and contact toric manifolds, we relegate the proof of Lemma \ref{l:stclocconnected} to Appendix \ref{app:stcs}.

\begin{lemma}\label{l:stcaprestack}
Let $\psi:W\to \fg^*$ be a homogeneous unimodular local embedding. Then the presheaf $\stc_\psi: \op{\open_\R(W)}\to \gpoid$ is a prestack (see Definition \ref{d:prestack}).
\end{lemma}

\begin{proof}
To show $\stc_\psi$ is a prestack, we must show that, for every $\R$-invariant open subset $U$ of $W$ and for any two symplectic toric cones $(M,\omega,\pi:M\to U)$ and $(M',\omega',\pi'\:M'\to U)$ in $\stc_\psi(U)$, the presheaf
\[\mathrm{Hom}((M,\omega,\pi),(M',\omega',\pi')):\op{\open_{\R}(U)}\to \set,\;V\mapsto \mathrm{Hom}_{\stc_\psi}((M,\omega,\pi)|_V,(M',\omega',\pi')|_V)\]
is a sheaf of sets. It is routine to check that the additional properties imposed on the smooth maps of $\stc_\psi$ are local. 
\end{proof}

We may put together the results of this section to prove that $\hc$ is an isomorphism of presheaves.

\begin{thm}\label{t:hciso}
Let $\psi:W\to \fg^*$ be a homogeneous unimodular local embedding. Then $\hc :\hstb_\psi\to \stc_\psi$ is an isomorphism of presheaves.
\end{thm}

\begin{proof}
We have from Lemma \ref{l:stcaprestack} that $\stc_\psi$ is a prestack and from Proposition \ref{p:hstbastack} that $\hstb_\psi$ is a stack. To see that $\hstb_\psi(U)$ is non-empty for every open $\R$-invariant subset $U$ of $W$, note that, in Proposition \ref{p:homogconn}, we showed that every homogeneous principal $G$-bundle comes with a symplectic form with respect to which it is a homogeneous symplectic toric bundle; in particular, then, the trivial principal $G$-bundle $U\times G\to U$ comes with a symplectic form with respect to which it is a homogeneous symplectic toric bundle over $\psi|_U$.  

From Lemma \ref{l:stclocconnected}, we have that, for any open $\R$-invariant subset $U$ of $W$, any two elements in the groupoid $\stc_\psi(U)$ are locally isomorphic; in other words, $\stc_\psi$ is transitive. Finally, from Lemma \ref{l:hcfullyfaithful}, we have that $\hc_U:\hstb_\psi(U)\to \stc_\psi(U)$ is fully faithful for each $U$.

Thus, $\hstb_\psi$, $\stc_\psi$, and $\hc$ satisfy all the hypotheses of Lemma \ref{l:isoofstackstechlemma} (see also Remark \ref{r:isoofstackstechlemmarem}) and so we may conclude that $\hc$ is an isomorphism of presheaves. 
\end{proof}

%%%%%%%%%%%%%%%%%%%%%%%%%%%%%%%%%%%%%%%%%
\section{Characteristic classes for symplectic toric cones}\label{s:stccharclasses}
%%%%%%%%%%%%%%%%%%%%%%%%%%%%%%%%%%%%%%%%%

In this section, we prove that the collection of symplectic toric cones over a particular homogeneous unimodular local embedding $\psi:W\to \fg^*$ are naturally isomorphic the classes
$H^2(W,\Z_G)$. First, we set some notation.

\begin{notation}
 Given two categories $\mathcal{C}$ and $\mathcal{D}$ and a functor $F : \mathcal{C} \to \mathcal{D}$, denote by $\pi_0 \mathcal{C}$ and $\pi_0 \mathcal{D}$ the collections of isomorphism classes of $\mathcal{C}$ and $\mathcal{D}$, respectively, and denote by  $\pi_0 F :\pi_0 \mathcal{C} \to \pi_0 \mathcal{D}$ the function $\pi_0F([c]) := [F(c)]$ for each class $[c] \in\pi_0 \mathcal{C}$. Note $\pi_0 F$ is well-defined as $F$ is a functor.
\end{notation}

\begin{rem}\label{r:pizeroofafunctor}
For $X$ a topological space, suppose $\mathcal{F} : \op{\open(X)}\to \gpoid$ is a presheaf of groupoids. Then there is a presheaf of sets $\pi_0\mathcal{F} : \op{\open(X)}\to \set$ with $\pi_0\mathcal{F}(U) := \pi_0 (\mathcal{F}(U))$ for every open subset $U$ of $X$.
\end{rem}

We now show how $\hstb_\psi$ is related to a relatively simple presheaf of groupoids to classify.

\begin{prop}\label{p:hstbequalshomogbundles}
Let $\psi:W\to \fg^*$ be a homogeneous unimodular local embedding. Then for $R : \hstb_\psi\to\BG_\R$ the forgetful map of presheaves, $\pi_0 R : \pi_0\hstb_\psi \to \pi_0\BG_\R$ is an isomorphism of presheaves of sets over $\open_\R(W)$.
\end{prop}

\begin{proof}
Note that, for any $\R$-invariant subset $U$ of $W$, since $\BG_{\R} (U)$ is a groupoid, it is enough to show that $R_U : \hstb_\psi(U) \to \BG_{\R}(U)$ is essentially surjective and, for any two
homogeneous symplectic toric bundles $(\pi:P\to U,\omega)$ and $(\pi':P'\to U,\omega')$, $R(P,\omega)$ and $R(P',\omega')$ are isomorphic only if $(P,\omega)$ and $(P',\omega')$ are isomorphic.

The former fact is demonstrated by Proposition \ref{p:homogconn}: for any bundle $\pi:P\to U$ of $\BG_{\R}$, there exists a connection $A$ with respect to which $(P,d\langle\psi\circ \pi,A\rangle)$ is a homogeneous symplectic toric bundle. The latter fact is demonstrated by Proposition \ref{p:hstbnohorclass}: any two homogeneous symplectic toric bundles $(P,\omega)$ and $(P',\omega')$ are isomorphic in $\hstb_\psi(U)$ if and only if the underlying bundles are isomorphic in $\BG_{\R}(U)$. 
\end{proof}

Before we can finish, we need the following well-known theorem.

\begin{thm}\label{t:chernclasses}
Let $M$ be a manifold with corners and let $\BG(M)$ be the category of principal $G$-bundles over $M$ with morphisms isomorphisms of principal $G$-bundles. For $G$ our torus and $\Z_G$ the integral lattice of $\fg$ (that is, the kernel of $\exp : \fg \to G$), the function:
\[\chern:\pi_0\BG(M)\to H^2(M,\Z_G),\]
with $\chern([P])$ the first Chern class of $P$, is a bijection.
\end{thm}

\begin{rem}
By the naturality of characteristic classes, we may extend the bijection in Theorem \ref{t:chernclasses} to an isomorphism of presheaves of sets.  As explained in Remark \ref{r:pizeroofafunctor}, since $\BG : \op{\open(M)}\to \gpoid$ is a presheaf of groupoids, we get a presheaf $\pi_0\BG:\op{\open(M)}\to \set$.  So, for $H^2(\cdot,\Z_G) : \op{\open_{\R}(W)}\to \set$ the presheaf of sets $U\mapsto H^2(U,\Z_G)$, we have the isomorphism of presheaves of sets
\[\chern : \pi_0 \BG\to H^2(\cdot,\Z_G)\]
\end{rem}

Since we’ve been using the site of $\R$-invariant open subsets of $W$ and the presheaf $\BG_{\R}$ of principal $G$-bundles with free and proper $\R$ actions, we still need to show the isomorphism $\chern$ given above descends to $\open_{\R}(W)$.

\begin{prop}\label{p:homogchern}
For $\psi:W\to \fg^*$ a homogeneous unimodular local embedding, the isomorphism of presheaves $\chern:\pi_0\BG\to H^2(\cdot,\Z_G)$ descends to an isomorphism of presheaves $(\chern)_{\R}:\BG_{\R}\to H^2(\cdot,\Z_G)|_{\open_{\R}(W)}$.
\end{prop}

\begin{proof}
Since we are restricting the isomorphism $\chern$, we only need to show that $(\chern)_{\R}$ is still surjective. Fix a class $[\alpha] \in H^2 (U,\Z_G)$. Since the quotient map $q : U \to U/\R$ is a principal $\R$-bundle, it has a global section $i : U/\R \to U$. $i$ and $q$ establish homotopy equivalence between $U$ and $U/\R$.  

Let $\pi : P \to U/\R$ be a principal $G$-bundle with $\chern(P)=\iota^*[\alpha]$. Then $q^*P\cong P\times \R$ is an element of $\BG_{\R}(U)$ (where $P \times \R$ inherits the trivially extended action of $\R$). Hence,
\[\chern(q^*P)=q^*(\chern(P))=q^*(\iota^*([\alpha]))=[\alpha]\]
\end{proof}

Now we may classify homogeneous symplectic toric bundles over a homogeneous unimodular local embedding $\psi:W\to \fg^*$.

\begin{prop}\label{p:homogcherniso}
Let $\psi:W\to \fg^*$ be a homogeneous unimodular local embedding. Then, for $H^2(\cdot,\Z_G) : \op{\open_{\R}(W)}\to \set$ the presheaf of sets $U\mapsto H^2(U,\Z_G)$, there is an isomorphism of presheaves:
\[{\sf chern}:\pi_0\hstb_\psi\to H^2(\cdot,\Z_G) \]
\end{prop}

\begin{proof}
Recall we have isomorphisms of presheaves $\pi_0 R : \pi_0 \hstb_\psi\to \pi_0 \BG_{\R}$ of Proposition \ref{p:hstbequalshomogbundles} and $(\chern)_{\R}:\pi_0\BG_{\R}\to H^2(\cdot,\Z_G)$ from Proposition \ref{p:homogchern}.  Therefore, the composition ${\sf chern}:=(\chern)_{\R}\circ \pi_0 R$ is an isomorphism of presheaves.
\end{proof}

We may now prove our first main classification which we restate for the convenience of the reader.
\\

\noindent {\bf Theorem \ref{t:stcclassify}:}  Let $\psi:W\to \fg^*$ be a homogeneous unimodular local embedding. Then the set of isomorphism classes of symplectic toric cones $(M,\omega,\mu)$ with $G$-quotient $\pi:M\to W$ and orbital moment map $\psi$ is in natural bijective correspondence with the cohomology classes $H^2(W,\Z_G)$, where $\Z_G$ is the integral lattice of $G$, the kernel of the map $\exp:\fg \to G$. 
\\

\begin{proof}
This natural bijective correspondence arises from the composition of isomorphisms of presheaves: $(\pi_0 \hc)^{-1}:\pi_0 \stc_\psi \to \pi_0\hstb_\psi$ (see Theorem \ref{t:hciso}) and ${\sf chern}:\pi_0\hstb_\psi\to H^2(\cdot,\Z_G)$ (see Proposition \ref{p:homogcherniso}).
\end{proof}

We have an easy corollary.

\begin{cor}
Suppose a symplectic toric cone $(M,\omega)$ with orbital moment map $\bar{\mu}:M/G\to \fg^*$ satisfies $H^2(M/G,\Z) = 0$. Then $(M,\omega)$ is $(G \times \R)$-equivariantly symplectomorphic to every other symplectic toric cone admitting quotient space $M/G$ and orbital moment map $\bar{\mu}$.
\end{cor}

Additionally, note that a classification of symplectic toric manifolds automatically descends to a classification for contact toric manifolds.
\\

\noindent {\bf Corollary \ref{c:ctm}:} Let $\psi:W\to \fg^*$ be a homogeneous unimodular local embedding. Then the set of isomorphism classes of contact toric manifolds with symplectization $(M,\omega,\mu)$ having $G$-quotient $\pi:M \to W$ and orbital moment map $\psi$ is in natural bijective correspondence with the cohomology classes $H^2(W,\Z_G)$.
\\

\begin{rem}
Strictly speaking, Lerman in \cite{LermanCTM} proved this result in the case where $\psi$ is known to be the orbital moment map of the symplectization of a contact toric manifold.
However, our result has established exactly what form these orbital moment maps take.
\end{rem}

%%%%%%%%%%%%%%%%%%%%%%%%%%%%%%%%%%%%%%%%%
\part{Classifying symplectic toric stratified spaces with isolated singularities}\label{part:stss}
%%%%%%%%%%%%%%%%%%%%%%%%%%%%%%%%%%%%%%%%%

The goal of this part is to describe and classify symplectic toric stratified spaces with isolated singularities. To begin, we describe in Section \ref{s:singstc} {\em singular symplectic toric cones}: these are symplectic toric cones with an added point at infinity (see Definition \ref{d:singsymptoriccone}). These spaces serve as the local model for symplectic toric stratified spaces with isolated singularities.

In Section \ref{s:stss}, We define and describe {\em symplectic toric stratified spaces with isolated singularities}. These are (roughly) stratified spaces with torus actions locally modeled on singular symplectic toric cones (see Definition \ref{d:symptorstratspace}). We will see that, for such a space $(X,\omega,\mu:X\to \fg^*)$, the topological quotient $X/G$ inherits the structure of a {\em cornered stratified space with isolated singularities}; essentially, a stratified space for which the strata are allowed to be manifolds with corners. Furthermore, the moment map $\mu:X\to \fg^*$ descends to a {\em stratified unimodular local embedding} $\bar{\mu}:X/G\to \fg^*$ (see Definition \ref{d:sule}). This is the
continuous extension of a unimodular local embedding to a cornered stratified space which additionally look like homogeneous unimodular local embeddings near each singularity.

As isomorphic symplectic toric stratified spaces must have isomorphic orbit spaces and intertwined orbital moment maps, we group the spaces $(X,\omega,\mu)$ with fixed $G$-quotient maps $\pi:X\to W$ and orbital moment map the stratified unimodular local embedding $\psi:W\to \fg^*$ together into a groupoid $\stss_\psi(W)$, {\em the groupoid of symplectic toric stratified spaces over $\psi$} (see Definition \ref{d:stss}). The morphisms of this groupoid are $G$-equivariant strata preserving homeomorphisms restricting to symplectomorphisms on the open dense strata which intertwine $G$-quotient maps. As in the case of symplectic toric cones, we may form a presheaf of groupoids $\stss_\psi:\op{\open(W)}\to \gpoid.$

In Section \ref{s:cstb}, we define {\em conical symplectic toric $G$-bundles over $\psi$}. These are principal $G$-bundles $\pi:P\to \reg{W}$ for $\reg{W}$ the open dense stratum of $W$ with $G$-invariant symplectic forms for which $\psi\circ \pi$ is a moment map that satisfy a special ``conical'' condition (see Definition \ref{d:cstb}). Together with $G$-equivariant symplectomorphic bundle isomorphisms, these form a groupoid $\cstb_\psi(W)$, {\em the groupoid of conical symplectic toric $G$-bundles over $\psi$}. As in the case of homogeneous symplectic toric bundles, the collection of groupoids $\cstb_{\psi|_U}(U)$ forms a presheaf of groupoids $\cstb_\psi:\op{\open(W)}\to \gpoid$.

In Section \ref{s:ssc}, we build a map of presheaves $\ssc:\cstb_\psi\to \stss_\psi$ adapted from the equivalence of categories $\cut$ presented by Karshon and Lerman in \cite{KarshonLerman} (see Section \ref{s:KL}). As in the case of $\cut$ and $\hc:\hstb_\psi\to \stss_\psi$ of Part \ref{part:stc}, we are also able to show that $\ssc$ is an isomorphism of presheaves (Theorem \ref{t:sscaniso}).

In Section \ref{s:ccforstss}, we show that the characteristic classes of Karshon and Lerman yield characteristic classes for symplectic toric stratified spaces via the restriction of a symplectic toric stratified space to its dense top stratum. In Proposition \ref{p:conicalboundarycond}, we establish that, for any principal bundle $\pi:P\to \reg{W}$, the Chern class of $P$ establishes a collection of boundary conditions near each singularity that a conical symplectic toric bundle must satisfy.  Via the isomorphism of presheaves $\ssc$, this establishes a natural bijection between the isomorphism of symplectic toric stratified spaces with isolated singularities over a particular stratified unimodular local embedding $\psi:W\to \fg^*$ and a subspace of $H^2(\reg{W},\Z_G\times \R)$ determined by the topology of $W$. In particular, we note that this subspace is in fact the full space of cohomology classes $H^2(\reg{W},\Z_G\times \R)$ except potentially when $\dim(G) = 3$.  In the case where $\dim(G)=3$, we instead may identify the subspace of classes associated to stratified spaces with an extension of $H^2(\reg{W},\Z_G)$ by $H^2(W,\R)$.

We finish Section \ref{s:ccforstss} with two illustrative examples. Example \ref{ex:Lshape} demonstrates a stratified unimodular local embedding for a compact symplectic toric stratified space with exactly one singularity exhibiting two properties not allowed by Burns, Guillemin, and Lerman in their classification of compact connected symplectic toric stratified spaces \cite{BurnsGuilleminLerman}. Example \ref{ex:octahedron} first revisits the example presented by Burns, Guillemin, and Lerman of the rational octahedron in $\R^3$ which corresponds to exactly one compact symplectic toric stratified space. We then discuss two variations on this example demonstrating stratified unimodular local embeddings $\psi:W\to \fg^*$ where the classes of $H^2(\reg{W},\Z_G\times \R)$ are not in bijective correspondence with the isomorphism class of symplectic toric stratified space over $\psi$.

%%%%%%%%%%%%%%%%%%%%%%%%%%%%%%%%%%%%%%%%%
\section{Singular symplectic toric cones}\label{s:singstc} 
%%%%%%%%%%%%%%%%%%%%%%%%%%%%%%%%%%%%%%%%%

Symplectic toric stratified spaces with isolated singularities are symplectic toric manifolds except on a discrete set of isolated singularities fixed by the torus $G$. These singularities have neighborhoods modeled by neighborhoods of $-\infty$ of {\em singular symplectic toric cones:} stratified spaces with exactly one singularity that are conical both in the topological sense as well as the symplectic sense. We will make this precise with a series of definitions.

\begin{defn}
Let $L$ be a manifold (possibly with corners). Then {\sf the open cone on $L$}, denoted $c(L)$, is the topological space $(L \times [-\infty,\infty))/(L \times \{-\infty\})$. Here, $[-\infty,\infty)$ is the topological space given by compactifying $\R$ at one end. The vertex of the cone is denoted by $*$ (i.e., the image of $\{-\infty\}\times L$ in $c(L)$ under the quotient).
\end{defn}

While this convention for a topological cone is a bit awkward, it fits the convention for symplectic cones nicely. What follows is a definition for stratified spaces with isolated singularities.

\begin{defn}
A {\sf stratified space with isolated singularities} is a Hausdorff paracompact topological space $X$ with a partition $X = \reg{X}\coprod_{\alpha\in A}\{x_\alpha\}$ such that $\reg{X}$ is a manifold and, for each $x_\alpha$, there exists a neighborhood $U_\alpha$ of $x_\alpha$ in $X$, a compact manifold $L_\alpha$ (called a {\sf link of $x_\alpha$}), and an open embedding $\varphi_\alpha:U_\alpha\to c(L_\alpha)$ such that
\begin{itemize}
	\item $\varphi_\alpha(x_\alpha)=*$ (i.e., $\varphi_\alpha$ maps $x_\alpha$ to the vertex of $c(L_\alpha)$); and
	\item $\varphi_\alpha$ restricts to a diffeomorphism between $\reg{X}\cap U_\alpha$ and its image in $c(L_\alpha)\backslash\{*\}\cong L_\alpha\times \R$.
\end{itemize}
Formally, this data will be represented by the pair $\left(X,\reg{X}\coprod_{\alpha\in A}\{x_\alpha\}\right)$ though informally the partition may be suppressed. Call a choice of link $L_\alpha$, neighborhood $U_\alpha$, and embedding $\varphi_\alpha:U_\alpha\to c(L_\alpha)$ a {\sf local structure datum for $x_\alpha$}.

A {\sf map of stratified spaces with isolated singularities}
\[f:\left(X,\reg{X}\coprod_{\alpha\in A}\{x_\alpha\}\right) \to \left(X',\reg{X'}\coprod_{\beta\in B}\{x_\beta\}\right)\]
is a continuous map $f:X\to X'$ so that $f(\reg{X})\subset \reg{X'}$, $f|_{\reg{X}}$ is a smooth map between $\reg{X}$ and $\reg{X'}$, and for every $\alpha\in A$, $f(x_\alpha)=x'_\beta$ for some $\beta\in B$. Such a map is an {\sf isomorphism of stratified spaces with isolated singularities} if it is a homeomorphism (it follows that $f^{-1}$ is a map of stratified spaces since $f$ is bijective and a map of stratified spaces).
\end{defn}

To model both symplectic toric stratified spaces and their quotients, we assume that stratified spaces with isolated singularities may be modeled on cones of either manifolds or manifolds with corners; we distinguish the latter with the name {\em cornered stratified spaces}.

\begin{defn}
A {\sf cornered stratified space with isolated singularities} is a stratified space $X$ with isolated singularities for which the top stratum and the links of the singularities are manifolds with corners.
\end{defn}

\begin{rem}
The open dense stratum of a stratified space $X$ will always be denoted $\reg{X}$. Note that, for any (cornered) stratified space $X$, an open subset $U \subset X$ also inherits the structure of a (cornered) stratified space. The open dense part of $U$ with respect to this structure is exactly the intersection $\reg{U}=\reg{X}\cap U$.
\end{rem}

Note that, for any compact manifold $L$, its open cone $c(L)$ is a stratified space with one isolated singularity. If $L$ is a compact manifold with corners, then $c(L)$ is a cornered stratified
space with one isolated singularity. In particular, from any symplectic cone with a compact link, we can build a stratified space by adding a point at $-\infty$. To make this precise, we should determine which open neighborhoods should be neighborhoods of $-\infty$.

\begin{defn} 
Given a symplectic cone $(M,\omega)$, a {\sf neighborhood of $-\infty$} is any open subset $U$ so that $U$ intersects every $\R$-orbit of $M$ and, for each $\lambda \in \R$ with $\lambda\leq 0$, $\lambda\cdot U \subset U$.
\end{defn}

We now define singular symplectic cones. These are topological cones (hence, singular) with a symplectic cone structure on the top stratum.

\begin{defn} \label{d:singsympcone}
A {\sf singular symplectic cone (with corners)} is a (cornered) stratified space with one isolated singularity $X=\reg{X}\coprod\{x_0\}$ with a symplectic form $\omega \in \Omega^2(\reg{X},\R)$ and an action of $\R$ restricting to a smooth and proper action on $\reg{X}$ and fixing $x_0$ such that
\begin{itemize}
	\item $(\reg{X},\omega)$ is a symplectic cone;
	\item $\reg{X}/\R$ is compact; and
	\item every neighborhood $U$ of $x_0$ in $X$ contains a neighborhood of $-\infty$ of $(\reg{X},\omega)$.
\end{itemize}
\end{defn}

From any symplectic cone, we can construct a singular symplectic cone.

\begin{prop}\label{p:extendsctoassc}
Any symplectic cone $(M,\omega)$ with $L = M/\R$ compact extends to a singular symplectic cone.
\end{prop}

\begin{proof}
Define the topological space $\tilde{M}$ as follows: as a set, it is simply the disjoint union $M\coprod \{*\}$, for the point $*$ representing our (soon to be) vertex.  $\tilde{M}$ is then given the topology generated by sets of the form:
\begin{enumerate}
	\item $U$, an open subset of $M$; or
	\item $V\coprod\{*\}$, where $V\subset M$ is a neighborhood of $-\infty$.
\end{enumerate}
More succinctly, we topologize the set $\tilde{M}$ by specifying that all neighborhoods of $-\infty$ in $M$ are in fact deleted open neighborhoods of the singular point $*$. By definition, $(\tilde{M},\omega)$ is a singular symplectic toric cone. 
\end{proof}

We now show that any singular symplectic cone is the topological cone on its underlying contact toric manifold.

\begin{prop}\label{p:trivializationextend}
Let $\left(X=\reg{X}\coprod\{x_0\},\omega\right)$ be a singular symplectic cone. Then, for $\R$-quotient $L := \reg{X}/\R$, every trivialization $\varphi:\reg{X}\to L\times \R$ of $\reg{X}$ as a principal $\R$-bundle admits an extension to a homeomorphism $\tilde{\varphi}:X\to c(L)$.
\end{prop}

\begin{proof}
First, $\varphi$ extends to a bijection $\tilde{\varphi}:X\to c(L)$ taking $x_0$ in $X$ to $*$ in $c(L)$. Notice that, since $L$ is compact, any neighborhood $U$ of $*$ in $c(L)$ contains a neighborhood of the form $L\times (-\infty,\epsilon)\coprod\{*\}$, for some $\epsilon \in \R$.  It follows that $\tilde{\varphi}$ takes open neighborhoods of $x_0$ in $X$ to open neighborhoods of $*$ in $c(L)$. A similar argument shows that $\varphi^{-1}$ extends to an open bijection $\tilde{\varphi}^{-1}$ and so $\tilde{\varphi}$ is a homeomorphism.
\end{proof}

Now, we introduce a toric action.

\begin{defn}\label{d:singsymptoriccone}
A {\sf singular symplectic toric cone} is a singular symplectic cone $(X=\reg{X},\coprod\{x_0\})$ with a continuous action of the torus $G$ and a continuous map $\mu:X\to \fg^*$ such that
\begin{itemize}
	\item $G$ fixes the point $x_0$ and restricts to a smooth action on $\reg{X}$; and
	\item The action of $G$ on $\reg{X}$ makes the symplectic cone $(\reg{X},\omega)$ a symplectic toric cone for which $\mu|_{\reg{X}}$ is the homogeneous moment map of $(\reg{X},\omega)$ (see Definition \ref{d:sympcone}).
\end{itemize}
We represent this data as the triple $(X,\omega,\mu)$.

A {\sf map of singular symplectic toric cones} is a $G$-equivariant map of stratified spaces that restricts to a map of symplectic toric cones on the open strata of the source and target.
\end{defn}

\begin{rem}
For singular symplectic toric cone $(X=\reg{X}\coprod\{x_0\},\omega,\mu:X\to \fg^*)$, as $\mu|_{\reg{X}}$ is homogeneous, it follows from the continuity of $\mu$ that $\mu(x_0)=0$.
\end{rem}

Now, we show that the quotient of a singular symplectic toric cone is a cornered stratified space. This will be a local model for the domain of the orbital moment maps of symplectic toric stratified spaces with isolated singularities.

\begin{lemma}\label{l:sstcquotacss}
Let $(X,\omega,\mu:X\to \fg^*)$ be a singular symplectic toric cone. Then for $B = \reg{X}/\R$, $X/G$ is a cornered stratified space with link $B/G$.
\end{lemma}

\begin{proof}
The quotient $B$ comes with a contact structure $\xi$ and the action of $G$ on $\reg{X}$ descends to a contact toric action on $(B,\xi)$  (see Appendix \ref{app:stcs}).  By Proposition \ref{p:trivializationextend}, any trivialization $\phi:\reg{X}\to B\times \R$ extends to a homeomorphism $\tilde{\varphi}:X\to c(B)$.  Furthermore, by Proposition \ref{p:equivtrivialization}, we may choose $\varphi$ to be a
$G$-equivariant trivialization.

As the actions of $G$ and $\R$ commute and $\tilde{\varphi}$ is $G$-equivariant, $\tilde{\varphi}$ descends to a homeomorphism $\tilde{\varphi}:X/G\to c(B/G)$. By Lemma \ref{l:contquotamwc}, $B/G$ is a manifold with corners and therefore $\tilde{\varphi}$ gives local structure data for the singularity $x_0$ of $X$ as a cornered stratified space.
\end{proof}

As in the case of symplectic cones, symplectic toric cones admit trivialization independent extensions to singular symplectic toric cones.

\begin{prop}\label{p:extendtoasstc}
Every symplectic toric cone $(M,\omega,\mu:M\to \fg^*)$ over compact base $L = M/\R$ extends to a singular symplectic toric cone.
\end{prop}

\begin{proof} First, Proposition \ref{p:extendsctoassc} tells us how to transform $(M,\omega)$ into a singular symplectic cone $(\tilde{M},\omega)$.  The toric action descends to a contact toric $G$ action on $L$. With respect to this $G$ action, we can pick a $G$-equivariant trivialization of $M$ as a principal $\R$-bundle $\varphi:M\to L\times \R$ (see Proposition \ref{p:equivtrivialization}) which extends to a homeomorphism $\tilde{\varphi}:\tilde{M}\to c(L)$.

As each set of the form $L \times (-\infty,\epsilon)$ is $G$-invariant, it follows that every neighborhood of $-\infty$ for $M$ contains a $G$-invariant neighborhood of $-\infty$. Thus, the action of $G$ on $M$ extends to a continuous action $\rho:G\times \tilde{M}\to \tilde{M}$ on $\tilde{M}$ fixing the singular point. Then, since for $V$ any $G$-invariant neighborhood of $-\infty$ we have that $\rho^{-1}(V \coprod\{*\}) = G \times (V\coprod\{*\})$, it follows from the observation above that $\rho$ is continuous.

Finally, note that since $\mu$ is smooth and homogeneous, it follows we can continuously extend $\mu$ to $\tilde{\mu}:\tilde{M}\to \fg^*$ by defining $\tilde{\mu}(*)=0$.
\end{proof}

To finish this section, we prove that any isomorphism of symplectic toric cones extends to an isomorphism between their extensions as singular symplectic toric cones. 

\begin{lemma}\label{l:morphismsextend}
Let $(X,\omega,\mu:X\to \fg^*)$ and $(X',\omega',\mu':X'\to \fg^*)$ be two singular symplectic toric cones for which there is an isomorphism of symplectic toric cones
\[f:(\reg{X},\omega,\mu|_{\reg{X}})\to (\reg{X'},\omega',\mu'|_{\reg{X'}})\]
Then $f$ extends to an isomorphism of singular symplectic toric cones.
\end{lemma}

\begin{proof}
Since $f$ is $(G\times \R)$-equivariant, $f$ takes $G$-invariant neighborhoods of $-\infty$ in $(\reg{X},\omega)$ to $G$-invariant neighborhoods of $-\infty$ in $(\reg{X'},\omega')$. $f^{-1}$ satisfies the same property and so we may conclude $f$ and $f^{-1}$ extend to isomorphisms of singular symplectic toric cones.
\end{proof}

%%%%%%%%%%%%%%%%%%%%%%%%%%%%%%%%%%%%%%%%%
\section{Symplectic toric stratified spaces with isolated singularities}\label{s:stss}
%%%%%%%%%%%%%%%%%%%%%%%%%%%%%%%%%%%%%%%%%

After all the work of the previous section, we are finally ready to give a definition of symplectic toric stratified spaces with isolated singularities.

\begin{defn}\label{d:symptorstratspace}
A {\sf symplectic toric stratified space with isolated singularities} is a stratified space with isolated singularities $(X,\reg{X}\coprod_{\alpha\in A} \{x_\alpha\})$ with a symplectic form $\omega\in \Omega^2(\reg{X})$, a continuous map $\mu:X\to \fg^*$, and a continuous action of torus $G$ on $X$ fixing each $x_\alpha$ and restricting to a smooth, toric action on $(\reg{X},\omega)$ with moment map $\mu|_{\reg{X}}:\reg{X}\to \fg^*$. Furthermore, for each $x_\alpha$, there exists a $G$-invariant neighborhood $U$ of $x_\alpha$ in $X$, a singular symplectic toric cone $(C,\omega,\nu:C\to \fg^*)$, a $G$-invariant neighborhood $V$ of the vertex of $C$, and a $G$-equivariant homeomorphism $\varphi:U\to V$ such that
\begin{itemize}
	\item $\varphi(x_\alpha)=*$ (for $*$ the vertex of $C$);
	\item  $\varphi$ restricts to a symplectomorphism between $\reg{U}$ and $\reg{V}$; and
	\item  $\mu|_U = \nu\circ \varphi + \mu(x_\alpha)$.
\end{itemize}
These objects are represented as triples $(X,\omega,\mu)$.
\end{defn}

As explained in Proposition \ref{p:extendtoasstc}, any symplectic toric cone with homogeneous moment map can be extended to a singular symplectic cone; these serve for now as our only example of a symplectic toric stratified space with isolated singularities. More exotic examples are given at the end of Section \ref{s:ccforstss}.

As with symplectic toric manifolds in \cite{KarshonLerman}, symplectic toric stratified spaces with a fixed quotient map are grouped together by their orbital moment map. To make sense of this, it is important to first understand what form their quotients and orbital moment maps take.

\begin{defn}\label{d:sule}
 Let $(W,\reg{W}\coprod_{\alpha\in A}\{w_\alpha\})$ be a cornered stratified space with isolated singularities. A continuous map $\psi:W\to \fg^*$ is called a {\sf stratified unimodular local embedding} if
 \begin{itemize}
 	\item $\psi|_{\reg{W}}$ is a unimodular local embedding; and
 	\item for each $w_\alpha$, there exists a local structure datum $\varphi_\alpha:U_\alpha\to c(L_\alpha)$ and a homogeneous unimodular local embedding (see Definition \ref{d:hule})
 		 $\phi_\alpha :L_\alpha \times \R \to \fg^*$ such that $\psi|_{\reg{(U_\alpha)}} = \phi_\alpha\circ\varphi_\alpha+\psi(w_\alpha)$, where $\phi_\alpha$ is homogeneous with respect to the action of $\R$ by translation on $L_\alpha\times \R$.
 \end{itemize}
We will call the piece of local structure datum $\varphi_\alpha:U_\alpha\to c(L_\alpha)$ as above a {\sf homogeneous local structure datum}. The manifold with corners $L_\alpha$ will be known as {\sf the link of $w_\alpha$}.
\end{defn}

\begin{prop}\label{p:stssommasule}
 Suppose that $(X,\omega,\mu:X\to \fg^*)$ is a symplectic toric stratified space with isolated singularities. Then $X/G$ is a cornered stratified space with isolated singularities and, for quotient map $\pi:X\to X/G$, the unique map $\bar{\mu}:X/G\to \fg^*$ satisfying $\bar{\mu}\circ \pi=\mu$ is a stratified unimodular local embedding.
\end{prop}

\begin{proof}
Suppose that $X=\reg{X}\coprod_{\alpha\in A}\{x_\alpha\}$.  As $(\reg{X},\omega,\mu|_{\reg{X}})$ is a symplectic toric manifold, $\reg{X}/G$ is a manifold with corners
and $\bar{\mu}|_{\reg{X}/G}$ is a unimodular local embedding (see Proposition \ref{p:KLorbmmaule}). Additionally, by definition, each singular point $x_\alpha$ comes with a $G$-invariant neighborhood $U_\alpha$ and a $G$-equivariant embedding $\varphi_\alpha:U_\alpha \to C_\alpha$ of $U_\alpha$ onto a neighborhood of the vertex of a singular symplectic
toric cone $(C_\alpha,\omega_\alpha,\nu_\alpha:C_\alpha\to \fg^*)$ so that $\mu|_{U_\alpha}=\nu_\alpha\circ \varphi_\alpha+\mu(x_\alpha)$. 

In turn, by Lemma \ref{l:sstcquotacss}, the quotient $C_\alpha/G$ of this singular symplectic toric cone is a cornered stratified space and the associated orbital moment map $\bar{\nu}_\alpha:C_\alpha/G\to \fg^*$ is a homogeneous unimodular local embedding.  $\varphi_\alpha$ descends to a homeomorphism $\bar{\varphi}_\alpha:U_\alpha/G\to C_\alpha/G$ and the condition $\mu|_{U_\alpha}=\nu_\alpha\circ \varphi_\alpha+\mu(x_\alpha)$ descends to the required condition relating $\bar{\varphi}_\alpha$ and $\bar{\nu}_\alpha$.
\end{proof}

We can now define our main category of interest. 

\begin{defn}\label{d:stss}
Let $\psi:W\to \fg^*$ be a stratified unimodular local embedding. A {\sf symplectic toric stratified space over $\psi$} is a symplectic toric stratified space $(X,\omega,\mu:X\to \fg^*)$ with a $G$-quotient map $\pi:X\to W$ so that $\psi\circ \pi = \mu$.

A {\sf map of symplectic toric stratified spaces over $\psi$} between $(X,\omega,\pi:X\to W)$ and $(X',\omega',\pi':X'\to W)$ is a $G$-equivariant isomorphism of stratified spaces $\varphi:X\to X'$ that restricts to a symplectomorphism between $(\reg{X},\omega)$ and $(\reg{X'},\omega')$ and satisfies $\pi'\circ\varphi = \pi$.

{\sf The groupoid of symplectic toric stratified spaces over $\psi$} $\stss_\psi(W)$ is the groupoid with objects and morphisms as described above.
\end{defn}

\begin{rem}\label{r:stssapsheaf}
Note that for any open subset $U\subset W$, $\psi|_U$ is also a stratified unimodular local embedding. Therefore, it makes sense to define the presheaf
\[\stss_\psi:\op{\open(W)}\to \gpoid,\;\;\;U\mapsto \stss_\psi(U):=\stss_{\psi|_U}(U)\]
Implicit here is the fact that, for $U$ not containing singularities, the condition that an object of $\stss_\psi(U)$ must be modeled on certain neighborhoods of singular symplectic toric
cones is empty. Hence, here $\stss_\psi(U)$ and $\stm_{\psi|_{\reg{W}}}(U)$ (see Definition \ref{d:stm}) agree.
\end{rem}

To finish this section, we describe how symplectic toric stratified spaces may be pulled back over certain open embeddings of cornered stratified spaces.

\begin{rem}\label{r:stsspullback}
Suppose that $\psi:W\to \fg^*$ and $\psi':W'\to \fg^*$ are two stratified unimodular local embeddings and let $\varphi:W'\to W$ be an open embedding of cornered stratified spaces
with $\psi\circ \varphi=\psi'$. Then, for any $(X,\omega,\pi:X\to W)\in \stss_\psi(W)$, we may pullback $(X,\omega,\pi:X\to W)$ via $\varphi$ to the symplectic toric stratified space $\varphi^*(X,\omega,\pi:X\to W)$ over $\psi'$ (this is simply the restriction of $(X,\omega,\pi:X\to W)$ to $\varphi(W')$ with $G$-quotient map $\varphi^{-1}\circ \pi$).

We can also pullback isomorphisms: any isomorphism
\[f:(X,\omega,\pi:X\to W)\to (X',\omega',\pi':X'\to W)\]
of $\stss_\psi(W)$ induces a unique isomorphism
\[\varphi^*f:\varphi^*(X,\omega,\pi:X\to W)\to \varphi^*(X',\omega',\pi':X'\to W)\]
of $\stss_{\psi'}(W')$.
\end{rem}

%%%%%%%%%%%%%%%%%%%%%%%%%%%%%%%%%%%%%%%%%
\section{Conical symplectic toric bundles}\label{s:cstb}
%%%%%%%%%%%%%%%%%%%%%%%%%%%%%%%%%%%%%%%%%

As in \cite{KarshonLerman} and in Part \ref{part:stc}, we define a new groupoid of principal $G$-bundles with special properties. The bundles in this section, conical symplectic toric bundles, are symplectic toric bundles with an extra condition to satisfy over deleted open neighborhoods of the singularity of the base cornered stratified space.

\begin{defn}\label{d:cstb}
 Let $(W,\reg{W}\coprod_{\alpha\in A}w_\alpha)$  be a cornered stratified space with isolated singularities and let $\psi:W\to \fg^*$ be a stratified unimodular local embedding. Then a {\sf conical symplectic toric principal $G$-bundle over $\psi$} is a symplectic toric bundle $(\pi:P\to \reg{W},\omega)$ over $\psi|_{\reg{W}}$ (see Definition \ref{d:stb}) satisfying the following local condition: for each singularity $w_\alpha$ of $W$, there exists
\begin{itemize}
	\item a homogeneous local trivialization datum $\varphi:U\to c(L)$ of $w_\alpha$ with associated homogeneous unimodular local embedding $\phi:L\times \R\to \fg^*$ (see Definition \ref{d:sule});
	\item a homogeneous symplectic toric bundle $(\varpi: Q \to L\times \R,\eta)$ over $\phi$ (see Definition \ref{d:hstb}); and
	\item for $V := \varphi(U)$, there is a $G$-equivariant symplectomorphism $\tilde{\varphi}: (P|_{\reg{U}},\omega) \to (Q|_{\reg{V}},\eta)$ so that the diagram
	\[\xymatrix{P|_{\reg{U}}\ar[r]^{\tilde{\varphi}} \ar[d]_\pi & Q|_{\reg{V}} \ar[d]^{\varpi} \\ U \ar[r]_\varphi & c(L)}\]
	commutes.
\end{itemize}
{\sf The category $\cstb_\psi(W)$ of conical symplectic toric bundles over $\psi$} is the groupoid with objects as described above and morphisms maps of symplectic toric bundles: bundle isomorphisms which are also symplectomorphisms.
\end{defn}

We may more efficiently characterize conical symplectic toric bundles using pullbacks.

\begin{lemma}\label{l:pullbackdefncstb}
Let $\psi:W\to \fg^*$ be a stratified unimodular local embedding. Then a symplectic toric bundle $(\pi:P\to \reg{W}, \omega)$ is a conical symplectic toric bundle exactly when, for each singularity $w_\alpha$ of $W$, there exists
\begin{itemize}
	\item an open neighborhood $U$ of $w_\alpha$;
	\item a homogeneous local trivialization datum $\varphi:U\to c(L)$ with homogeneous unimodular local embedding $\phi:L\times \R\to \fg^*$ satisfying $\psi|_{\reg{U}} = \phi\circ \varphi+\psi(w_\alpha)$; and 
	\item a homogeneous symplectic toric bundle $(\varpi : Q \to L \times \R,\eta) \in \hstb_\phi(L\times \R)$
\end{itemize}
so that, thinking of $(Q,\eta)$ as a symplectic toric bundle over $\varphi+\psi(w_\alpha)$ (see Remark \ref{r:shiftedmomentmap}), $\varphi^*(Q,\eta)$ and $(P,\omega)|_{U}$ are isomorphic in $\stb_{\psi|_{\reg{U}}}(\reg{U})$.
\end{lemma}

\begin{proof}
This is easily confirmed using Definition \ref{d:cstb}. 
\end{proof}

\begin{rem}\label{r:cstbapsheaf}
As in the case of $\stss_\psi$ (see Remark \ref{r:stssapsheaf}), we have a presheaf of groupoids $\cstb_\psi:\op{\open(W)} \to \gpoid$. Open subsets $U$ of $W$ not containing singularities renders
the extra conditions of Definition \ref{d:cstb} empty and here $\cstb_\psi(U)$ and $\stb_{\psi|_{\reg{W}}}(U)$ (the presheaf of symplectic toric principal $G$-bundles, see Definition \ref{d:stb}) agree.

In fact, $\cstb_\psi(W)$ is a stack over $W$. As the proof of this is more or less just a retelling of the proof that the presheaf of principal bundles over a topological space is a stack, we relegate this proof to the appendix (see Proposition \ref{p:cstbastack}).
\end{rem}

As in the case of symplectic toric stratified spaces (see Remark \ref{r:stsspullback}), conical symplectic toric bundles may be pulled back in special cases.

\begin{rem}\label{r:cstbpullback}
Let $\psi:W\to \fg^*$ and $\psi':W'\to \fg^*$ be two stratified unimodular local embeddings and let $\varphi:W'\to W$ be an open embedding of cornered stratified spaces with
$\psi'\circ \varphi=\psi$. Then, for $(\pi:P\to \reg{W},\omega)\in \cstb_\psi(W)$, note that $\varphi^*P$ is a conical symplectic toric bundle over $\psi'$ (with symplectic form pulled back from $P$).
\end{rem}

As the forgetful functor from conical symplectic toric bundles to symplectic toric bundles is important, we give it a symbol and reference here:

\begin{defn}\label{d:forgetfulfunctor}
Given a stratified unimodular local embedding $\psi:W\to \fg^*$, let $\iota:\cstb_\psi(W)\to \stb_{\psi|_{\reg{W}}}(\reg{W})$ denote the forgetful functor.
\end{defn}

\begin{lemma}\label{l:forgetfulfunctorff}
For any stratified unimodular local embedding $\psi:W\to \fg^*$, $\iota:\cstb_\psi(W)\to \stb_{\psi|_{\reg{W}}}(\reg{W})$ is fully faithful.
\end{lemma}

\begin{proof}
Since maps of conical symplectic toric bundles are maps of symplectic toric bundles without any added conditions, this follows by definition.
\end{proof}

It is not clear that $\cstb_\psi(W)$ is non-empty for certain choices of $W$ or $\psi$. Unlike the case of symplectic toric bundles, we will soon see that the existence of a $G$-invariant symplectic form $\omega$ on a given principal $G$-bundle $\pi:P\to \reg{W}$ with respect to which $(\pi:P\to \reg{W},\omega)$ is a conical symplectic toric bundle may depend on the topologies of $P$ and $W$.

%%%%%%%%%%%%%%%%%%%%%%%%%%%%%%%%%%%%%%%%%
\section{The morphism of presheaves $\ssc:\cstb_\psi \to \stss_\psi$}\label{s:ssc}
%%%%%%%%%%%%%%%%%%%%%%%%%%%%%%%%%%%%%%%%%

In this section, we will define the functor $\ssc:\cstb_\psi(W)\to \stss_\psi(W)$ for any stratified unimodular local embedding $\psi:W\to \fg^*$. This is done using the following steps.\\
\begin{enumerate}[label={\bf Step \arabic*:}]
\item We first define the functor $\sctop$ taking conical symplectic toric bundles $(\pi:P\to \reg{W},\omega)$ in $\cstb_\psi(W)$ to pairs $\sctop(\pi:P\to \reg{W},\omega)=(\bar{P},\bar{\pi}:\bar{P}\to W)$ of topological $G$-spaces with quotient maps to $W$ and maps of conical symplectic toric bundles to maps of topological $G$-spaces over $W$: $G$-equivariant homeomorphisms $f : (X,\varpi : X \to W) \to (X',\varpi' : X' \to W)$ with $\varpi = \varpi'\circ f$.

\item We then define $\ssc(\pi:P\to \reg{W},\omega)$ as the tuple
\[(X=\reg{X}\coprod_{\alpha\in A}\{x_\alpha\},\omega,\bar{\pi}:X\to W)\]
where $(X,\bar{\pi}) = \sctop(P,\omega)$ and where $X=\reg{X} \coprod_{\alpha\in A}\{x_{\alpha}\}$ is a partition of $X$ such that $(\reg{X},\omega,\bar{\pi}|_{\reg{X}}:\reg{X}\to \reg{W})$ is a symplectic toric manifold over $\psi|_{\reg{W}}$.  We may also show that, for any morphism of conical symplectic toric principal bundles $\varphi$, the morphism $\ssc(\varphi):=\sctop(\varphi)$ restricts to a symplectomorphism on the open dense pieces of the source and target partitioned symplectic toric spaces.  This therefore establishes the functor $\ssc:\cstb_\psi(W)\to \psts_\psi(W)$ to {\em the category of partitioned symplectic toric spaces over $\psi$} (see Definition \ref{d:psts}).

\item We next show $\ssc$ commutes with pullbacks: for any open embedding $\varphi:W'\to W$ and stratified unimodular local embeddings $\psi:W\to \fg^*$ and $\psi':W'\to \fg^*$ for which $\psi\circ \varphi=\psi'$, given any conical symplectic toric bundle $(P,\omega)$, $\ssc(\varphi^*(P,\omega))=\varphi^*(\ssc(P,\omega))$ (see Remarks \ref{r:stsspullback}, \ref{r:cstbpullback}, and \ref{r:pstspullbacks}). In particular, $\ssc$ commutes with restrictions.

\item We show for a particular model case of $\psi$ and $(P,\omega)$ that $\ssc(P,\omega)$ is a symplectic toric cone.

\item With a small amount of work, the model case of Step 4 allows us to show that every $\ssc(P,\omega)$ is a symplectic toric stratified space over $\psi$.  Since $\stss_\psi(W)$ is a full subcategory of $\psts_\psi(W)$, we may thereby conclude that $\ssc$ restricts to a functor $\ssc:\cstb_\psi(W)\to \stss_\psi(W)$.\\
\end{enumerate}

We now flesh out the details of each step.\\

\noindent{\bf Step 1:} To start, fix a stratified unimodular local embedding $\psi:W\to \fg^*$ and fix $(\pi:P\to \reg{W},\omega)$ a conical symplectic toric bundle over $\psi$. Suppose $W$ has partition $W = \reg{W} \coprod_{\alpha\in A}\{w_\alpha\}$.  First we define $\sctop(P,\omega)$, a topological $G$-space with a $G$-quotient to $W$.

We construct the cornered stratified space $\tilde{P}$ as follows: as a set, $\tilde{P}:= P\coprod_{\alpha\in A}\{p_\alpha\}$ for a set $\{p_\alpha\}_{\alpha \in A}$ in bijection with the singularities of $W$.  We give $\tilde{P}$ the topology generated by
\begin{enumerate}
	\item open subsets of $P$; and
	\item sets $\{p_\alpha\}\coprod P|_{\reg{U}}$ for $U$ a neighborhood of $w_\alpha$ in $W$.
\end{enumerate}
Note that $\tilde{P}$ inherits a $G$ action extended from the free action of $G$ on $P$; specifically, one where each $p_\alpha$ is fixed. This is a continuous extension as every open neighborhood of a $p_\alpha$ by design contains a $G$-invariant deleted neighborhood of $p_\alpha$. Furthermore, note that the $G$-quotient $\pi:P\to \reg{W}$ extends to a $G$-quotient $\tilde{\pi}:\tilde{P}\to W$ with $\tilde{\pi}(p_\alpha)=w_\alpha$ for every $\alpha$.

Now, for each $w\in \reg{W}$, let $K_w \leq  G$ be the subtorus determined by the image of the unimodular local embedding $\psi|_{\reg{W}}$ (see Section \ref{s:KL}). Let $\sim$ be the equivalence relation on $\tilde{P}$ with each singularity $p_\alpha$ occupying its own equivalence class and defined for elements of $P\subset \tilde{P}$ by
\[p\sim p' \text{ when there exists }k\in K_{\pi(p)}\text{ such that }p\cdot k=p'.\]

Since $\sim$ only identifies elements of the same $G$-orbit, the $G$ action $g\cdot[p] := [g\cdot p]$ on $\tilde{P}/\sim$ is well-defined and continuous. Furthermore, the $G$-quotient map $\tilde{\pi}$ descends to a $G$-quotient map $\bar{\pi}:\tilde{P}/\sim\,\to W$ for this action. Define $\sctop(P,\omega):=(\tilde{P},\bar{\pi}:\tilde{P}\to W)$.

Now, to define $\sctop$ on morphisms, let $\varphi:(P,\omega)\to (P',\omega')$ be a map of conical symplectic toric bundles. Again, due to how the topology of $\tilde{P}$ was defined, $\varphi$ extends to a $G$-equivariant homeomorphism $\tilde{\varphi}$ intertwining the two $G$-quotient maps $\tilde{\pi}$ and $\tilde{\pi}'$ for $\tilde{P}$ and $\tilde{P}'$, respectively. Therefore, $\tilde{\varphi}:\tilde{P}\to \tilde{P}'$ descends to an isomorphism of topological $G$-spaces over $W$:
\begin{equation}\label{eq:sctoponmorphisms}\sctop(\varphi):\tilde{P}/\sim\,\to \tilde{P}'/\sim,\;\;\;\sctop(\varphi)([p]):=[\tilde{\varphi}(p)]\end{equation}
such that $\bar{\pi}'\circ \ctop(\varphi) = \bar{\pi}$.\\

\noindent{\bf Step 2:} In this step, we define $\ssc$ first as a functor to the category of {\em partitioned symplectic toric spaces over $\psi$}.

\begin{defn}\label{d:psts}
For $\psi:W\to \fg^*$ a stratified unimodular local embedding, a {\sf partitioned symplectic toric stratified space over $\psi$} is a triple $(X=\reg{X}\coprod_{\alpha \in A}\{x_\alpha\},\omega,\pi:X\to W)$ with a continuous action of $G$ on $X$ fixing each $x_\alpha$ and restricting to a smooth action on $\reg{X}$ such that:
\begin{itemize}
	\item the action of $G$ on $X$ has $G$-quotient map $\pi:X\to W$; and
	\item $(\reg{X},\omega,\pi|_{\reg{X}})$ is a symplectic toric manifold over $\psi|_{\reg{W}}$ with respect to the restricted action of $G$ to $\reg{X}$.
\end{itemize}

A {\sf map of partitioned symplectic toric stratified spaces over $\psi$} from $(X=\reg{X}\coprod_{\alpha\in A}\{x_\alpha\},\omega,\pi)$ to  $(Y=\reg{Y}\coprod_{\beta\in B}\{y_\beta\},\omega',\varpi)$ is a homeomorphism $\varphi:X\to Y$ with $\varpi\circ\varphi=\pi$ restricting to a map of symplectic toric manifolds over $\psi|_{\reg{W}}$ from $(\reg{X},\omega,\pi|_{\reg{W}})$ to $(\reg{Y},\omega,\varpi|_{\reg{Y}})$.

Denote by $\psts_\psi(W)$ {\sf the category of partitioned symplectic toric spaces over $\psi$}.
\end{defn}

\begin{rem}\label{r:stssfullsubcat}
Notice that symplectic toric stratified spaces over $\psi$ are partitioned symplectic toric spaces over $\psi$ with an extra conical condition near each singularity (again, here the ``conical'' condition is both a topological and a symplectic condition).  Indeed, the category $\stss_\psi(W)$ of symplectic toric stratified spaces over $\psi$ is a full subcategory of the category $\psts_\psi(W)$. 
\end{rem}

\begin{rem}\label{r:pstspullbacks}
Notice that, as in the case of symplectic toric stratified spaces over $\psi$ and conical symplectic toric spaces over $\psi$, $\psts_\psi$ also defines a presheaf over $W$, with $\psts_\psi(U):=\psts_{\psi|_U}(U)$ and $(X,\omega,\pi)|_U := (\pi^{-1}(U),\omega|_{\reg{(\pi^{-1}(U))}},\pi|_{\pi^{-1}(U)})$ for $U$ any open subset of $W$.  In thinking of $\stss_\psi(W)$ as a full subcategory of $\psts_\psi(W)$, it is not difficult to confirm that restriction in $\psts_\psi(W)$ matches restriction in $\stss_\psi(W)$ and therefore $\stss_\psi$ is in fact a full subpresheaf of $\psts_\psi$.

More generally, for $\psi':W'\to \fg^*$ and $\psi:W\to \fg^*$ and $\varphi:W'\to W$ an open embedding of cornered stratified spaces, we may pullback an element $(X,\omega,\pi)$ of $\psts_\psi(W)$ to an element of $\psts_{\psi'}(W')$, replacing the $G$-quotient map of the restriction $(X,\omega,\pi)|_{\varphi(W')}$ with the $G$-quotient to $W'$ given by $\varphi^{-1}|_{\varphi(W')}\circ \pi$.
Similarly, we can pullback any morphism of $\psts_\psi(W)$ to a morphism of $\psts_{\psi'}(W')$.  Again, note that the pullback operation defined for partitioned symplectic toric spaces matches the pullback operation for symplectic toric stratified spaces (as defined in Remark \ref{r:stsspullback}).
\end{rem}

Let $\psi:W\to \fg^*$ be a stratified unimodular local embedding and let $(\pi:P\to \reg{W},\omega)$ be a conical symplectic toric bundle. Let $\tilde{P}$ and $\sim$ be the extension of $P$ and equivalence relation as described in the previous step. Note that the subset $P\subset \tilde{P}$ descends to the open dense subset $P/\sim \,\subset \tilde{P}/\sim$. Furthermore, $P/\sim$ is exactly the topological space $\ctop(\iota(P,\omega))$ (recall $\iota:\cstb_\psi(W)\to \stb_{\psi|_{\reg{W}}}(\reg{W})$ is the forgetful functor; see Remark \ref{r:stssapsheaf} and Definition \ref{d:forgetfulfunctor}).

So $\tilde{P}/\sim$ has partition $\ctop(\iota(P,\omega))\coprod_{\alpha\in A}\{[p_\alpha]\}$. Recall that, by definition, $\cut(\iota(P,\omega))$ is a symplectic toric manifold over $\psi|_{\reg{W}}$ homeomorphic to the topological space $\ctop(P,\omega)$.

Write $\cut(\iota(P,\omega))= (P/\sim,\bar{\omega}, \varpi:P/\sim\,\to \reg{W})$. Note that the $G$-quotient map $\varpi$ is just the $G$-quotient map for the topological space $\ctop(\iota(P,\omega))\subset \tilde{P}/\sim$ and so must match the restriction of $\bar{\pi}:\tilde{P}/\sim\,\to W$.

We now define $\ssc$ as the partitioned symplectic toric space over $\psi$
\begin{equation}\label{eq:sconobjects}\ssc(\pi:P\to \reg{W},\omega):=(\tilde{P}/\sim,P/\sim \coprod_{\alpha\in A}\{[p_\alpha]\}, \bar{\omega},\bar{\pi}:\tilde{P}/\sim\,\to W)\end{equation}

Now, let $\varphi:(P,\omega)\to (P',\omega')$ be a map of conical symplectic toric bundles over $\psi$. Then since our definition for $\sctop(\varphi)$ matches $\ctop(\iota(\varphi))$ (again, see Section \ref{s:KL}), it follows that the morphism $\sctop(\varphi)$ restricts to an isomorphism of symplectic toric manifolds between the open dense pieces of $\ssc(P,\omega)$ and $\ssc(P',\omega')$.

Now, we define $\ssc$.

\begin{defn}
Let $\psi:W\to \fg^*$ be a stratified unimodular local embedding. Then $\ssc:\cstb_\psi(W)\to \psts_\psi(W)$ is the functor such that $\ssc(P,\omega)$ is the partitioned symplectic toric space over $\psi$ given in equation \eqref{eq:sconobjects} and $\ssc(\varphi):=\sctop(\varphi)$ is the morphism given in equation \eqref{eq:sctoponmorphisms}.
\end{defn}

Hence, since $\stss_\psi(W)$ is a full subcategory of $\psts_\psi(W)$, the remainder of the steps towards defining $\ssc$ involve checking that $\ssc$ is natural with respect to restriction and that the image of $\ssc:\cstb_\psi(W)\to \psts_\psi(W)$ is contained in $\stss_\psi(W)\subset \psts_\psi(W)$.\\

\noindent{\bf Step 3:} The content of this step is the following lemma regarding $\ssc$ commuting with restrictions and, more generally, with pullbacks (as defined in Remarks \ref{r:stsspullback} and \ref{r:pstspullbacks}).

\begin{lemma}\label{l:ssccommuteswithpullbacks}
Suppose $\varphi:W'\to W$ is an open embedding of cornered stratified spaces and $\psi:W\to \fg^*$, $\psi':W'\to \fg^*$ are stratified unimodular local embeddings with $\psi\circ\varphi = \psi'$.  Then for any conical symplectic toric bundle $(\pi:P\to \reg{W},\omega)$, $\ssc(\varphi^*(P,\omega))=\varphi^*\ssc(P,\omega)$.  In particular, for any open subset of $U$ in $W$, $\ssc((P,\omega)|_U) = \ssc(P,\omega)|_U$.
\end{lemma}

\begin{proof}
Fix open $U$ in $W$. Notice that $(P,\omega)|_U$ is a saturated open set with respect to the quotient map associated to the equivalence relation $\sim$ defining $\sctop$.  Therefore, $\ssc((P,\omega)|_U)$ is an open subset of $\ssc(P,\omega)$. Since the $G$-quotient map of $\ssc(P,\omega)$ to $W$ is inherited from that of $P$, it follows that the open subset $\ssc((P,\omega)|_U)$ is exactly the restriction $\ssc(P,\omega)|_U$.

More generally, since pullbacks of conical symplectic toric bundles and partitioned symplectic toric spaces by open embeddings $\varphi:W'\to W$ are restrictions followed by alteration of the quotient map, it follows from the above that $\ssc(\varphi^*(P,\omega))=\varphi^*\ssc(P,\omega)$.
\end{proof}

\noindent {\bf Step 4:} We now prove that $\ssc(P,\omega)$ is a symplectic toric stratified space for the local conical model of neighborhoods of singular points in $W$.

\begin{prop}
Let $B$ be a a manifold with corners and $\psi:c(B)\to \fg^*$ be a stratified unimodular local embedding for which $\psi|_{B\times \R}$ is a homogeneous unimodular local embedding
with respect to the action of translation on the second factor (see Definition \ref{d:hule}). Suppose that $(\pi:P\to B\times \R,\omega)$ is a homogeneous symplectic toric bundle over $\psi|_{B\times \R}$. Then $(P,\omega)$ is a conical symplectic toric bundle over $\psi$ and $\ssc(P,\omega)$ is a singular symplectic toric cone over $\psi$.
\end{prop}

\begin{proof}
First, by definition, any homogeneous symplectic toric bundle over $\psi|_{B\times \R}$ is a conical symplectic toric bundle over $\psi$ (indeed, $(P,\omega)$ itself satisfies the required local condition for a conical symplectic toric bundle; see again Definition \ref{d:cstb}).

Let $\tilde{P}$ and $\tilde{\pi}:\tilde{P}\to W$ correspond to the same extension of $P$ and $\pi$ as defined in Step 1 in constructing $\sctop$. Then $(\tilde{P},\omega,\psi\circ \tilde{\pi})$ is exactly the extension of the symplectic toric cone $(P,\omega,\psi\circ\pi)$ to a singular symplectic toric cone as defined in Proposition \ref{p:extendtoasstc}.

Since $P\subset \tilde{P}$ is a homogeneous symplectic toric bundle over $\psi|_{B\times \R}$, we may apply $\hc$ to $(P,\omega)$. Recall that $\hc:\stb_{\psi|_{B\times \R}}(B\times \R)\to \stc_{\psi|_{B\times \R}}(B\times \R)$ is just the functor $\cut$ that remembers the action of $\R$ on a homogeneous symplectic toric bundle (see Definition \ref{d:hc}).  So, using $\hc$, $\cut(P,\omega)$ inherits the structure of a symplectic toric cone over $\psi|_{B\times \R}$.

To finish, we must show that $\ctop$ is homeomorphic to a singular cone. let $L := P/\R$ and let $\varphi:P\to L\times \R$ be a $G$-equivariant trivialization of $P$ as a principal $\R$-bundle (by Proposition \ref{p:equivtrivialization}, such a trivialization exists). Then, by Proposition \ref{p:trivializationextend}, this extends to a homeomorphism $\tilde{\varphi}:\tilde{P}\to c(L)$. Note that we may apply the construction $\ctop$ to each slice $L\times \{\tau\}\to B\times \R$ and, since the actions of $G$ and $\R$ on $L \times \R$ commute, it follows that the topological $G$-spaces $\ctop(L\times \{\tau\})$ are equivariantly homeomorphic for each $\tau$. So, defining $\bar{L}:= \ctop(L\times\{0\})$, it follows that $\sctop(c(L))\cong c(\bar{L})$ as topological $G$-spaces and furthermore that $[p,t]$ corresponds to an element of $\bar{L}\times \{t\}\subset c(\bar{L})$ for each $(p,t)\in L\times \R\subset c(L)$.

Therefore, the map $\varphi:P\to L\times \R$ then descends to a $(G\times \R)$-equivariant homeomorphism $\phi:\tilde{P}/\sim\to c(\bar{L})$.  As $\phi$ restricts to a $(G\times \R)$-equivariant homeomorphism between $\hc(P,\omega)\subset \tilde{P}/\sim$ and $\bar{L} \times \R$, we may conclude that $\bar{L}$ is in fact homeomorphic to the manifold $\hc(P,\omega)/\R$. Therefore, $\phi$ gives local trivialization data for $[*]$ in $\tilde{P}$ and so $(\tilde{P}/\sim,P/\sim \coprod\{[*]\})$ is a stratified space with one singularity and it follows by definition that $\ssc(P,\omega) = (\tilde{P}/\sim,P/\sim\coprod\{[*]\},\hc(\omega),\bar{\pi}:\tilde{P}/\sim\to c(B))$ is a singular symplectic toric cone over $\psi$ (i.e., a singular symplectic toric cone with moment map $\psi\circ \bar{\pi}$).
\end{proof}

{\bf Step 5:} Finally, we now show that $\ssc:\cstb_\psi(W)\to \psts_\psi(W)$ has image contained in $\stss_\psi(W)$ and hence restricts to a functor $\ssc:\cstb_\psi(W)\to \stss_\psi(W)$.

\begin{prop}
Let $\psi:W\to \fg^*$ be a stratified unimodular local embedding and let $(\pi:P\to \reg{W},\omega)$ be a conical symplectic toric bundle.  Then $\ssc(P,\omega)$ is a symplectic toric stratified space over $\psi$ and therefore defines a functor $\ssc:\cstb_\psi(W)\to \stss_\psi(W)$.  Furthermore, this yields a well-defined map of presheaves $\ssc:\cstb_\psi\to \stss_\psi$.
\end{prop}

\begin{proof}
 Let $\tilde{P}$, $\tilde{\pi}$, $\sim$, etc. be the same objects as defined in the previous steps.  To show that $\ssc(P,\omega)$, the partitioned symplectic toric space over $\psi$ defined in Step 2, is in fact a symplectic toric stratified space over $\psi$, we must show that
\begin{enumerate}
	\item $\ssc(P,\omega)$ is a stratified space with isolated singularities; and
	\item each singularity of $\ssc(P,\omega)$ has a neighborhood isomorphic to the neighborhood of the singularity in a singular symplectic toric cone.
\end{enumerate}
As neighborhoods of the singularity of any singular symplectic toric cone are homeomorphic to $c(L)$ for some compact manifold $L$, point (1) is subsumed by point (2).  

So fix a singularity $p_\alpha$ of $\tilde{P}$ lying over $w_\alpha$ in $W$. Then by Lemma \ref{l:pullbackdefncstb}, there exists an open neighborhood $U$ of $w_\alpha$, a homogeneous local trivialization datum $\varphi:U\to c(L)$ with homogeneous unimodular local embedding $\phi:L\times \R\to \fg^*$ satisfying $\psi|_{\reg{U}}=\phi\circ\varphi+\psi(w_\alpha)$, a homogeneous symplectic toric bundle $(\varpi : Q \to L \times \R,\eta) \in \hstb_\phi(L\times \R)$, and a $G$-equivariant isomorphism $f:(P,\omega)|_U\to \varphi^*(Q,\eta)$ in $\cstb_\psi(U)$.

Since $\ssc(\varphi^*(Q,\eta))=\varphi^*\ssc(Q,\eta)$ and $\ssc((P,\omega)|_U) = \ssc(P,\omega)|_U$ (see Lemma \ref{l:ssccommuteswithpullbacks}), we have an isomorphism
\[\ssc(f):\ssc(P,\omega)|_U\to \varphi^*\ssc(Q,\eta)\]
Thus, we have an isomorphism from $\ssc(P,\omega)|_U$ to $\varphi^*(Q,\eta)$, a neighborhood of the singularity of the singular symplectic toric cone $\ssc(Q,\eta)$.  It follows that $\ssc(P,\omega)$ is a symplectic toric stratified space over $\psi$.

Therefore, since $\stss_\psi(W)$ is a full subcategory of $\psts_\psi(W)$ (see Remark \ref{r:stssfullsubcat}), it follows that $\ssc:\cstb_\psi(W)\to \psts_\psi(W)$ restricts to a functor $\ssc:\cstb_\psi(W)\to \stss_\psi(W)$.  Furthermore, since $\ssc$ commutes with restriction (see Lemma \ref{l:ssccommuteswithpullbacks}), $\ssc$ is in fact a map of presheaves $\ssc:\cstb_\psi\to \stss_\psi$.
\end{proof}

% We finally define $\ssc$.

Now, we begin to show that $\ssc$ is an isomorphism of presheaves. To start, we show $\ssc$ is fully faithful.

\begin{lemma}\label{l:sccfullyfaithful} 
For any stratified unimodular local embedding $\psi:W\to \fg^*$, $\ssc:\cstb_\psi(W)\to \stss_\psi(W)$ is fully faithful.
\end{lemma}

\begin{proof}
Let $\iota:\cstb_\psi(W)\to \stb_{\psi|_{\reg{W}}}(\reg{W})$ be the forgetful functor (as defined in Definition \ref{d:forgetfulfunctor}). Let $\res:\stss_\psi(W)\to \stm_{\psi|_{\reg{W}}}(\reg{W})$ be the restriction functor; i.e., the functor which restricts $(X,\omega,\pi:X\to W)$, a symplectic toric stratified space over $\psi$, to the open and dense symplectic toric manifold over $\psi|_{\reg{W}}$ contained within.

Since $\iota$ and $\cut$ are both fully faithful (see Lemma \ref{l:forgetfulfunctorff} and Theorem \ref{t:cutequiv}, respectively), $\cut\circ \iota$ is fully faithful as well. Note
that the diagram
\[\xymatrix{\cstb_\psi(W)\ar[r]^{\iota} \ar[d]_{\ssc} & \stb_{\psi|_{\reg{W}}}(\reg{W}) \ar[d]^{\cut} \\ \stss_\psi(W)\ar[r]_{\res} & \stm_{\psi|_{\reg{W}}}(\reg{W})}\]
commutes. It follows that $\res\circ \ssc$ is also fully faithful. Notice that $\res$ is fully faithful: any map
of symplectic toric manifolds 
\[g : (X_1 ,\omega_1 ,\pi_1 : X_1 \to W)|_{\reg{W}} \to (X_2 ,\omega_2, \pi_2 : X_2 \to W)|_{\reg{W}}\] 
on the open and dense top strata of two symplectic toric stratified spaces over $\psi$ must extend to a map of symplectic toric stratified spaces which, by virtue of continuity, is uniquely determined.  To see this, notice that such a map defines a map between the symplectic toric cone local models for $(X_1,\omega_1 ,\pi_1)$ and $(X_2 ,\omega_2, \pi_2)$ on deleted neighborhoods of each singularity which, by Lemma \ref{l:morphismsextend}, must extend to maps of singular symplectic toric cones.  Hence, $\ssc$ must be fully faithful.
\end{proof}

We now show the elements of $\stss_\psi$ are locally isomorphic to elements of the image of $\ssc$.

\begin{lemma}\label{l:localisopropertystss}
Let $\psi:W\to \fg^*$ be a stratified unimodular local embedding. Then for any symplectic toric stratified space $(X, \omega,\pi:X\to W)$ over $\psi$ and for any point $w\in W$, there is an open neighborhood $U_w$ of $w$ and a conical symplectic toric bundle $(\varpi : P \to \reg{(U_w)} ,\eta)$ in $\cstb_\psi(U_w)$ so that $\ssc(P,\eta)$ is isomorphic to $(X,\omega,\pi)|_{U_w}$.
\end{lemma}

\begin{proof}
First, assume $w\in \reg{W}$.  Let $U_w$ be a contractible open neighborhood of $w$ small enough so that $U_w \subset \reg{W}$. Since $U_w \subset \reg{W}$, $\stss_\psi(U_w) = \stm|_{\psi|_{\reg{W}}}(U_w)$. Since $U_w$ is contractible, all elements of $\stm_{\psi|_{\reg{W}}}(U_w)$ are isomorphic. Finally, since $\cstb_\psi(U_w)=\stb_{\psi|_{\reg{W}}}(U_w)$, $\cstb_\psi(U_w)$ is nonempty and therefore $\stm|_{\psi|_{\reg{W}}}(U_w)$ contains elements of the image of $\ssc$.

So assume $w$ is a singularity of $W$. Let $U_w$ be any neighborhood of $w$ for which there exists a singular symplectic toric cone $(C,\omega',\nu:C\to \fg^*)$ with neighborhood $V$ of the vertex $*$ of $C$, and a map $\varphi:\pi^{-1}(U_w)\to V$ satisfying the conditions described in Definition \ref{d:stss}.

For orbital moment map $\bar{\nu}:C/G\to \fg^*$, $\bar{\nu}|_{\reg{(C/G)}}$ is a homogeneous unimodular local embedding. It follows that, since $\hc$ is an isomorphism of presheaves, there exists a homogeneous symplectic toric bundle $(\pi:P\to \reg{(C/G)},\alpha)$ over $\bar{\nu}$ for which $\hc(P,\alpha) = (C,\omega',\bar{\pi}:C\to \reg{(C/G)})$.

Via Proposition \ref{p:homogconn} and Proposition \ref{p:hstbnohorclass}, we may assume that $\alpha = d\langle \bar{\nu}\circ \pi,A\rangle$ for $A$ an $\R$-invariant connection 1-form on $P$. As noted in the proof of Proposition \ref{p:stssommasule}, $\varphi$ descends to an isomorphism $\bar{\varphi}:U_w\to C/G$ satisfying $\psi|_{U_w}=\bar{\nu}\circ \bar{\varphi}+\psi(w)$. Let $(\varpi:Q\to \reg{(U_w)})$ pullback of $\pi:P\to \reg{(C/G)}$ to $\reg{(U_w)}$ via $\bar{\varphi}$ at let $\tilde{\varphi}:Q\to P$ be the corresponding isomorphism
covering $\bar{\varphi}$. Let $A'=\tilde{\varphi}^*A$. Then
\[\tilde{\varphi}^*d\langle\bar{\nu}\circ \pi, A\rangle = d\langle \bar{\nu}\circ \pi \circ \tilde{\varphi},A'\rangle=d\langle \bar{\nu}\circ \bar{\varphi}\circ \varpi, A'\rangle = d\langle 
\psi \circ \varpi, A'\rangle - d\langle \psi(w),A'\rangle\]

$d\langle \psi(w),A'\rangle$ is basic: to see this, it is enough to show that it is $G$-invariant and horizontal.  It is $G$-invariant since $G$ is commutative, meaning $A'$ is $G$-invariant. To see it is horizontal, notice that, for any $X\in \fg$ with associated fundamental vector field $X_P$ on $P$,
\[\iota_{X_P}d\langle \psi(w),A'\rangle =  \mathcal{L}_{X_P}\langle \psi(w),A'\rangle - d\iota_{X_P}\langle \psi(w),A'\rangle = 0-d\langle \psi(w),X\rangle = 0\]

So, since $d\langle \psi(w),A'\rangle$ is basic, there exists a closed 2-form $\beta$ on $\reg{(U_w)}$ with $\pi^*\beta = -d\langle \psi(w),A'\rangle$.  As shown by Karshon and Lerman in \cite{KarshonLerman} (see Lemma \ref{l:KLbijectionlemma}), $d\langle\psi\circ \varpi,A'\rangle+\pi^*\beta$ is a $G$-invariant symplectic form on $Q$ with moment map $\psi\circ\varpi$. Therefore, $(Q,\langle\psi\circ \varpi,A'\rangle+\pi^*\beta)$ is a symplectic toric bundle and, since $(Q,\langle\psi\circ \varpi,A'\rangle+\pi^*\beta) = \varphi^*(P,d\langle \bar{\nu}\circ \pi, A\rangle)$, it follows that $(Q,\langle\psi\circ \varpi,A'\rangle+\pi^*\beta)$ is a conical symplectic toric bundle over $\psi|_{U_w}$.

Finally, note that
\[\ssc(Q,\langle\psi\circ \varpi,A'\rangle+\pi^*\beta) = \ssc(\varphi^*(P,d\langle\bar{\nu}\circ\pi\rangle))= \varphi^*\ssc(P,d\langle\pi\circ\bar{\nu}\rangle)\cong (X,\omega,\pi)|_{U_w}\]
\end{proof}

We need one more lemma before we can finish.

\begin{lemma}\label{l:stssaprestack}
Let $\psi:W\to \fg^*$ be a stratified unimodular local embedding. Then the presheaf of groupoids $\stss_\psi:\op{\open(W)}\to \gpoid$ is a prestack (see Definition \ref{d:prestack}).
\end{lemma}

\begin{proof}
Let $U\subset W$ be any open subset and let $(X,\omega,\pi:W\to U)$ and $(X',\omega',\pi':X'\to U)$ be two symplectic toric stratified spaces over $\psi|_U$. As any map $f:(X,\omega,\pi)\to (X',\omega',\pi')$  in $\stss_\psi(U)$ is, in particular, a homeomorphism $f:X\to X'$, $f$ is uniquely determined by its restrictions to any open cover of $U$. Since the various properties of maps in $\stss_\psi$ glue, it follows that $U\mapsto {\sf Hom}(X|_U,X'|_U)$ is a sheaf of sets.
\end{proof}

We may now prove that $\ssc$ is an isomorphism of presheaves.

\begin{thm}\label{t:sscaniso} 
For any stratified unimodular local embedding $\psi:W\to \fg^*$, the map of presheaves $\ssc:\cstb_\psi\to \stss_\psi$ is an isomorphism of presheaves.
\end{thm}

\begin{proof}
We've shown in in Lemma \ref{l:ssccommuteswithpullbacks} that $\ssc$ is a map of presheaves, in Lemma \ref{l:sccfullyfaithful} that $\ssc_U:\cstb_\psi(U)\to \stss_\psi(U)$ is fully faithful for every $U$, and in Lemma \ref{l:stssaprestack} that $\stss_\psi$ is a prestack. In Proposition \ref{p:cstbastack}, we explain why $\cstb_\psi$ is a stack.  Therefore, with Lemma \ref{l:localisopropertystss}, we have that $\ssc:\cstb_\psi\to \stss_\psi$ satisfies the hypotheses of Lemma \ref{l:isoofstackstechlemma} and is an
isomorphism of presheaves.
\end{proof}

%%%%%%%%%%%%%%%%%%%%%%%%%%%%%%%%%%%%%%%%%
\section{Characteristic classes for symplectic toric stratified spaces with isolated singularities}\label{s:ccforstss}
%%%%%%%%%%%%%%%%%%%%%%%%%%%%%%%%%%%%%%%%%

Now we describe characteristic classes for $\cstb_\psi(W)$ and hence, via our equivalence of categories $\ssc : \cstb_\psi(W)\to \stss_\psi(W)$, characteristic classes for $\stss_\psi(W)$.  Instead of creating new classes for $\cstb_\psi(W)$, we will use the forgetful functor $\iota:\cstb_\psi(W)\to \stb_{\psi|_{\reg{W}}}(\reg{W})$ together with the characteristic classes for symplectic toric bundles established by Karshon and Lerman for symplectic toric bundles (again, see Section \ref{s:KL}). We will then identify exactly which cohomology classes in $H^2(\reg{W},\Z_G\times \R)$ correspond to symplectic toric stratified spaces.

To begin, let’s formally define horizontal classes for conical symplectic toric bundles.

\begin{defn}
 Let $\psi:W\to \fg^*$ be a stratified unimodular local embedding. Then for $(\pi:P\to \reg{W},\omega)$ a conical symplectic toric bundle over $\psi$, define
 \[\shor([P,\omega]):=\chor([\iota(P,\omega)])\in H^2(\reg{W},\R) \]
for $\chor$ the horizontal class for symplectic toric bundles (see Definition \ref{d:chor}) and $\iota:\cstb_\psi(W)\to \stb_{\psi|_{\reg{W}}}(\reg{W})$ the forgetful functor.
\end{defn}

We now show that, at the very least, the Chern class and horizontal class together give a manner to distinguish isomorphism classes of conical symplectic toric bundles.

\begin{prop}\label{p:charclassesinjective}
The first Chern class $\chern$ and the horizontal class $\shor$ define an injective morphisms of presheaves of sets over $\reg{W}$
\[(\chern,\shor):\pi_0\cstb_\psi\to H^2(\cdot,\Z_G\times \R)\cong H^2(\cdot,\Z_G)\times H^2(\cdot,\R)\]
\end{prop}

\begin{proof}
Since the forgetful functor $\iota:\cstb_\psi(W)\to \stb_{\psi|_{\reg{W}}}(\reg{W})$ is fully faithful (see Lemma \ref{l:forgetfulfunctorff}), $\iota$ does not identify non-isomorphic elements of $\cstb_\psi(W)$. Therefore, since the characteristic class of Karshon and Lerman $(\chern,\chor):\pi_0\stb_{\psi|_{\reg{W}}}\to H^2(\cdot,\Z_G\times \R)$ produces an isomorphism of presheaves, and since for any conical symplectic toric bundle $(P,\omega)$ we have $(\chern,\shor)([P,\omega])=(\chern,\chor)([\iota(P,\omega)])$, it follows that $(\chern,\shor)$ must be injective.
\end{proof} 

In general, we will not be able to say in general that $(\chern,\shor)$ is surjective.  Indeed, note that a conical symplectic toric bundle must, in particular, be exact in a deleted neighborhood of each singularity.  Since the horizontal class gives an obstruction to a symplectic toric bundle to be exact, one might suspect that the image of $(\chern,\shor)$ are exactly those with $\shor$ exact near each singularity.  However, this turns out to be true only for flat bundles.  As we will see below, the curvature of the given bundle and the value of $\psi$ at each singularity imposes a particular condition on the horizontal classes associated to $\psi$.
%  However, we will be able to identify the image of $(\chern,\shor)$ as the kernel of the following map.

% % \begin{lemma}
% % Let $W$ be a cornered stratified space with isolated singularities $\{w_\alpha\}_{\alpha\in A}$ having links $L_\alpha$.  Then there is a short exact sequence in singular cohomology
% % \[\xymatrix{0 \ar[r] H^2_{\sf sing}(W,\R)\ar[r]^{\iota^*} & H_{\sing}^2(\reg{W},\R) \ar[r]^{{\sf res}} & H_{\sf sing}^2\left(\coprod_{\alpha\in A} L_\alpha\right) \ar[r] & 0}\]
% % for $\iota:\reg{W}\to W$ the inclusion map and ${\sf res}$ the restriction of a singular cochain $\sigma$ on $\reg{W}$ to a slice $L_\alpha$ of the cone. 
% % \end{lemma}

% % \begin{proof}
% % First, for each singularity $w_\alpha$, recall by definition there exists a compact manifold with corners $L_\alpha$, neighborhood $U_\alpha$ of $w_\alpha$, and open embedding $\varphi_\alpha:U_\alpha \to c(L_\alpha)$ taking $w_\alpha$ onto the vertex of $c(L_\alpha)$.  As each $c(L_\alpha)$ is contractible and since $U_\alpha\cap \reg{W}$ is homotopy equivalent to $L_\alpha$ for each $\alpha$, the standard Mayer-Vietoris long exact sequence reduces to the following:
% % \[ \xymatrix{\ldots \ar[r] & H^1(\displaystyle \coprod_\alpha L_\alpha,\R) \ar[r] & H_{\sf sing}^2(W,\R) \ar[r]^-{\iota^*} &  H^2(\reg{W},\R) \ar[r] & H^2(\displaystyle \coprod_\alpha L_\alpha,\R) \ar[r] & \ldots  }\]

% % \end{proof}

To demonstrate this boundary condition, we will need the following theorem for principal bundles in the category of manifolds with corners; this is of course well-known for topological principal bundles.

\begin{thm}\label{t:homotopyinv}
Let $\pi:E\to B$ be a principal $G$-bundle of manifolds with corners and, for a manifold with corners $X$, let $f_0:X\to B$ and $f_1 : X \to B$ be two smoothly homotopic maps. Then the pullbacks $f_0^*E$ and $f^*_1E$ are isomorphic as principal $G$-bundles in the category of manifolds with corners.
\end{thm}

\begin{proof}
First, recall that, in the manifolds without corners case, this theorem can be proven via the flow of the horizontal lift of $\frac{d}{dt}$ on the pullback bundle $f_t^*E\to X\times[0,1]$ after any choice of equivariant connection on $f_t^*E$.  

To see that this argument extends to the manifolds with corners case, notice that via local triviality $f_t^*E\to X\times[0,1]$ decomposes into principal bundles of manifolds without corners $S^\ell(f_t^*E)\to S^\ell(X\times[0,1])$ (again, here $S^{\ell}$ refers to the manifold of points with index or depth $\ell$).  Since $\frac{d}{dt}$ is tangent to each $S^{\ell}(X\times[0,1])$ (excluding at $t=0,1$), it follows that its horizontal lift will be tangent to each submanifold of constant index in $f_t^*E$.  Hence, we may reduce to the case without corners.
\end{proof}

Finally, we provide the necessary and sufficient boundary conditions each horizontal class associated to a conical symplectic toric bundle over a given stratified unimodular local embedding $\psi$ any horizontal class must satisfy.  This will end up being enough to fully describe the image of $(\chern,\shor)$ in $H^2(\reg{W},
\Z_G\times \R)$.  First, we exhibit the following class of real valued 2-form associated to any principal bundle $\pi:P\to \reg{W}$.

\begin{rem}\label{r:Chernclasspairing}
Given a principal $G$-bundle $\pi:P\to M$ over a manifold (with corners) $M$, it is well known that one method for calculating the first Chern class of $P$ is to pick a connection $A \in \Omega^1(P,\fg)$, find the class $[\eta]\in H^2(M,\fg)$ with $\pi^*\eta = dA$, and prove that $[\eta]$ actually came from a class $H^2(M,\Z_G)$.  

On the other hand, given a fixed element $\xi \in \fg^*$, note that the pairing $\langle \xi, \eta \rangle$ yields a 2-form $\Omega^2(M,\R)$.  Furthermore, by linearity of the canonical pairing $\langle \cdot, \cdot \rangle$, it follows that this form is closed with class uniquely determined by the class of $\eta$.  Therefore, it makes sense to write $\langle \xi,[\chern(P)]\rangle \in H^2(M,\R)$ with the understanding that this represents the class $[\langle \xi, \eta\rangle]$.
\end{rem} 

\begin{prop}\label{p:conicalboundarycond}
Let $\psi:W\to \fg^*$ be a homogeneous unimodular local embedding and $(\pi:P\to \reg{W},\omega)$ a symplectic toric bundle over $\psi|_{\reg{W}}$.  Then $(P,\omega)$ is a conical symplectic toric bundle precisely when for each singularity $w_\alpha$ of $W$, there exists a deleted open neighborhood $U_\alpha$ of $w_\alpha$ such that
\begin{equation}\label{eq:localconicalcondition}\langle\psi(w_\alpha), \chern(P)|_{U_\alpha}\rangle  = - [\chor(P,\omega)]|_{U_\alpha} \end{equation}
In particular, if $P$ is flat, $(P,\omega)$ is a conical symplectic toric bundle precisely when $\chor(P,\omega)$ is exact in a deleted open neighborhood of each singularity.
\end{prop}

\begin{proof}
First, suppose that $(P,\omega)$ is a conical symplectic toric bundle.   Fix a singularity $w_\alpha$ with homogeneous local trivialization datum $\varphi_\alpha:U_\alpha\to c(L_\alpha)$ and with homogeneous unimodular local embedding $\nu_\alpha:L_\alpha\times \R\to \fg^*$ satisfying $\psi|_{\reg{(U_\alpha)}} = \nu_\alpha\circ \varphi_\alpha+\psi(w_\alpha)$. Then, by Lemma \ref{l:pullbackdefncstb}, there exists a homogeneous symplectic toric bundle $\varpi:Q\to L_\alpha\times \R$ and an isomorphism $\tilde{\varphi}_\alpha:(P,\omega)|_{\reg{(U_\alpha)}} \to \varphi_\alpha^*(Q,\eta)$ in $\stb_{\psi|_{\reg{(U_\alpha)}}}(\reg{(U_\alpha)})$.  By the naturality of $\chor$, we may proceed by showing that $\varphi_\alpha^*Q$ has the desired property.  

By Proposition \ref{p:hstbnohorclass}, we may assume that $\eta = d\langle \nu_\alpha\circ\varpi,A\rangle$, for $A$ an appropriately chosen connection on $Q$, as in Proposition \ref{p:homogconn}.  Then for $\tilde{\varphi}_\alpha:\varphi_\alpha^*Q\to Q$ the morphism covering $\varphi_\alpha$, we have 
\[\tilde{\varphi}_\alpha^*d\langle \nu_\alpha\circ \varpi,A\rangle = d\langle \psi|_{\reg{(U_\alpha)}}\circ \varpi-\psi(w),\tilde{\varphi}_\alpha^*A\rangle = d\langle \psi|_{\reg{(U_\alpha)}}\circ \varpi,\tilde{\varphi}^*_\alpha A\rangle - d\langle \psi(w),\tilde{\varphi}_\alpha^*A\rangle\]
Since $\tilde{\varphi}_\alpha^*A$ is a connection on $\varphi_\alpha^*Q$, it follows by definition that $-d\langle \psi(w),\tilde{\varphi}_\alpha^*A\rangle$ is exactly the form that descends to $\chor(\varphi_\alpha^*Q)$.  On the other hand, as discussed in Remark \ref{r:Chernclasspairing}, $-d\langle \psi(w_\alpha),\varphi_\alpha^*A\rangle = -\langle \psi(w_\alpha),\varphi_\alpha^*dA\rangle$ descends to the cohomology class $-\langle \psi(w_\alpha),\chern(P)|_{U_\alpha}\rangle$.

Conversely, suppose that a symplectic toric bundle $(\pi:P\to \reg{W}, \omega)$ over $\psi|_{\reg{W}}$ satisfies the hypothesized local condition \eqref{eq:localconicalcondition} on deleted open neighborhoods $U_\alpha$ of each singularity $w_\alpha$ of $W$. By shrinking each neighborhood $U_\alpha$ as necessary, we may assume each $U_\alpha$ is the domain of a homogeneous local structure datum $\varphi_\alpha:U_\alpha\to c(L_\alpha)$ with corresponding homogeneous unimodular local embedding $\nu_\alpha:L_\alpha\times \R\to \fg^*$ satisfying $\psi|_{\reg{(U_\alpha)}} = \nu_\alpha \circ \varphi_\alpha+\psi(w_\alpha)$.  Furthermore, again by shrinking $U_\alpha$ as necessary, we may assume $\varphi_\alpha$ maps $U_\alpha $ onto $L_\alpha\times (-\infty,\varepsilon)$ for some constant $\varepsilon$.
% Fix a bundle $\pi:P\to \reg{W}$ with Chern class $\chern(P)=[\beta]$.  We must construct a symplectic form on $P$ with respect to which $(P,\omega)$ is a conical symplectic toric bundle over $\psi$ and for which the horizontal class $\shor(P,\omega)=[\gamma]$. So let $w_\alpha$ be a singularity of $W$ with homogeneous local structure datum $\varphi_\alpha:U_\alpha\to c(L_\alpha)$ and homogeneous unimodular local embedding $\nu_\alpha:L_\alpha\times \R\to \fg^*$. By shrinking $U_\alpha$ as necessary, we may assume that $\varphi_\alpha$ maps $U_\alpha$ onto $L_\alpha\times (-\infty,\varepsilon)$ for some constant $\varepsilon$.

Fix $\tau\in \R$ with $\tau<\varepsilon$ and let $\Sigma := \varphi^{-1}(L_\alpha\times \{\tau\})$. Define $Q_\alpha :=P|_\Sigma$. Since $U_\alpha$ is isomorphic to $L_\alpha\times (-\infty,\varepsilon)$, it follows that $U_\alpha$ admits a retraction onto $\Sigma$; thus, via Theorem \ref{t:homotopyinv}, we may conclude that there exists an isomorphism of principal $G$-bundles $\tilde{\varphi}_\alpha:P|_{\reg{(U_\alpha)}}\to Q_\alpha \times (-\infty,\varepsilon)$ covering $\varphi_\alpha$.  

To finish, note that $\varpi:Q_\alpha \times \R \to L_\alpha\times \R$ satisfies the conditions of Proposition \ref{p:homogconn}; hence, there exists a connection $A'$ on $Q_\alpha\times \R$ such that $(Q_\alpha\times \R,d\langle \nu_\alpha \circ \varpi, A'\rangle)$ is a homogeneous symplectic toric bundle over $\nu_\alpha$. Let $A_\alpha$ be the pullback of $A'$ to $P|_{\reg{(U_\alpha)}}$.

Similar to the calculation from above, notice we have
\[\tilde{\varphi}_\alpha^*d\langle\nu_\alpha\circ \varpi,A'\rangle=d\langle \psi\circ \pi-\psi(w_\alpha),A_\alpha\rangle = d\langle \psi\circ \pi,A_\alpha\rangle -d\langle \psi(w_\alpha),A_\alpha\rangle \]
and thus the horizontal class of $(P|_{\reg{(U_\alpha)}},\tilde{\varphi}_\alpha^*d\langle\nu_\alpha\circ \varpi,A'\rangle)$ is $-\langle \psi(w_\alpha),[\chern(P)]|_{\reg{(U_\alpha)}}\rangle$.  By assumption, this is also the horizontal class of $(P,\omega)|_{\reg{(U_\alpha)}}$ and therefore $(P,\omega)|_{\reg{(U_\alpha)}}$ is isomorphic as a symplectic toric bundle to $\tilde{\varphi}_\alpha^*(Q_\alpha\times \R,d\langle\nu_\alpha\circ \varpi,A'\rangle)$.  It follows again via Lemma \ref{l:pullbackdefncstb} that $(P,\omega)$ satisfies the necessary local condition to be a conical symplectic toric bundle.  
\end{proof}

Now, we may prove the main classification of this paper.\\

\noindent {\bf Theorem \ref{t:maintheorem}.}  Let $\psi:W\to \fg^*$ be a stratified unimodular local embedding. Then the set of isomorphism classes of symplectic toric stratified spaces with isolated singularities $(X,\omega,\mu:X\to \fg^*)$ with $G$-quotient map $\pi:X\to W$ and orbital moment map $\psi$ is naturally isomorphic to a subspace $\Cc$ of the cohomology classes $H^2(\reg{W},\Z_G\times \R)$.  Here, $\Z_G$ denotes the integral lattice of $G$ (the kernel of $\exp:\fg \to G$) and $\reg{W}$ denotes the top stratum of $W$.  

In particular:
\begin{itemize}
\item If $\dim(G)\neq 3$, $\Cc = H^2(\reg{W},\Z_G\times \R).$
\item If $\dim(G)=3$, $\Cc$ is an extension of $H^2(\reg{W},\Z_G)$ by $H^2(W,\R)$. 
\end{itemize}

\begin{proof}
 By Theorem \ref{t:sscaniso}, $\ssc:\cstb_\psi\to \stss_\psi$ is an isomorphism of presheaves from the presheaf of of conical symplectic toric bundles over $\psi$ to the presheaf of symplectic toric stratified spaces over $\psi$. This descends to an isomorphism of presheaves of sets $\pi_0\ssc:\pi_0\cstb_\psi \to \pi_0\stss_\psi$.  Proposition \ref{p:charclassesinjective} establishes that we have an injective morphism of presheaves $(\chern,\shor):\pi_0\cstb_\psi \to H^2(\cdot,\Z_G\times \R)$.  

 Let $\mathcal{C}$ denote the image of $(\chern,\shor)$ in $H^2(\reg{W},\Z_G\times \R)$.  By Proposition \ref{p:conicalboundarycond}, it follows that the classes in $\mathcal{C}$ are precisely those satisfying boundary condition \eqref{eq:localconicalcondition}.  Notice this condition is checked on deleted open neighborhoods of the singularities.  Since every such neighborhood $U_\alpha$ near a singularity $w_\alpha$ is (after shrinking as necessary) homotopic to $L_\alpha$, for $L_\alpha$ the link of $W$ at $w_\alpha$, it follows that $\mathcal{C}$ is highly determined by the cohomology on the links.  From here, we split into cases by dimension.

 In the case where $\dim(G)\neq 3$: since $\ssc$ is essentially surjective, it follows the links of $W$ must be homeomorphic to (disjoint unions of) such $G$-quotients.  As proven by Lerman \cite{LermanCTM}, the $G$-quotients of connected contact toric manifolds of dimension greater than 3 (hence $\dim(G)>2$) are either spheres or contractible.   Hence, for $\dim(G)>3$, the links of $W$ are contractible or are disjoint unions of spheres of dimension at least 3 while $\dim(G)<3$, the links of $W$ have dimension at most 1.  Therefore, in either case, the links satisfy $H^2(L_\alpha,\R)=0$ and thus boundary condition \eqref{eq:localconicalcondition} is always trivially satisfied.

 In the case where $\dim(G)=3$: denote by $p:\mathcal{C}\to H^2(\reg{W},\Z_G)$ the restriction of the projection $H^2(\reg{W},\Z_G)\times H^2(\reg{W},\R)\to H^2(\reg{W},\Z_G)$ and $i:H^2(W,\R)\to H^2(\reg{W},\Z_G)\times H^2(\reg{W},\Z_G)$ the map $[\sigma]\mapsto (\iota^*[\sigma],0)$, for $\iota^*:H^2(W,\R)\to H^2(\reg{W},\R)$ (i.e., for $\iota:\reg{W}\to W$ the inclusion map).  Here, $H^2(W,\R)$ denotes the singular cohomology of $W$ with values in $\R$ and we are implicitly passing between the de Rham cohomology of $H^2(\reg{W},\R)$ and its singular cohomology valued in $\R$. We will show that $i$ takes values in $\mathcal{C}$ and that the corresponding sequence
 \begin{equation}\label{eq:the3DSES}\xymatrix{0\ar[r] & H^2(W,\R)\ar[r]^-i & \mathcal{C} \ar[r]^-p & H^2(\reg{W},\Z_G)\ar[r] & 0}\end{equation}
is exact.

 For each singularity $w_\alpha$ of $W$, pick a neighborhood $U_\alpha$ of $w_\alpha$ in $W$ homeomorphic to an open neighborhood of the vertex of $C(L_\alpha)$.  By shrinking the $U_\alpha$s as necessary, we have that each intersection $\reg{W}\cap U_\alpha$ is homeomorphic to $L_\alpha\times \R$.  Using this fact together with the fact that the $U_\alpha$s are contractible, one can show that the Mayer-Vietoris sequence in singular cohomology associated to the open cover $\{\left(\coprod_\alpha U_\alpha\right),\reg{W}\}$ reduces to the sequence
 \[\xymatrix{\ldots \ar[r] & H^1(\coprod_\alpha L_\alpha,\R) \ar[r] & H^2(W,\R) \ar[r] & H^2(\reg{W},\R) \ar[r] & H^2(\coprod_{\alpha} L_\alpha,\R)\ar[r] & \ldots}\]
 Again, via the classification of Lerman \cite{LermanCTM}, we have that the links of $W$ are disjoint unions of 2-spheres or contractible manifolds with corners.  Hence, the links $L_\alpha$ of $W$ have zero cohomology outside of degree 2 and the above long exact sequence reduces to the short exact sequence
 \[\xymatrix{0 \ar[r] & H^2(W,\R) \ar[r]^{\iota^*} & H^2( \reg{W},\R) \ar[r]^r & H^2(\coprod_\alpha L_\alpha,\R) \ar[r] & 0}\]
 for $r:H^2(\reg{W},\R) \to H^2(\coprod_\alpha L_\alpha,\R)$ the restriction of a class over $\reg{W}$ to a slice near each singularity homeomorphic to a link.  In particular, this shows us that $\iota^*$ is injective with image precisely those classes which are exact near each singularity.  Therefore, via Proposition \ref{p:conicalboundarycond}, we have that $\iota^*$ maps $H^2(W,\R)$ onto the horizontal classes associated to flat conical symplectic toric bundles and thus the map $i:H^2(W,\R)\to H^2(\reg{W},\R)\times H^2(\reg{W},\Z_G)$ takes values in $\mathcal{C}$.  Additionally, since $\iota^*$ is injective, $i$ is injective as well and since the image of $\iota^*$ yields {\em all} horizontal classes of flat bundles, it follows that \eqref{eq:the3DSES} is exact at $\mathcal{C}$.

 Finally, it remains to be shown that $p$ of $\eqref{eq:the3DSES}$ is surjective.  It's enough to show that every principal $G$-bundle $\pi:P\to \reg{W}$ admits a symplectic form with respect to which $(P,\omega)$ is a conical symplectic toric bundle.  But note that since the retraction $r:H^2(\reg{W},\R) \to H^2(\coprod_\alpha L_\alpha,\R)$ is surjective, any class of forms defined near the singularities of $W$ extends to a class on all $W$.  In particular,the local degree 2 classes $-\langle \psi(w_\alpha),\chern(P)|_{\reg{(U_\alpha)}}\rangle$ from the boundary conditions \eqref{eq:localconicalcondition} of Proposition \ref{p:conicalboundarycond} extend to a global degree 2 class over all $\reg{W}$.  This global class necessarily corresponds to the horizontal class of a conical symplectic toric bundle structure on $P$.
\end{proof}

\begin{rem}\label{r:SEStoprincbudle}
As with any group extension, the sequence \eqref{eq:the3DSES} induces a principal $H^2(W,\R)$-bundle structure on the projection $p:\mathcal{C}\to H^2(\reg{W},\Z_G)$, where $H^2(W,\R)$ acts on $\mathcal{C}$ by translation.
\end{rem}

The following two corollaries are more or less immediate.    

\begin{cor}
Let $\psi:W\to \fg^*$ be a stratified unimodular local embedding and suppose that $\dim(G) \neq 3$. Then every symplectic toric manifold over $\psi|_{\reg{W}}$ is isomorphic to the restriction of a symplectic toric stratified space over $\psi$. 
\end{cor}

\begin{cor}
Let $\psi:W\to \fg^*$ be a stratified unimodular local embedding and suppose that $\dim(G)=3$.  Then the space of characteristic classes $\mathcal{C}$ for $\stss_\psi(W)$ is non-canonically isomorphic to $H^2(W,\R)\times H^2(\reg{W},\Z_G)$.
\end{cor}

\begin{rem}\label{r:classificationdistinctions}
 It’s worth comparing this to a previous related result by Burns, Guillemin, and Lerman \cite{BurnsGuilleminLerman}. They proved that compact connected symplectic toric stratified spaces are
classified by rational polytopes in $\fg^*$ that may fail to be simple only at the vertices (an example being the octahedron, as mentioned in Example \ref{ex:octahedron}). Our main result generalizes the result of Burns, Guillemin, and Lerman in a couple of notable ways. First, the compact and connected conditions are unnecessary for the main result (at the price of classification
via moment map image).

A couple of additional technical conditions included by Burns, Guillemin, and Lerman have also been removed: first, we do not assume that the links of our isolated singularities
are not spheres. Notice this allows for symplectic toric stratified spaces that are homeomorphic to topological 4-manifolds and the ``singularity'' in fact indicates a singularity in the
symplectic structure. Additionally, we do not assume that the links of our stratified spaces have convex moment map images. Example \ref{ex:Lshape} demonstrates an example that fails both
of these technical conditions.
\end{rem}

We conclude this section with some illustrative examples.

\begin{ex}\label{ex:Lshape}
Suppose that $\psi$ is the embedding of the following region into $\R^2$
\[\begin{tikzpicture}\fill[gray] (0,0) -- (0,2) -- (1,2) -- (1,1) -- (2,1) -- (2,0) -- (0,0); \draw (0,0) -- (0,2) -- (1,2) -- (1,1) -- (2,1) -- (2,0) -- (0,0);\node at (1,1) {$\bullet$}; \node at (-0.4,-0.2)  {\tiny $(0,0)$}; \node at (1.4,2.2) {\tiny $(1,2)$}; \node at (2.4,1.2){\tiny $(2,1)$}; \end{tikzpicture}\]
where the ``$\bullet$'' is a singular point of the region as a cornered stratified space.  Then our main result tells us that, since this is a contractible region, there is exactly one symplectic
toric stratified space with orbital moment map $\psi$ having one singularity.

Furthermore, it is interesting to note that the link of this symplectic toric stratified space is a three sphere with an overtwisted contact structure. Hence, this symplectic toric stratified space inherits the topological structure of a 4-manifold; however, the singularity is indicating that the associated symplectic structure cannot be extended to this 4-manifold structure.
\end{ex}

\begin{ex}\label{ex:octahedron}
Let $G=\T^3$ and, for $\{v_1,v_2,v_3\}$ a basis of the integral lattice of $\fg^*$, let $\Delta$ be the octahedron in $\fg^*$ that is the convex hull of $\{\pm v_1,\pm v_2,\pm v_3\}$.  Let $\iota:\Delta \to \fg^*$ be the inclusion of $\Delta$ into $\fg^*$.  Then $\iota$ is a stratified unimodular local embedding, where we think of $\Delta$ as a cornered stratified space with its vertices as singularities.  As $\Delta$ is still convex when we remove its vertices, we have that $H^2(\reg{\Delta},\R)=H^2(\reg{\Delta},\Z_G)=0$ and so there is a unique symplectic toric stratified space over $\iota$.

However, if we look instead at $\Delta_0:=\Delta\backslash\{0\}$ with embedding $j:\Delta_0\to \fg^*$, we have that $H^2(\reg{\left(\Delta_0\right)},\Z_G)\cong H^2(S^2,\Z_G)\neq 0$.  Here, each isomorphism class of principal $G$-bundle over $\Delta_0$ (which are in bijection with $H^2(S^2,\Z_G)$) corresponds to a unique collection of classes of symplectic toric stratified space over $j$ which are themselves each in bijective correspondence with the cohomology classes of $H^2(S^2,\R)$.

We may take things one step further: suppose that we add back the origin back in to $\Delta_0$ as a singular point to get the cornered stratified space $\Delta'$ (with singularities at the origin and at each vertex) and with stratified unimodular local embedding $k:\Delta\to \fg^*$ the extension of $j$ above with $k(0)=0\in \fg^*$.  In this case, our main theorem tells us that every principal bundle over $\Delta_0$ will descend to exactly one symplectic toric stratified space over $k$.  In other words, for every principal bundle over $\Delta_0$, there is exactly one class of symplectic toric stratified space over $j:\Delta_0\to \fg^*$ which extends to a symplectic toric stratified space over $k:\Delta'\to \fg^*$.
\end{ex}

\begin{ex}
As a relatively simple example illustrating why an extension is necessary to describe the space $\mathcal{C}$ of allowable classes for symplectic toric stratified spaces with isolated singularities (i.e., rather than just the product $H^2(W,\R)\times H^2(\reg{W},\Z_G)$), take $G=\T^3$ and fix $\eta$ and $\xi$ two distinct elements of $\fg^*$.  Let $W:=\fg^*\backslash \{\eta,\xi\}\coprod \{\eta\}$, i.e., the stratified space with one singularity at $\eta$.  Let $\psi:W\to \fg^*$ be the inclusion.  Then for each principal $G$-bundle $\pi:P\to \reg{W}$, the horizontal classes associated to $P$ are those satisfying the local condition $\langle \eta, [\chern(P)]\rangle$ near $\eta$, yielding a 1-dimensional affine space modeled on $H^2(W,\R)$.  In particular, when $P$ is not flat and when $\eta\neq 0$, this affine space needn't be an honest subspace of $H^2(\reg{W},\Z_G)\times H^2(\reg{W},\R)$.  
\end{ex}

\appendix

\part*{Appendices}

For this paper, we provide two appendices. Appendix \ref{app:stcs} describes the well-known relationship between symplectic toric cones and contact toric manifolds. Appendix \ref{app:stacks} gives a description of stacks (or, more aptly for this paper, strict sheaves of groupoids).

%%%%%%%%%%%%%%%%%%%%%%%%%%%%%%%%%%%%%%%%%
\section{Symplectic cones and contact manifolds}\label{app:stcs}
%%%%%%%%%%%%%%%%%%%%%%%%%%%%%%%%%%%%%%%%%

This appendix gives the definition of and details about symplectic cones and contact toric manifolds used throughout the paper. Sources for this information include \cite{LermanCTM} and \cite{LermanCTM2}.

While we found it most convenient to use Definition \ref{d:sympcone} as our definition for what it meant to be a symplectic cone, an alternative characterization, sometimes also used as a definition, is often quite useful.

\begin{prop} Let $(M,\omega)$ be a symplectic cone and let $\Xi$ be the vector field generating the required action of $\R$ on $M$. Then $\mathcal{L}_\Xi\omega = \omega$.  Conversely, if $\Xi$ is a vector field on $(M,\omega)$ generating a free and proper action of $\R$ and satisfying $\mathcal{L}_\Xi\omega=\omega$, then $(M,\omega)$ is a symplectic cone with respect to the action generated by $\Xi$. $\Xi$ is known as {\em the Liouville} or {\em expanding vector field} of $(M,\omega)$.
\end{prop}

Symplectic cones may be naturally associated with (co-oriented) contact structures on their base. For completeness, we recall the definition of these structures.

\begin{defn}
Let $B$ be a manifold. Then a {\sf contact form} is a 1-form $\alpha$ on $B$ such that $\xi=\ker(\alpha)$ is a codimension 1 distribution on $B$ with the property that $(\xi,d\alpha|_\xi)$ is a symplectic vector bundle over $B$. Say that two contact forms $\alpha$ and $\alpha'$ are in the same {\sf conformal class} of contact forms if there is a function $f\in C^\infty(B)$ such that $e^f\alpha=\alpha'$.

Call a pair $(B,\xi)$ of manifold with codimension 1 distribution $\xi$ a {\sf co-orientable contact manifold} if there exists a contact form $\alpha$ with $\ker(\alpha) = \xi$. Call a co-orientable contact manifold $(B,\xi)$ together with a choice of conformal class a {\sf co-oriented contact manifold}.

A {\sf map of co-oriented contact manifolds} $\varphi:(B,\xi)\to (B',\xi')$ is a smooth map $\varphi:B\to B'$ so that, for $\alpha$ and $\alpha'$ representatives of the conformal class for $(B,\xi)$ and $(B',\xi')$ respectively, $\varphi^*\alpha'$ and $\alpha$ are in the same conformal class. A {\sf contactomorphism} between co-oriented $(B,\xi)$ and $(B',\xi')$ is a diffeomorphism $\varphi:B\to B'$ so that $\varphi$ and $\varphi^{-1}$ are both maps of co-oriented contact manifolds.
\end{defn}

\begin{rem}
The reasoning behind the name ``co-orientable'' is the fact that, for $(B,\xi)$ a co-orientable contact manifold, the line bundle $\xi^o$ (the annihilator of $\xi$ in $T^*B$) is orientable. Thinking of the 1-form $\alpha$ as a section $\alpha:B\to T^*B$, a contact 1-form for $\xi$ functions as a nowhere 0 section trivializing $\xi^o$. It follows that $\xi^o\backslash 0$ (i.e., $\xi^o$ without its 0 section) has two components.

In the case where $(B,\xi)$ is co-oriented, we label by $\xi^o_+$ the component of $\xi^o\backslash 0$ selected by the conformal class (co-)orienting $(B,\xi)$.  This is called {\em the symplectization of $(B,\xi)$} and the restriction of the canonical symplectic form on $T^*B$ is a symplectic form on $\xi^o_+$.  

Furthermore, for any $b\in B$ and $\eta\in T^*_bB$, the action $t\cdot \eta :=e^t\eta$ is free and proper and restricts to a free and proper action on $\xi^o_+$ with quotient map the restriction of the natural projection $T^*B\to B$ to $\xi^o_+$.  With respect to this $\R$ action, $\xi^o_+$ is a symplectic cone.  We can also show that, for any choice of contact form $\alpha$ thought of as a global section $\alpha:B\to \xi^o_+$, the induced isomorphism $\xi^o_+\cong B\times \R$ induces an $\R$-equivariant symplectomorphism $\xi^o_+\cong(B\times \R,d(e^t\alpha))$.
\end{rem}

The relationship between a co-oriented contact manifold and its symplectization can be generalized to any symplectic cone.

\begin{prop}
Let $(\pi:M\to B,\omega)$ be a symplectic cone (for $\pi:M\to B$ the $\R$-quotient). Then $B$ has a natural co-oriented contact structure $\xi$ with respect to which $\xi^o_+\cong (M,\omega)$.  Furthermore, any map of symplectic cones $f:(\pi:M\to B,\omega) \to (\pi':M'\to B',\omega')$ descends to a co-orientation preserving contact map $\bar{f}:(B,\xi)\to(B',\xi')$.
\end{prop}

It is not difficult to show that, if a symplectic cone $(M,\omega)$ is in fact a symplectic toric cone, its quotient inherits the structure of a contact toric manifold.

\begin{defn}
A {\sf contact toric manifold} is a co-oriented contact manifold $(B,\xi)$ with an effective contact action by a torus $G$ such that $2\dim(G) = \dim(B) + 1$.
\end{defn}

\begin{prop}
Let $(M,\omega,\mu)$ be a symplectic toric cone. Then the co-oriented contact manifold $(B,\xi)$ with $B:=M/\R$ and $\xi$ the contact distribution descending from $(M,\omega)$ is a contact toric manifold.
\end{prop}

It is important for the work above to find a $G$-equivariant trivialization of a symplectic cone with torus action commuting with the action of $\R$. This is always possible for the symplectization of contact toric manifolds, where this task is exactly the same as finding a $G$-invariant contact form.

\begin{prop} 
Let $(B,\xi)$ be a co-oriented contact manifold with an effective action of torus $G$. Then there exists a $G$-invariant contact form $\alpha$ for $(B,\xi)$.
\end{prop}

\begin{rem}\label{r:contmm}
One application of finding a $G$-invariant contact form $\alpha$ on a contact toric manifold $(B,\xi)$ (as utilized in \cite{LermanCTM}) is the existence of a {\em contact moment map}: this is the map $\Psi_\alpha:=\mu\circ \alpha$ for $\mu:\xi^o_+\to \fg^*$ the homogeneous moment map of the symplectization of $(B,\xi)$ and where we think of $\alpha:B\to \xi^o_+\subset T^*B$ as a global section of $\xi^o_+\to B$.  Note that $\Psi_\alpha$ is called the contact map as it satisfies $\langle \Psi_\alpha(p),X\rangle = \alpha_p(X_M(p))$ for all $p\in M$, $X\in \fg^*$.
\end{rem}

$G$-invariant $\alpha$ in the above proposition is found by averaging any contact form against $G$. The proposition below shows how we can generalize this process.

\begin{prop}\label{p:equivtrivialization}
Given a symplectic cone $(M,\omega)$ with $\R$-quotient $\pi:M\to B$ and an effective action by torus $G$ that commutes with the action of $\R$, there exists a trivialization of $M$ as a principal $\R$-bundle $\varphi : M \to  B \times \R$ that is $G$-equivariant (where the contact action of $G$ on $B$ is trivially extended to an action on $B\times \R$).
\end{prop}

This fact is not standard, so here is a proof.

\begin{proof}
Fix a global section $s : B \to M$ of $\pi$ (this is always possible, as principal $\R$-bundles are always trivializable).  Recall for any principal $K$-bundle $\pi:P\to N$  for Lie group $K$, there exists a smooth ``division map'' $d : P \times_N P \to K$, the map uniquely defined by $p\cdot d(p,p')=p'$.

Then define the map $f : B \times G \to \R$ by:
\[f(b,g):=d(s(g\cdot b),g\cdot s(b))\]
This is well-defined as $\pi$ is $G$-equivariant. Essentially, $f$ measures the failure of $s$ to be $G$-equivariant and will be used to properly adjust $s$ into an equivariant section.  $f$ satisfies the following useful property: for $b\in B$, $g,h \in G$, one may show that
\[f(h\cdot b,g)=f(b,gh)-f(b,h)\]
For $d\lambda$ a $G$-invariant measure on $G$ with $\int_Gd\lambda = 1$, define $\bar{f}:B\to \R$ by:
\[\bar{f}(b):= \int_Gf(b,g)d\lambda\]
As $\bar{f}$ is the result of integrating a smooth family of functions on $G$ parameterized by $B$, it is smooth. We also have that $\bar{f}(h\cdot b)=\bar{f}(b)-\bar{f}(b,h)$.

Finally, let $\bar{s}:B\to M$ be the section of $\pi$ with $\bar{s}(b) := s(b) \cdot (-\bar{f}(b))$. It follows from how we’ve defined $\bar{s}$ that it is a $G$-equivariant section of $\pi:M\to B$.  Therefore, it induces a $G$-equivariant trivialization of $M$ as a principal $\R$-bundle.
\end{proof}

Here is another important proposition we use in the paper.

\begin{prop}\label{p:nozeroinhomogmm}
Let $(M,\omega,\mu:M\to \fg^*)$ be a symplectic toric cone with homogeneous moment map $\mu:M\to \fg^*$. Then $0$ is not in the image of $\mu$.
\end{prop}

\begin{proof}
Let $(B,\xi)$ be the co-oriented contact toric manifold with $B = M/\R$ as described above. As shown in Lemma 2.12 of \cite{LermanCTM}, for $\xi^o_+$ the symplectization of $(B,\xi)$, the image of
the homogeneous moment map $\nu:\xi^o_+\to \fg^*$  does not contain 0. Then, since every symplectic toric cone $(M,\omega)$ over $(B,\xi)$ is $(G\times\R)$-equivariantly symplectomorphic to $\xi^o_+$, it follows that the choice of homogeneous moment map $\mu$ for $(M,\omega)$ also does not contain $0$ in its image.
\end{proof}

Following Lerman's approach in \cite{LermanCTM}, we define a notion of the symplectic slice representation for contact toric manifolds.

\begin{defn} 
Let $(B,\xi)$ be a co-oriented contact toric manifold with $G$-invariant contact form $\alpha$. Let $\omega:=(d\alpha)|_{\xi}$. Then for any point $x\in B$, the $\alpha$ {\sf symplectic slice representation at $x$} is the $G_x$-vector space:
\[(V,\omega_V)_\alpha :=\left(\frac{(T_x(G\cdot x)\cap\xi_x)^\omega}{T_x(G\cdot x)\cap \xi_x},\omega|_V\right)\]
Note that another choice of $G$-invariant contact form $\alpha'=e^f\alpha$ for $(B,\xi)$ defines the same vector space $V$ with symplectic form $d(e^f \alpha)$. Thus, the symplectic vector space $(V,\omega_V)_\alpha$ depends on a choice of contact form.
\end{defn}

\begin{rem}
This matches the definition of ``symplectic slice representation'' in Definition 3.8 of \cite{LermanCTM}. We choose to label this with the contact form $\alpha$ defining this symplectic representation to avoid confusion with the standard symplectic slice representation for a symplectic toric manifold $(M,\omega)$: the vector space $W:=(T_x(G\cdot x))^\omega/T_x(G\cdot x)$ with symplectic form $\omega_W:=(\omega_x)|_W$ and restricted action of $G_x$. We will use both in the following lemma.
\end{rem}

\begin{lemma}\label{l:slicereps}
Suppose $(M,\omega,\mu:M\to \fg^*)$ is a symplectic toric cone with $G$-quotient $\pi:M\to B$ to the co-oriented contact toric manifold $(B,\xi)$. Then for each $p\in M$ there is a $G$-invariant contact form $\alpha$ for $(B,\xi)$ such that the $\alpha$ symplectic slice representation $(V,\omega_V)_\alpha$ of $\pi(p)$ in $(B,\xi)$ and the symplectic slice representation $(W,\omega_W)$ of $p$ in $(M,\omega)$ are isomorphic symplectic $G_{\pi(p)}=G_p$ representations.
\end{lemma}

\begin{proof}
Fix $p \in M$ and let $b := \pi(p)$. Since the actions of $\R$ and $G$ on $M$ commute, we have that $G_p = G_b$.

Let $\varphi:B\times \R\to M$ be a $G$-equivariant trivialization of $M$ as a principal $\R$-bundle such that $\varphi(b,0)=p$. Then for $\Xi$ the Liouville vector field associated to $(M,\omega)$, define $\alpha:=(\varphi^*(\iota_\Xi\omega))|_{B\times \{0\}}$. We have that
\[d\alpha=d(\varphi^*(\iota_\Xi\omega))|_{B\times\{0\}}=(\varphi^*(d\iota_\Xi\omega))|_{B\times \{0\}}=(\varphi^*\omega)|_{B\times \{0\}}.\]

$d\varphi_{(b,0)}$ descends to a linear map 
\[\bar{\varphi}:(V,\omega_V)_\alpha \to (W,\omega_W),\;\;\; \bar{\varphi}([v]):=[d\varphi_{(b,0)}(v)]. \]
To see $\bar{\varphi}$ is well-defined, note that, since $\varphi$ is equivariant, $d\varphi_{(b,0)}$ restricts to an isomorphism between $T_{(b,0)}(G\cdot (b,0))$ and $T_p(G\cdot p)$.

Next, note that $\bar{\varphi}$ is injective: for any $[v],[v']\in (V,\omega_V)_\alpha$, $[d\varphi_{(b,0)}(v)]=[d\varphi_{(b,0)}(v')]$ if and only if $d\varphi_{(b,0)}(v-v') = X_M(p)$ for some $X\in \fg$.  This implies that $v-v' = X_B(b)$.  By assumption, $v,v'\in \xi_b$, hence $v-v'=X_B(b) \in \xi_b$ as well.  Thus, $[v]=[v']$ in $(V,\omega_V)_\alpha$.

As noted in Remark \ref{r:homogmomentmapformula}, we have for $\mu:M\to \fg^*$ and for all $X\in \fg$ that $\langle \mu,X\rangle = \omega(\Xi,X_M)$.  Since by Proposition \ref{p:nozeroinhomogmm} $\mu(p)\neq 0$, we therefore have that there exists $Y\in \fg$ with $\omega_p(\Xi(p),Y_M(p))\neq 0$.  It follows that $Y_B(b)\notin \xi_b$ and hence $\dim(\xi_b\cap T_b(G\cdot b)) = \dim(G)-\dim(G_b)-1$.  Therefore 
\begin{align*}\dim(V) &= \dim((\xi_b\cap T_b(G\cdot b))^{\omega_V})-\dim(\xi_b\cap T_b(G\cdot b)) = \dim(\xi_b)-2(\xi_b\cap T_b(G\cdot b))\\ &= \dim(M)-2-2(\dim(G)-\dim(G_b)-1) = \dim(M)-2(\dim(G)-\dim(G_b)) \\ &= \dim(W).\end{align*}
Hence, $\bar{\varphi}$ is an isomorphism.  Since $d\varphi_{(b,0)}$ is equivariant and symplectic, it follows that $\bar{\varphi}$ is an equivariant symplectic isomorphism.

\end{proof}

We now quote the following lemma of Lerman.	

\begin{lemma}[Lemma 3.9, \cite{LermanCTM}]\label{l:lermantechcondition}
Let $(B,\xi)$ and $(B',\xi')$ be two co-oriented contact toric manifolds with $G$-invariant contact forms $\ker\alpha=\xi$ and $\ker\alpha'=\xi'$. Suppose $x\in B$ and $x'\in B'$ satisfy
\begin{itemize}
	\item $\Psi_\alpha(x)=\lambda \Psi_{\alpha'}(x')$ for $\Psi_\alpha$, $\Psi'_\alpha$ the contact moment maps for $\alpha$, $\alpha'$ and $\lambda>0$ (see Remark \ref{r:contmm});
	\item $G_x = G_{x'}$ (i.e., the isotropy groups for each point are equal); and
	\item For $(V,\omega)_\alpha$ and $(V',\omega_{V'})_{\alpha'}$ the $\alpha$ and $\alpha'$ symplectic slice representations for $x$ and $x'$, there is an $G_x$-equivariant linear isomorphism $l : V \to V'$ such that $l^*\omega_{V'}=(d(e^g\alpha)_x)|_V$ for some function $g\in C^\infty(B)$.
\end{itemize}
Then there are $G$-invariant open neighborhoods $U$ of $x$ and $U'$ of $x'$ and a $G$-equivariant diffeomorphism $\varphi:U\to U'$ satisfying $\varphi(x)=x'$ and $\varphi^*\alpha'=f\alpha$ for some $f\in C^\infty(U)$.
\end{lemma}

This allows us to prove the following extension of a standard symplectic toric result to symplectic toric cones. 

\begin{prop}\label{p:finalctmtechlemma}
Let $(M,\omega,\mu:M\to \fg^*)$ and $(M',\omega',\mu':M'\to \fg^*)$ be two symplectic toric cones. Suppose there are two points $p\in M$ and $p'\in M'$ so that:
\begin{itemize}
	\item $G_p = G_{p'}$;
	\item The symplectic slice representations $(V,\omega_V)$ and $(V',\omega_{V'})$ at $p$ and $p'$ are isomorphic as symplectic $G_p = G_{p'}$ vector spaces; and
	\item $\mu(p)=\mu'(p')$.
\end{itemize}
Then there exist $(G \times \R)$-invariant neighborhoods $U$ and $U'$ of $p$ and $p'$ respectively and a $(G \times \R)$-equivariant symplectomorphism $f:U\to U'$ with $f(p) = p'$ and $\mu'\circ f = \mu|_U$.
\end{prop}

\begin{proof}
Let $(B,\xi)$ and $(B',\xi')$ be the contact toric bases of $(M,\omega)$ and $(M',\omega')$. Denote the $\R$-quotient maps by $\pi:M\to B$ and $\pi':M'\to B'$ and define $b:=\pi(p)$ and $b' := \pi'(p')$. Then by Lemma \ref{l:slicereps}, there exist $G$-equivariant trivializations $\varphi:B\times \R\to M$ and $\varphi':B'\times \R\to M'$ so that $\varphi(b,0)=p$, $\varphi'(b',0)=p'$, and, for $\varphi^*\omega=d(e^t\alpha)$ and ${\varphi'}^*\omega'=d(e^t\alpha')$, the symplectic slice representations $(V,\omega_V)$ and $(V',\omega_{V'})$ are isomorphic to the $\alpha$ and $\alpha'$ symplectic slice representations $(W,\omega_w)_\alpha$ and $(W',\omega_{W'})_{\alpha'}$ of $b$ and $b'$, respectively.

By Lemma \ref{l:lermantechcondition} there are $G$-invariant neighborhoods $U$ and $U'$ of $b$ and $b'$ and a $G$-equivariant co-orientation preserving contactomorphism $\phi:U\to U'$ with $\phi(b)=b'$ and $\phi^*\alpha'=e^g\alpha$, for some $g\in C^\infty(B)$. The map $\tilde{\phi}:U\times \R\to U'\times \R$, defined by $\tilde{\phi}(b,t):=(\phi(b),t-g(b))$ is $(G\times \R)$-equivariant and satisfies $\tilde{\phi}^*(d(e^t\alpha'))=d(e^t\alpha)$. Hence, $f:=\varphi'\circ \tilde{\phi}\circ \varphi^{-1}$ yields a map of symplectic toric cones between $\pi^{-1}(U)$ and ${\pi'}^{-1}(U')$. Since $\mu|_U$ and $\mu'\circ f$ are both homogeneous moment maps for $\pi^{-1}(U)$, it follows that $\mu|_U = \mu'\circ f$.

Finally, note that
\[f(p)=\varphi'(\tilde{\phi}(\varphi^{-1}(p)))=\varphi'(\tilde{\phi}(b,0))=\varphi'(\phi(b),-g(b))=\varphi'(b',-g(b))\]
Thus, $f(p) = -g(b)\cdot p'$. But then
\[\mu(p)=\mu'(f(p))=\mu'(-g(b)\cdot p')=e^{-g(b)}\mu'(p'),\]
so, since $\mu(p)=\mu'(p')$, we must conclude that $g(b) = 0$ (as the image of $\mu'$ doesn't contain $0$; see Proposition \ref{p:nozeroinhomogmm}) and therefore $f(p) = p'$.
\end{proof}

Finally, we may prove Lemma \ref{l:stclocconnected}; for convenience, we recall its statement:\\

\noindent {\bf Lemma \ref{l:stclocconnected}.} Let $\psi:W\to \fg^*$ be a homogeneous unimodular local embedding. Then any two symplectic toric cones over $\psi$ are locally isomorphic; explicitly, for $(M,\omega,\pi:M\to W)$ and $(M',\omega',\pi':M'\to W)$ two symplectic toric cones over $\psi$, there is an open cover $\{U_\alpha\}_{\alpha\in A}$ of $W$ by $\R$-invariant open subsets and a collection of isomorphisms
\[\{\varphi_\alpha : (M,\omega,\pi)|_{U_\alpha} \to (M', \omega',\pi')|_{U_\alpha} \in \stc_\psi(U_\alpha)\}_{\alpha\in A}.\]

\begin{proof}
We may proceed using the method of Lemma B.4 of \cite{KarshonLerman}. Fix $p\in M$ and $p'\in M'$ with $\pi(p)=\pi'(p')$. Since $\psi\circ\pi$ and $\psi\circ \pi'$ are moment maps for $(M,\omega)$ and $(M',\omega')$, respectively, we may conclude from the local normal form for symplectic toric manifolds that both $p$ and $p'$ have stabilizer $K_{\pi(p)}$ and symplectic slice representations isomorphic to $\C^k$ with symplectic weights $\{v^*_1,\ldots,v^*_k\}$, for $C_{\{v_1,\ldots,v_k\},\psi(\pi(p))}$ the unimodular cone with vertex $\psi(\pi(p))$ uniquely determined by $\psi$ near $\pi(p)$.

Therefore, by applying Proposition \ref{p:finalctmtechlemma}, we have our result. 
\end{proof}

Again following Lerman in \cite{LermanCTM}, orbits in contact toric manifolds have a local normal form (as we are interested only in the structure of the neighborhood as a $G$-manifold, we suppress the additional information from the lemma regarding contact structure and moment map).  This local normal form in turn yields a manifolds with corner structure for the quotient of any contact toric manifold.

\begin{lemma}[Lemma 3.10, \cite{LermanCTM}] 
Let $(L,\xi)$ be a contact toric manifold with $G$-invariant form $\alpha$ and contact moment map $\Psi_\alpha$ (see Remark \ref{r:contmm}). Given point $p\in L$, denote the $\alpha$ symplectic slice representation by $G_p \to \mathrm{Sp}(V,\omega_V)$ and let $\fk := (\R\Psi_\alpha(p))^o$ (the so-called characteristic subalgebra of the embedding $G\cdot p\to (M,\xi)$). Then there exists a $G$-invariant neighborhood of the orbit of $p$ in $L$ that is $G$-equivariantly diffeomorphic to a neighborhood of the $0$ section of the vector bundle $N = G \times_{G_p}((\fg/\fk)^* \oplus V)$.
\end{lemma}

\begin{lemma}\label{l:contquotamwc} 
Let $(B,\xi)$ be a contact toric manifold. Then the quotient $B/G$ is a manifold with corners.
\end{lemma}

\begin{proof}
For each point $p \in B$, the stabilizer $G_p$ is a torus (see Lemma 3.13 of \cite{LermanCTM}). Let $\alpha$ be a $G$-invariant contact form for $(B,\xi)$. From the above lemma, there exists subspace $U\subset \fg^*$ so that, for $\alpha$ symplectic slice representation $G_p\to \mathrm{Sp}(V,\omega_V)$, there is a $G$-invariant neighborhood of $G\cdot p$ in $B$ that is $G$-equivariantly diffeomorphic to a neighborhood of the 0 section of the vector bundle: $N = G \times_{G_p} (U \times  V )$ (where $U$ has trivial $G_p$ action).

So to understand what B/G locally looks like, it is enough to understand $N/G$:
\[N/G = ((G \times U \times V )/G_p)/G = ((G \times U \times V )/G)/G_p = U \times V/G_p\]
Here, we may reverse the quotients as the actions of $G$ and $G_p$ commute. Since $G_p$ is a torus, we can decompose $V$ into weight spaces (as in the appendix of \cite{LermanTolman}) to see $V/G_p$ is diffeomorphic to a sector (i.e., a manifold with corners of the form $[0,\infty)\times \R^l$). Therefore, the above local normal form for $(B,\xi)$ near $p$ descends to a manifolds with corners chart for $[p] \in B/G$ centered at the origin in $U \times V/G_p$.
\end{proof}

%%%%%%%%%%%%%%%%%%%%%%%%%%%%%%%%%%%%%%%%%
\section{Stacks}\label{app:stacks}
%%%%%%%%%%%%%%%%%%%%%%%%%%%%%%%%%%%%%%%%%

In this appendix, we will provide some brief notes on stacks. This will also contain proofs that our presheaves of groupoids $\hstb_\psi$ and $\cstb_\psi$ are stacks as well as the major technical lemma we require to prove that $\hc:\hstb_\psi\to \stc_\psi$ and $\ssc:\cstb_\psi\to \stss_\psi$ are isomorphisms of presheaves of groupoids.

For simplicity's sake, we will be using a less general definition that would perhaps be more accurately named a strict sheaf of groupoids. Since the stacks of this paper are, in fact, (strict!) presheaves of groupoids, we needn't worry about more intricate/subtle definitions (i.e., involving lax presheaves or categories fibered in groupoids).  A few good sources for the complete story on stacks are \cite{Vistoli} (which is focused more on stacks in algebraic geometry), \cite{BehrendXu} (which is focused on using stacks in differential geometry), and \cite{Heinloth} (which discusses stacks over manifolds and over topological spaces).

Fix a topological space $X$. For $\{U_\alpha\}_{\alpha\in A}$ an open cover of $X$, we write $U_{\alpha\beta} := U_\alpha\cap U_\beta$ and $U_{\alpha\beta\gamma} := U_\alpha\cap U_\beta \cap U_\gamma$. First, we need some preliminaries.

\begin{defn}
We denote by $\open(X)$ {\sf the category of open sets of $X$}: the objects of $\open(X)$ are open subsets $U \subset X$ and the morphisms are inclusions $\iota:U\to V$.
\end{defn}

One approach for defining stacks involves so-called {\em descent data}: roughly, the local data we would like to uniquely (up to isomorphism) glue into global data.

\begin{defn}
Let $\{U_\alpha\}_{\alpha\in A}$ be an open cover of $X$ and let $\mathcal{F} : \op{\open(X)} \to \gpoid$ be a presheaf of groupoids. Then a {\sf descent datum} of $\mathcal{F}$ defined with respect to $\{U_\alpha\}_{\alpha\in A}$ is a pair of tuples
\[\left(\{\xi_\alpha\in \mathcal{F}(U_\alpha)\}_{\alpha\in A},\{\varphi_{\alpha\beta}:\xi_\alpha|_{U_{\alpha\beta}} \to \xi_\beta|_{U_{\alpha\beta}}\in \mathcal{F}(U_{\alpha\beta})\}_{\alpha,\beta\in A}\right)\]
such that the morphisms $\{\varphi_{\alpha\beta}\}_{\alpha,\beta\in A}$ (called {\sf transition morphisms}) satisfy {\sf the cocycle condition}: for every non-empty triple intersection $U_{\alpha\beta\gamma}$, we have that $\varphi_{\beta\gamma}|_{U_{\alpha\beta\gamma}} \circ \varphi_{\alpha\beta}|_{U_{\alpha\beta\gamma}} = \varphi_{\alpha\gamma}|_{U_{\alpha\beta\gamma}}$.

A {\sf morphism of descent data}
\[\{\eta_\alpha\}_{\alpha\in A}:\left(\{\xi_\alpha\in \mathcal{F}(U_\alpha)\}_{\alpha\in A},\{\varphi_{\alpha\beta}\}_{\alpha,\beta\in A}\right) \to \left(\{\xi'_\alpha\in \mathcal{F}(U_\alpha)\}_{\alpha\in A},\{\varphi'_{\alpha\beta}\}_{\alpha,\beta\in A}\right)\]
is a collection of morphisms $\{\eta_\alpha:\xi_\alpha\to \xi'_\alpha\in \mathcal{F}(U_\alpha)\}_{\alpha\in A}$ so that the diagram
\begin{equation}\label{eq:generalddmorphism}
\xymatrix{\xi_\alpha|_{U_{\alpha\beta}} \ar[r]^{\eta_\alpha|_{U_{\alpha\beta}}} \ar[d]_{\varphi_{\alpha\beta}} & \xi'_\alpha|_{U_{\alpha\beta}} \ar[d]^{\varphi'_{\alpha\beta}} \\ \xi_\beta|_{U_{\alpha\beta}} \ar[r]_{\eta_\beta|_{U_{\alpha\beta}}} & \xi'_\beta|_{U_{\alpha\beta}} }
\end{equation}
commutes for every $\alpha$ and $\beta$ with $U_{\alpha\beta}$ non-empty.

We denote by $\mathcal{D}_{\mathcal{F}}(\{U_\alpha\})$ {\sf the descent category}: the category of descent data of $\mathcal{F}$ with respect to $\{U_\alpha\}_{\alpha\in A}$ with morphisms of descent data.
\end{defn}

Here is the formal definition of a stack we will be using.

\begin{defn}\label{d:stack}
 Let $\mathcal{F}:\op{\open(X)} \to \gpoid$ be a presheaf of groupoids (i.e., a {\em strict} functor). For an open cover $\{U_{\alpha}\}_{\alpha\in A}$ of $X$, {\sf the restriction functor} $\Phi:\mathcal{F}(X) \to \mathcal{D}_{\mathcal{F}}(\{U_\alpha\})$ is that which takes an object $\xi\in \mathcal{F}(X)$ to the descent datum
\[\left(\left\{\xi|_{U_\alpha}\right\},\left\{\mathrm{id}:\left(\xi|_{U_\alpha}\right)|_{U_{\alpha\beta}} \to \left(\xi|_{U_\beta}\right)|_{U_{\alpha\beta}}\right\}_{\alpha,\beta\in A}\right)\]
and a morphism $\varphi:\xi\to \xi'$ to the morphism of descent data $\{\varphi|_{U_\alpha}:\xi|_{U_\alpha}\to \xi'|_{U_\alpha}\}_{\alpha\in A}$. 

$\mathcal{F}$ is a {\sf stack} if, for every open subset $U$ of $X$ and for every open cover $\{U_\alpha\}_{\alpha\in A}$ of $U$, the restriction morphism $\Phi:\mathcal{F}_U(X) \to \mathcal{D}_{\mathcal{F}_U}(\{U_\alpha\})$ is an equivalence of categories, where $\mathcal{F}_U$ is the natural restriction of the presheaf of groupoids $\mathcal{F}$ to the full subcategory $\open(U)\subset \open(X)$.
\end{defn}

\begin{rem}
For $W$ a manifold with a proper $\R$ action, we may just as easily define a stack over the category $\open_{\R}(W)$, as defined in Definition \ref{d:opensubR}. Indeed, a presheaf of groupoids over this category is again just a functor $\mathcal{F} : \op{\open_\R(W)}\to \gpoid$ and we may replace the open covers of $\open(W)$ as in Definition \ref{d:stack} with open covers of elements of $\open_{\R}(W)$ by $\R$-invariant subsets.
\end{rem}

A half step from presheaves to stacks is the {\em prestack}.  For $\mathcal{F} : \op{\open(X)} \to \gpoid$ a presheaf and $U$ any open subset of $X$, note that, for $V \subset U$ an open subset of $U$, the restriction morphism from $U$ to $V$ is a map of  groupoids. Thus, for any two objects $\xi$ and $\xi'$ in $\mathcal{F}(U)$, we have a map of sets
\[\mathrm{Hom}_{\mathcal{F}(U)}(\xi,\xi')\to \mathrm{Hom}_{\mathcal{F}(V)}(\xi|_V,\xi'|_V)\]
This collection of restrictions defines a presheaf of sets $\mathrm{Hom}(\xi,\xi'):\op{\open(U)} \to \set$.

\begin{defn}\label{d:prestack}
A presheaf of groupoids $\mathcal{F}:\op{\open(X)}\to \gpoid$ is a {\sf prestack} if for every open subset $U \subset X$ and for any$\xi,\xi'\in \mathcal{F}(U)$, $\mathrm{Hom}(\xi,\xi')$ is a sheaf of sets.
\end{defn}

\begin{rem}\label{r:prestacks}
 It is a routine exercise to show that a presheaf of groupoids $\mathcal{F}:\op{\open(X)}\to \gpoid$ is a prestack if and only if for every open subset $U\subset X$ and for any open cover $\{U_\alpha\}_{\alpha\in A}$ of $U$ the restriction functor $\Phi:\mathcal{F}(U)\to \mathcal{D}_{\mathcal{F}_U}(\{U_\alpha\})$ is fully faithful.

Since the presheaves of this paper all consist of spaces with extra structure (bundles/manifolds/stratified spaces with symplectic forms/group actions), all the presheaves of groupoids we consider in this paper are prestacks.
\end{rem}

One of the standard first examples of a stack is the presheaf of principal $G$-bundles for $G$ any Lie group.

\begin{ex}\label{ex:bg}
Let $\BG:\op{\open(X)} \to \gpoid$ be the presheaf of principal $G$-bundles over $X$: for each open $U$, $\BG(U)$ is defined as the groupoid of principal $G$-bundles over $U$ with isomorphisms of principal $G$-bundles covering the identity on $U$.

The proof that BG is a stack comes in two parts. Let $U$ be an open subset of $X$ with open cover $\{U_\alpha\}_{\alpha\in A}$ and let $\Phi:\BG(U)\to \mathcal{D}_{\BG}\left(\{U_\alpha\}_{\alpha\in A}\right)$ be the restriction functor. Let $\pi:P\to U$ and $\pi':P'\to U$ be two principal $G$-bundles. Then $\underline{\sf Hom}(P,P')$ is a sheaf of sets.  Thus, $\Phi$ is fully faithful (see Remark \ref{r:prestacks}).

To show $\Phi$ is essentially surjective, let
\[\left(\{\pi_\alpha:P_{\alpha}\to U_\alpha\}_{\alpha\in A},\{\varphi_{\alpha\beta}:\left(P_\alpha\right)|_{U_{\alpha\beta}}\to \left(P_\beta\right)|_{U_{\alpha\beta}}\}_{\alpha,\beta\in A}\right)\]
be a piece of decent data. Then we may use the standard construction
\[P:=\left.\left(\coprod_{\alpha\in A} P_\alpha\right)\right/\sim,\;\;\;p\sim q \text{ if }p\in P_\alpha|_{U_{\alpha\beta}}\text{, }q\in P_{\beta}|_{U_{\alpha\beta}}\text{, and }\varphi_{\alpha\beta}(p)=q\]
to build a principal bundle $P$ that restricts to the given descent datum.
\end{ex}

\begin{prop}\label{p:hstbastack}
Let $\psi:W\to \fg^*$  be a homogeneous unimodular local embedding. Then $\hstb_\psi:\op{\open_\R(W)}\to \gpoid$ is a stack.
\end{prop}

\begin{proof}
Recall $\hstb_\psi$ is a presheaf of groupoids over the $\R$-invariant subsets of $W$ with an $\R$-invariant open subset $U\subset W$ corresponding to the groupoid $\hstb_{\psi|_U}(U)$.  Fix an open cover of $U$ by $\R$-invariant subsets $\{U_\alpha\}_{\alpha\in A}$ and again denote by $\Phi:\hstb_\psi(U)\to \mathcal{D}_{(\hstb_\psi)|_{U}}(\{U_\alpha\})$ the restriction functor.  

That $\Phi$ is fully faithful follows from the fact that maps between any two homogeneous symplectic toric bundles $(\pi:P\to U,\omega)$ and $(\pi':P'\to U,\omega')$ in $\hstb_\psi(U)$ are $(G \times \R)$-equivariant symplectomorphisms; hence, coherent maps collections of maps on any open $\R$-invariant cover of $U$ will uniquely glue.

To see that $\Phi$ is essentially surjective, let
\begin{equation}\label{eq:hstbdescdata}\left(\{(\pi_\alpha:P_\alpha\to U_\alpha,\omega_\alpha\}_{\alpha\in A},\{\varphi_{\alpha\beta}:(P_\alpha,\omega_\alpha)|_{U_{\alpha\beta}}\to (P_\beta,\omega_\beta)|_{U_{\alpha\beta}}\}_{\alpha,\beta\in A}\right)\end{equation}
be a descent datum. Let $\pi:P\to U$ be the bundle built from this data as in Example \ref{ex:bg}.  Since the transition maps $\varphi_{\alpha\beta}$ are $\R$-equivariant, it follows that the actions of $\R$ on each $P_\alpha$ patch together to give a free action on $P$.

To see that this action of $\R$ on $P$ is proper, note that, since the action of $\R$ on $P$ commutes with the action of $G$ on $P$, this $\R$ action descends to an $\R$ action on $W$. By Lemma \ref{l:uniqueraction}, this $\R$ action matches the free and proper action on $W$ with respect to which $\psi$ is equivariant.  Therefore, by Lemma \ref{l:descendingproperactions}, the action of $\R$ on $P$ must have been proper.

As the transition maps $\varphi_{\alpha\beta}$ must also be symplectomorphisms, the symplectic forms from each piece must patch together. Finally, since the condition $\rho_{\lambda}^*\omega = e^\lambda \omega$ for $\rho_\lambda:P\to P$ the action diffeomorphism for real $\lambda$ is local, it follows that, since each $\omega_\alpha$ satisfies this property, the global symplectic form $\omega$ they patch to must satisfy this property as well. So any descent datum patches together to an element $(\pi:P\to U, \omega)$ with $\Phi(\pi:P\to U,\omega)$ isomorphic to descent datum \eqref{eq:hstbdescdata}.

Thus, $\hstb_\psi$ is a stack.
\end{proof}

\begin{prop}\label{p:cstbastack}
Let $\psi:W\to \fg^*$ be a stratified unimodular local embedding. Then $\cstb_\psi:\op{\open(W)}\to \gpoid$ is a stack.
\end{prop}

\begin{proof}
Fix an open subset $U$ in $W$ with open cover $\{U_\alpha\}_{\alpha\in A}$. We must show $\Phi:\cstb_\psi(U) \to \mathcal{D}_{(\cstb_\psi)|_U}(\{U_\alpha\})$ is an equivalence of categories.

As we saw in the case of homogeneous symplectic toric bundles, we can see that $\Phi$ is fully faithful simply by noticing that the maps of conical symplectic toric bundles are bundle maps that preserve local data.

To see that $\Phi$ is essentially surjective, one only needs to check that, since each piece of a descent datum must satisfy the local conditions required of elements of $\cstb_\psi$, then the bundle resulting from applying the gluing construction for bundles as seen in Example \ref{ex:bg} together with the glued symplectic form as seen in Proposition \ref{p:hstbastack} must also satisfy this local condition. The details are left to the reader. 
\end{proof}

We will also be interested in a special class of presheaves of groupoids known as transitive presheaves.

\begin{defn}
A presheaf of groupoids $\mathcal{F} : \op{\open(X)}\to \gpoid$ is called {\sf transitive} if for every open subset $U\subset X$, any two objects $\xi$ and $\xi'$ in $\mathcal{F}(U)$ are locally isomorphic; that is, there exists a cover $\{U_\alpha\}_{\alpha\in A}$ such that the restrictions $\xi|_{U_\alpha}$ and $\xi'|_{U_\alpha}$ are isomorphic in $\mathcal{F}(U_\alpha)$ for each $\alpha$.
\end{defn}

The payoff for working with stacks in our case will be the following technical lemma. This is a generalized version of the proof presented in \cite{KarshonLerman} that, for $\psi:W\to \fg^*$ a unimodular local embedding, the functor $c_U:\stb_\psi(U)\to \stm_\psi(U)$ is essentially surjective on each open subset $U\subset W$.

\begin{lemma}\label{l:isoofstackstechlemma}
Let $X$ be a topological space. Suppose $\mathcal{F} : \op{\open(X)}\to \gpoid$ is a stack, $\mathcal{G}:\op{\open(X)} \to \gpoid$ is a prestack, and $\Psi:\mathcal{F}\to \mathcal{G}$ is a map of presheaves of groupoids. If for each open set $U\subset X$
\begin{enumerate}
	\item $\Psi_U:\mathcal{F}(U)\to \mathcal{G}(U)$ is fully faithful; and
	\item for each $x\in U$ and each $\xi \in \mathcal{G}(U)$, there is an open subset $V\subset U$ and an element $\eta\in \mathcal{F}(V)$ such that $\Psi(\eta)$ is isomorphic to $\xi|_V$,
\end{enumerate}
then $\Psi$ is an isomorphism of presheaves. In particular, $\mathcal{G}$ is a stack.
\end{lemma}

\begin{rem}\label{r:isoofstackstechlemmarem}
Note that this lemma may also be applied to the map of presheaves $\hc:\hstb_\psi\to \stc_\psi$ over $\open_{\R}(W)$, assuming we use open $\R$-invariant subsets and covers by open $\R$-invariant subsets of $W$.

Additionally, note that if $\mathcal{G}$ is a transitive prestack and $\mathcal{F}(U)$ is non-empty for each $U \neq \emptyset$, then any map of presheaves automatically satisfies condition (2) of the above lemma. However, we require the slightly more nuanced condition (2) of Lemma \ref{l:isoofstackstechlemma} to apply to the case of $\ssc:\cstb_\psi \to \stss_\psi$, where $\stss_\psi$ in general need not be a transitive prestack.
\end{rem}

\begin{proof}[Proof of Lemma \ref{l:isoofstackstechlemma}]
Fix an open subset $U$ of $X$. To show $\Psi$ is an isomorphism of presheaves, it is enough to show that $\Psi_U:\mathcal{F}(U)\to \mathcal{G}(U)$ is an equivalence of groupoids for each $U$. By hypothesis, we have already that $\Psi_U:\mathcal{F}(U)\to \mathcal{G}(U)$ is fully faithful, so it remains to show that it is essentially surjective.

Fix an element $\xi\in \mathcal{G}(U)$.  Then by hypothesis there is an open cover $\{U_\alpha\}_{\alpha\in A}$ of $U$, elements $\{\eta_\alpha\in \mathcal{F}(U_\alpha)\}_{\alpha\in A}$, and a family of isomorphisms $\{\varphi_\alpha :\Psi(\eta_\alpha)\to \xi|_{U_\alpha}\}_{\alpha\in A}$.  Since $\Psi_{U_{\alpha\beta}}$ is full for every $U_{\alpha\beta}$ and $\Psi(\eta_\alpha)|_{U_{\alpha\beta}}=\Psi(\eta_\alpha|_{U_{\alpha\beta}})$, there exist morphisms $\phi_{\alpha\beta}:\eta_\alpha|_{U_\alpha\beta}\to \eta_\beta|_{U_{\alpha\beta}}$ such that $\Psi(\phi_{\alpha\beta})=\varphi_\beta^{-1}\circ \varphi_\alpha$.

As $\Psi_{U_{\alpha\beta\gamma}}$ is faithful, it follows that $\phi_{\beta\gamma}|_{U_{\alpha\beta\gamma}}\circ\phi_{\alpha\beta}|_{U_{\alpha\beta\gamma}}=\phi_{\alpha\gamma}|_{U_{\alpha\beta\gamma}}$ for any $\alpha$, $\beta$, $\gamma$ with $U_{\alpha\beta\gamma}\neq \emptyset$.  Thus, the family of isomorphisms $\{\phi_{\alpha\beta}\}_{\alpha,\beta\in A}$ satisfies the cocycle condition and $\{\{\eta_\alpha\}_{\alpha\in A}, \{\phi_{\alpha\beta}\}_{\alpha,\beta\in A}\}$ is a descent datum for $\mathcal{F}_U$ with respect to the cover $\{U_\alpha\}_{\alpha\in A}$.  As $\mathcal{F}$ is a stack, there is an $\eta \in \mathcal{F}(U)$ and an isomorphism of descent data $\{\rho_\alpha\}_{\alpha\in A}:\Phi(\eta)\to \{\{\eta_\alpha\}_{\alpha\in A}, \{\phi_{\alpha\beta}\}_{\alpha,\beta\in A}\}.$

For each $\alpha$, let $f_\alpha :\Psi(\eta)|_{U_\alpha}\to \eta|_{U_\alpha}$ be the composition $f_\alpha := \varphi_\alpha \circ \Psi(\rho_\alpha)|_{U_\alpha}$.  Notice that the diagram in $\mathcal{G}(U_{\alpha\beta})$
\[\xymatrix@C=15ex@R=8ex{\Psi(\eta)|_{U_{\alpha\beta}} \ar[r]^{\Psi(\rho_\alpha)|_{U_{\alpha\beta}}} \ar@{=}[d] & \Psi(\eta_\alpha)|_{U_{\alpha\beta}} \ar[d]_{\Psi(\phi_{\alpha\beta})} \ar[r]^{\varphi_\alpha|_{U_{\alpha\beta}}} & \xi|_{U_{\alpha\beta}} \ar@{=}[d] \\ \Psi(\eta)|_{U_{\alpha\beta}} \ar[r]_{\Psi(\rho_\beta)|_{U_{\alpha\beta}}} & \Psi(\eta_\beta)|_{U_{\alpha\beta}} \ar[r]_{\varphi_\beta|_{U_{\alpha\beta}}} & \xi|_{U_{\alpha\beta}} }\]
commutes; the left square is exactly the image under $\Psi$ of the diagram \eqref{eq:generalddmorphism} corresponding to the isomorphism of descent data $\{\rho_\alpha\}_{\alpha\in A}$ while the right hand side commutes by definition of $\phi_{\alpha\beta}$.  Therefore, since $\mathcal{G}$ is a prestack, $\mathrm{Hom}(\Psi(\eta),\xi)$ is a sheaf and thus the $\{f_\alpha\}_{\alpha\in A}$ glue to an isomorphism $f:\Psi(\eta)\to \xi$.  Hence, $\Psi_U:\mathcal{F}\to \mathcal{G}$ is essentially surjective.
\end{proof}

\end{document}